\documentclass{article}

\usepackage{PRIMEarxiv}

\usepackage{verbatim}
\usepackage{subfigure}

\usepackage{ulem}

\usepackage{bbm}

\usepackage{amsmath}
\usepackage{mathrsfs}
\usepackage{amssymb}

\usepackage{multirow} 
\usepackage{extarrows}

\usepackage{diagbox}
\usepackage{booktabs}
\usepackage{multirow}

\usepackage{lipsum}

\usepackage{floatrow}

\usepackage{environ}
\usepackage{tikz}

\usepackage{fancybox}
\NewEnviron{elaboration}{
\par
\begin{tikzpicture}
\node[rectangle,minimum width=0.5\textwidth] (m) {\begin{minipage}{0.79\textwidth}\BODY\end{minipage}};
\draw[dashed] (m.south west) rectangle (m.north east);
\end{tikzpicture}
}

\newfloatcommand{capbtabbox}{table}[][\FBwidth]
\DeclareSymbolFont{largesymbol}{OMX}{yhex}{m}{n}
\DeclareMathAccent{\Widehat}{\mathord}{largesymbol}{"62}

\newcommand{\zh}[1]{{\color{blue}{\bf\sf [ZH: #1]}}}

\newtheorem{theorem}{Theorem}
\newtheorem{definition}{Definition}

\newtheorem{remark}{Remark}
\newtheorem{lemma}{lemma}
\newtheorem{assumption}{Assumption}
\newtheorem{corollary}{Corollary}
\newtheorem{proposition}{Proposition}
\newtheorem{proof}{proof}

\usepackage{amsmath}  

\usepackage{booktabs} 

\usepackage{bm}

\usepackage{tikz}
\newcommand*{\circled}[1]{\lower.7ex\hbox{\tikz\draw (0pt, 0pt)%
    circle (.5em) node {\makebox[1em][c]{\small #1}};}}
    
\usepackage{graphicx}
\usepackage{caption, threeparttable}
\usepackage{times}
\usepackage{bm}

\usepackage{appendix}

\usepackage[ruled]{algorithm2e}

\usepackage{makecell}

\usepackage{bbding}

\usepackage[thinlines,thiklines]{easybmat}

\usepackage{graphicx}



\def\cS{\mathcal{S}}
\newcommand{\argmin}{\operatornamewithlimits{argmin}}
\newcommand{\argmax}{\operatornamewithlimits{argmax}}

\usepackage[utf8]{inputenc} 
\usepackage[T1]{fontenc}    
\usepackage{url}            
\usepackage{booktabs}       
\usepackage{amsfonts}       
\usepackage{nicefrac}       
\usepackage{microtype}      
\usepackage{lipsum}
\usepackage{fancyhdr}       
\usepackage{graphicx}       

\title{Partial Identification with Proxy of Latent Confoundings via Sum-of-ratios Fractional Programming
}

\author{
  Zhiheng Zhang \\
  Institute for Interdisciplinary Information Sciences, Tsinghua University\\
  \texttt{zhiheng-20@mails.tsinghua.edu.cn} \\
}

\begin{document}

\maketitle

\begin{abstract} 
Due to the unobservability of confoundings, there has been a widespread concern on how to compute causality quantitatively. To address this challenge, proxy based negative control approaches have been commonly adopted, where auxiliary outcome variables $\bm{W}$ are introduced as the proxy of confoundings $\bm{U}$. However, these approaches rely on strong assumptions such as reversibility, completeness or bridge functions. These assumptions lack intuitive empirical interpretation and solid verification technique, hence their applications in the real world is limited. For instance, these approaches are inapplicable when the transition matrix $P(\bm{W} \mid \bm{U})$ is irreversible. In this paper, we focus on a weaker assumption called the partial observability of $P(\bm{W} \mid \bm{U})$. We develop a more general single-proxy negative control method called Partial Identification via Sum-of-ratios Fractional Programming (PI-SFP). It is a global optimization algorithm based on the branch-and-bound strategy, aiming to provide the valid bound of the causal effect. In simulation, PI-SFP provides promising numerical results, and fill in the blank spots that can not be handled in the previous literature, such as we have partial information of $P(\bm{W} \mid \bm{U})$.
\end{abstract}

\keywords{Causality; Partial identification; Fractional programming; Branch-and-bound; Average causal effect}

\tableofcontents

\section{Introduction}

Identifying causal effects from observational data is a fundamental question in economics, social science and epidemiology~\cite{pearl2009causality}. Causal inference is usually challenging with the existence of latent confoundings, which impedes us from extracting useful causal information from directly applying statistical association studies \cite{pearl2009causality}. In order to adjust for the bias incurred from latent confoundings, people usually need to reply on auxiliary variables for confounding adjustment. These auxiliary variables mainly include instrument variable (IV) method~\cite{soderstrom2002instrumental}, and proximal variables \cite{kuroki2014measurement, tchetgen2020introduction}, or both~\cite{miao2018confounding, shi2020multiply, singh2020kernel, kallus2021causal}. In this paper, we are primarily interested in causal identification with proxies of latent confounders. Fig.~\ref{Fig.sub.1} and Fig.~\ref{Fig.sub.2} are examples of such methods, where to identify the causal effect of $X$ towards $Y$, we assume there are additional random variables such as $W$ or $Z$ that are aossicated with the latent confounders, and use these variables as ``proxies'' of latent founders for confounding adjustment.

Empirical studies on using proxies for confounding adjustment has a long history, the earliest could be traced back to ~\cite{wickens1972note}, in which the authors analyzed the potential benefit of using proxies as an alternative of latent confounding in least square estimations. This was further applied in observational studies such as~\cite{kolenikov2009socioeconomic, wooldridge2009estimating}. Other empirical studies include~\cite{frost1979proxy,rothman2008modern}. On the theoretical side, existing research could be mainly splitted into two categories, the first is the ``single-proxy scenario'', where we assume there is only a single proxy variable (Figure~\ref{Fig.sub.1}); and the second is the ``double-proxy scenario'', where we have access to data from at least two proxy variables (Figures~\ref{Fig.sub.2},~\ref{Fig.sub.3}). 

Figure~\ref{Fig.sub.1} illustrates the causal diagram considered in the single-proxy scenario. When both $W$ and $U$ are discrete random variables with finite number of choices, the state of the art research include \cite{pearl2012measurement}. More specifically, they prove that the true causal mechanism $p(y \mid do(x))$ is identifiable when the probability transition matrix $P(W \mid U)$ is fully observable and invertible. 

When $P(\bm{W}\mid \bm{U})$ is not observable, Pearl further considered the double-proxy cases extended from \cite{cai2012identifying}, where the exposure proxy control $\bm{Z}$ and the outcome proxy control $\bm{W}$ both exist. With the auxiliary of $\bm{Z}$, the reversibility of $P(\bm{W} \mid \bm{U})$ is strengthened to that of $P(\bm{Z},\bm{W}\mid x)$ and $P(y, \bm{Z},\bm{W}\mid x)$. This was further extended to a new topic 'negative control' \cite{miao2018identifying, cui2020semiparametric, tchetgen2020introduction, deaner2018proxy,shi2020multiply,singh2020kernel, nagasawa2018identification,kallus2021causal}. In all, these work are all cursed by certain reversibility or completeness assumptions or their weaker forms.

\begin{figure}[h]
\centering  
\subfigure[]{
\label{Fig.sub.1}
\includegraphics[scale = 0.5]{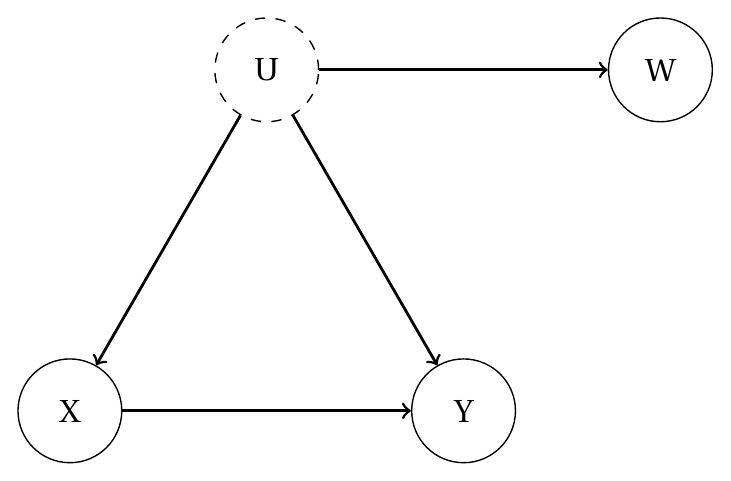}}
\subfigure[]{
\label{Fig.sub.2}
\includegraphics[scale = 0.5]{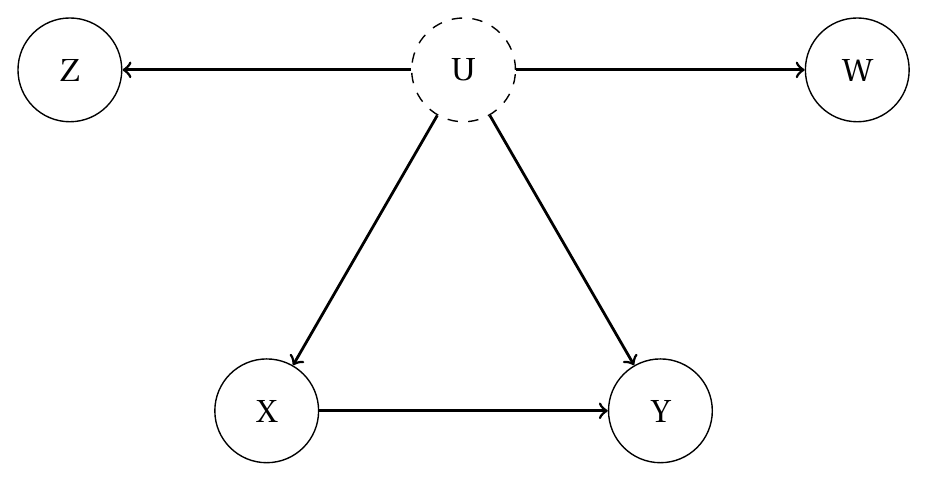}}
\subfigure[]{
\label{Fig.sub.3}
\includegraphics[scale = 0.5]{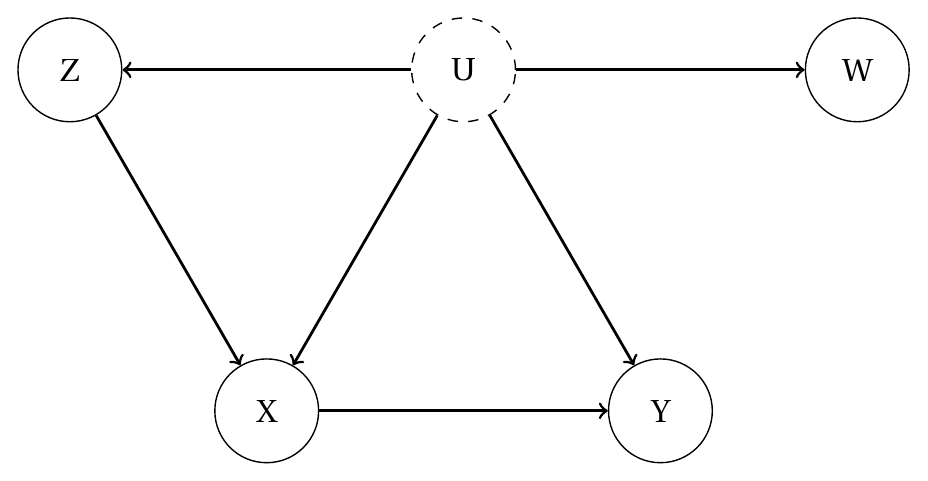}}
\caption{Estimating ACE with confoundings via (a) single control or (b)(c) double control. The nodes denote: $\bm{Z,W}-$ proxies, $\bm{X}-$ treatment, $\bm{Y}-$ outcome, and $\bm{U}-$ unobserved confoundings. }
\label{fig1}
\end{figure}

Double proxy requires observing new auxilliary variable, which may not be practical in real applications, in this paper, we revisit the single proxy case.
In conclusion, excessively strong conditions on $P(\bm{W} \mid \bm{U})$ are imposed to sufficiently achieve the accurate value of ACE. In this paper, we propose new algorithms to identify a bound of the causal effect when the probability transition matrix $P(\bm{W} \mid \bm{U})$ is only partially observable. Moreover, our method does not require the $P(\bm{W} \mid \bm{U})$ to be invertible for the desired identification guarantee. Our algorithm is a fractional programming based approach which seeks learning a bound of the causal effect via solving a constraint fractional program.

By this motivation, in our paper, we focus on the single-proxy case and attempt to weaken the condition 'total precise observability' of $P(\bm{W} \mid \bm{U})$ to 'partial observability'. That is, for each $dim(\bm{W})$-dimensional vector $P(\bm{W} \mid \bm{U} = u)$, we only assume that it is located in a particular subarea instead of a fixed point in the $dim(\bm{W})$-dimensional space. More importantly, such completeness/reversibility conditions in the previous proxy control is not required. On this basis, we formulate this as a constraint fractional programming problem and developed new optimization approaches to solve this problem. This is different from traditional fractional programming methods~\cite{stancu2012fractional}, since the original strong assumptions about the concavity do not exist. To summarize, compared with the previous literature, we quest for the partial identification of ACE, rather than its unique closed form, under the weaker assumption.

The paper is organized as follows. In section.~1, we introduce the basic knowledge of partial identification. In section.~2, we review the construction of ACE and the evolution of the relevant hypotheses in the previous literature. To address their shortcomings, our new hypothesis is proposed, which possesses deeper intuitivity, applicability, and verifiability. In section.~3, we establish the estimation of ACE as a sum-of-ratios fractional programming problem. In formulation, we explicitly construct our objective function and the identification region of the solutions. Then in section.~4. we solve our problem by branch-and-bound strategy in practice. In section.~5, we focus on the theoretical global convergence property of our algorithm. In section.~6, we make simulations to show the effectiveness of our algorithm. Finally, in section.~7, we provide several further topics and discussions to illustrate the great generalizability and scalability of our approach.

\section{Preliminaries}

The causal effect is strongly related to the 'do' operator~\cite{pearl2000models, reason:Pearl09a}, which can be seen as an external intervention. Specifically, the causal effect of treatment $\bm{X}$ on outcome $\bm{Y}$ is denoted as $f(y \mid do(x))$ in Fig.~\ref{fig1}, where the symbol $do(x)$ represents that the treatment $\bm{X}$ is forced to be a fixed value $x$, and $f(\cdot)$ denotes the probability mass/density function for discrete/continuous variables. According to the back-door criteria~\cite{pearl2000models}, $f(y \mid do(x))$ is identified as follows:
\begin{equation}
    \begin{aligned}
    f(y \mid do(x)) = \sum_{i=1}^{dim(\bm{U})}f(y \mid u_i, x)f(u_i) = f(y,x) + \sum_{i=1}^{dim(\bm{U})} \frac{f(y,u_i,X=x) f(u_i,X\neq x)}{f(u_i,X=x)}, \label{basic_formulation}
    \end{aligned}
\end{equation}
where $dim(\cdot)$ denotes the dimension of the variables. The decomposition in the second equation is due to $f(u_i) = f(u_i,X=x)+f(u_i,X\neq x)$. In \cite{kuroki2014measurement}, they assumed that the transition matrix $P(\bm{W}\mid \bm{U})$ is totally observable and reversible. Then they claimed that $f(y \mid do(x))$ is identifiable, namely that the value of each item in Eqn.~\eqref{basic_formulation} can be explicitly extracted as follows\footnote{For convenience in our paper, we use the bold letters to denote the column vector formed by all its corresponding possible values. For instance, $f(y,\bm{U},X=x) = [f(y,u_1,X=x), f(y,u_2,X=x),...f(y,u_{dim(\bm{U})},X=x)]^{T}$. Moreover, if there are two bold letters in a symbol such as $P(\bm{W} \mid \bm{U})$, it denotes the matrix namely that $[f(\bm{W} \mid {u_1}), f(\bm{W} \mid u_2),...f(\bm{W} \mid u_{dim(\bm{U}))}]$, where $f(\bm{W} \mid u_i) = [f(w_1 \mid u_i),f(w_2 \mid u_i),...f(w_{dim(\bm{W})} \mid u_i)]^{T}, i=1,2,...dim(\bm{U})$.}:
\begin{equation}
    \begin{aligned}
    \left[\begin{matrix} 
    &f(y,\bm{U},X=x) \\
    &f(\bm{U},X=x) \\
    &f(\bm{U},X\neq x) 
    \end{matrix}\right] = {P(\bm{W}\mid \bm{U})}^{-1} \left[\begin{matrix} 
    &f(y,\bm{W},X=x) \\
    &f(\bm{W},X=x) \\
    &f(\bm{W},X\neq x) 
    \end{matrix}\right].
    \end{aligned}
\end{equation}\label{inverse_matrix}

Our paper is for generalization. We consider the partial identification of $f(y \mid do(x))$ instead of its unique form computation. This is due to our weakening of assumption on $P(\bm{W} \mid \bm{U})$. Compared with \cite{kuroki2014measurement}, we relax its total observability to the partial observability, and delete the guarantee for its reversibility (thus ${P(\bm{W} \mid \bm{U})}^{-1}$ in Eqn.~\eqref{inverse_matrix} may not exist). Specifically, we extend the identification region of $P(\bm{W}\mid \bm{U})$ from a fixed distribution to the family $\mathscr{P}$, such that:
\begin{equation}
    \begin{aligned}
    \mathscr{P} = \{{P^{}(\bm{W} \mid \bm{U})}: \left[ {P^{}(\bm{W} \mid \bm{U})} - \underline{P^{}(\bm{W} \mid \bm{U})} , \overline{P^{}(\bm{W} \mid \bm{U})} - {P^{}(\bm{W} \mid \bm{U})} \right]\text{is~ non-negative}\}, \label{original_ass_wu}
    \end{aligned}
\end{equation}
where $\underline{P^{}(\bm{W} \mid \bm{U})}$ and $\overline{P^{}(\bm{W} \mid \bm{U})}$ are two priori known matrices to bound ${P^{}(\bm{W} \mid \bm{U})}$. This is a common scenario in the real-world. Although \cite{kuroki2014measurement, greenland2005multiple} have already generally corroborated that this partial observability is verifiable, it has not been fully discussed in the recent literature. In our paper, we will reiterate the condition $P(\bm{W}\mid \bm{U}) \in \mathscr{P}$ as the 'partial observability assumption' in our following text. Under this assumption, it is natural to set up our original goal - seeking the lower bound of $f(y \mid do(x))$ (upper bound is symmetric) via solving the following partial identification problem:
\begin{equation}
    \begin{aligned}
    f(y,X=x) + \min_{f(y, \bm{W}, \bm{U},\bm{X}) \in \mathcal{F}} \backslash \max_{f(y, \bm{W}, \bm{U},\bm{X}) \in \mathcal{F}} \sum_{i=1}^{dim(\bm{U})} \frac{f(y,u_i,X=x) f(u_i,X\neq x)}{f(u_i,X=x)}.\label{simple_formulation}
    \end{aligned}
\end{equation}
Here $f(y, \bm{W}, \bm{U},\bm{X})$ is a three-order ($dim(\bm{W})* dim(\bm{U})* dim(\bm{X})$) tensor indicating the joint probability distribution of each $ w \in \bm{W}, u \in \bm{U}, x \in \bm{X}$ together with $\bm{Y}=y$. Then the set $\mathcal{F}$ = \{$f(y, \bm{W}, \bm{U},\bm{X}): f(y, \bm{W}, \bm{U},\bm{X})$ is compatible with $P(\bm{W}\mid \bm{U}) \in \mathscr{P}$\}. 

Achieving this goal faces with challenges. Firstly, its tight bound is hard to be achieved. It is due to the difficulty of representing feasible region $\mathcal{F}$ in a closed form. Its boundary constraints contains the partial observable $P(\bm{W} \mid \bm{U})$, which can be seen as a well-known inverse problem called first-kind Fredholm integral equation\footnote{One of the boundary constraints of $\mathcal{F}$ can be expressed as $f(y,w,x) = \sum_{i=1}^{dim(\bm{U})} (f(w_j\mid u_i)\sum_{j=1}^{dim(\bm{W})}f(y,w_j,u_i,x))$, where each $f(w\mid u)$ is bounded by $\underline{P(\bm{W} \mid \bm{U})}$ and $\overline{P(\bm{W} \mid \bm{U})}$. It is in the form of the first-kind Fredholm integral equation in the discrete case.}~\cite{landweber1951iteration, tarantola2005inverse} in the discrete case. It is ill-posed when $P(\bm{W} \mid \bm{U})$ is irreversible and the closed form expression of $\mathcal{F}$ can only be approximated iteratively by complex numerical methods~\cite{strand1968statistical}. With this reason, we attempt to relax the feasible region from $\mathcal{F}$ to $\mathcal{\widetilde{F}}$ ($\mathcal{F} \subseteq \mathcal{\widetilde{F}}$), which contains a closed-form expression. Specifically, the relaxed condition $f(y, \bm{W}, \bm{U},\bm{X}) \in \mathcal{\widetilde{F}}$ is to keep the feasible region of $f(y,\bm{U}, X=x), f(y,\bm{U}, X= x), f(\bm{U}, X\neq x)$ in a calculable closed-form, which will be denoted as $IR_{F(y,U,X=x)}, IR_{F(U,X=x)},  IR_{F(U,X\neq x)}$ respectively in our final objective function.

Secondly, even if we retreat and seek its valid bound as above, it is still non-trivial. It is due to the difficulty of finding a corresponding optimization method. As the causal effect is expressed as a form of fractional summation, we naturally resort to techniques in sum-of-ratios fractional programming (SFP). The general form of SFP summarized in \cite{schaible2003fractional} is represented as follows:
\begin{equation}
    \begin{aligned}
    \min \{\sum_{i=1}^{M} \frac{g_{1i}(\bm{\phi})}{g_{2i}(\bm{\phi})}\}, \bm{\phi} \in S, g_{1i}(\bm{\phi})~\text{is convex}, g_{2i}(\bm{\phi})~\text{is concave}, g_{1i}(\bm{\phi})\geq 0, g_{2i}(\bm{\phi}) > 0,
    \end{aligned}\label{SFP_form}
\end{equation}
where $S$ is a convex set, and $M \geq 2$ is a integer. In order to ensure the global nature of optimal solutions, $g_{1i}(\Phi)$ and $g_{2i}(\Phi)$ are assumed to be convex and concave respectively. In contrast with Formulation.~\ref{simple_formulation}, we should choose
\begin{equation}
    \begin{aligned}
     & M = dim(\bm{U}), \bm{\phi} = (({f(y,u_i,X=x),...})^T,  (f(u_i,X=x), ...)^T,  (f(u_i,X\neq x),...)^T), \\
    & g_{1i}(\bm{\phi}) = f(y,u_i,X=x)f(u_i,X\neq x),~g_{2i}(\bm{\phi}) = f(u_i,X=x)~i=1,2,...dim(U).
    \end{aligned}\label{transformation}
\end{equation} 
However, our construction violates the traditional convex-concave assumption, since $g_{1i}(\bm{\phi})$ is not convex. Thus these previous SFP algorithms~\cite{schaible2003fractional} do not work.

In order to handle this case, we design a algorithm called Partial Identification with Sum-of-ratios Fractional Programming (PI-SFP). This algorithm is motivated by branch and bound strategy~\cite{lawler1966branch,branch-bound-SFP} and DC programming~\cite{horst1999dc, tao1997convex, pei2013global}, namely that we iteratively search the optimal bound by means of feasible region partition. We also provide the complete convergence analysis. To our knowledge, our paper is a new attempt to estimating casual effect via this optimization technique. Moreover, this algorithm also contributes to the existing literature on the convergence analysis in branch and bound strategy~\cite{dai2005conical,lawler1966branch,branch-bound-SFP, pei2013global}.

For recent literature, just because of these two challenges of solving the partial observability case, they avoided further discussion on the observability of $P(\bm{W}\mid \bm{U})$. Instead, they introduced another auxiliary variable $\bm{Z}$ and formalized the problem as the double negative control~\cite{miao2018identifying, cui2020semiparametric, tchetgen2020introduction, deaner2018proxy,shi2020multiply,singh2020kernel, nagasawa2018identification,kallus2021causal}. However, as shown in Table.~\ref{table_literature}, there is no free lunch. These work are also restricted by additional assumptions about $\bm{Z}$, such as completeness condition, bridge function condition, etc. Importantly, these work are still all based on the reversibility of $P(\bm{W} \mid \bm{U})$ just except for~\cite{kallus2021causal}, who substituted it as a weaker bridge function condition. Hence when the irreversibility (i.e., multilinearity in some rows or columns) of the $P(\bm{W}\mid \bm{U})$ occurs in our real world, these methods will be easily invalidated. In fact, difficulties have already been encountered when doing numerical computations if the conditional number of $P(\bm{W}\mid \bm{U})$ is too large\footnote{The conditional number of matrix $A$ is denoted as $\kappa(A) = \frac{\sigma_{max}(A)}{\sigma_{min}(A)}$, where $\sigma_{max}(A)$, $\sigma_{min}(A)$ denote the maximal/minimal singular values of $A$. If some rows/columns of $A$ are similar (or equal), then $\kappa(A)$ is large (or $+\infty$), and $A^{-1}$ is computationally hard (or even not exists.)}.

In conclusion, the revisit of single-proxy control under the partial observability of $P(\bm{W} \mid \bm{U})$ is challenging but necessary. It corresponds to a few common real-world scenarios, serving as the blank spots of these double-proxy control methods as illustrated above. In our paper, we will do a deeper exploration on estimating causal effect with this assumption. We propose algorithm called PI-SFP. Our contributions are summarized as follows:

\begin{table}[]
	\caption{Tools and assumptions of previous literature on partial identification. \cite{kuroki2014measurement}$(1)$ is with external studies, while $(2)$ is without external studies.}
	{
	\begin{tabular}{c | ccc | ccc}
		\hline
		\multirow{2}{*}{Literature}    & \multicolumn{3}{c|}{Tools} & \multicolumn{3}{c}{Assumptions }        \\ 
		\cline{2-7} 
		& \makecell[c]{Valid \\Instrument}        & \makecell[c]{Negative\\ exposure}   &\makecell[c]{Negative\\outcome}     & \makecell[c]{reversibility\\ completeness} &\makecell[c]{Bridge\\ function}  & \makecell[c]{Observability \\ of $P(\bm{W}\mid \bm{U})$}\\
		\hline
		\makecell[l]{{\cite{balke1994counterfactual}}\\ 	\cite{kitagawa2009identification}}  & \Checkmark          & \XSolid     &\XSolid        & \XSolid            &\XSolid         &\XSolid     \\
		\hline
		\makecell[l]{\cite{kuroki2014measurement}(1) \\ \cite{rothman2008modern} \\ \cite{leecausal}} & \XSolid          & \XSolid     &\Checkmark        & \Checkmark           &\XSolid         &\makecell[c]{\Checkmark \footnote{$P(\bm{W}\mid \bm{U})$ is assumed to be reversible and explicitly, totally observed.}}     \\
		\hline
		\makecell[l]{\cite{kuroki2014measurement}(2) \\ \cite{nagasawa2018identification}}& \XSolid          & \Checkmark     &\Checkmark        & \Checkmark            &\XSolid         &\XSolid     \\
		\hline
		\makecell[l]{\cite{miao2018identifying} \\\cite{shi2020multiply}\\ \cite{singh2020kernel} \\ \cite{cui2020semiparametric} \\\cite{tchetgen2020introduction} \\\cite{deaner2018proxy}} & \XSolid          & \Checkmark     &\Checkmark        & \Checkmark            &\Checkmark     &\XSolid     \\
		\hline
		\makecell[l]{\cite{kallus2021causal}} & \XSolid          & \Checkmark     &\Checkmark        & \XSolid          &\Checkmark     &\XSolid     \\ 
		\hline
		\makecell[l]{\textbf{Our paper}} & \XSolidBold          & \XSolidBold     &\CheckmarkBold        & \XSolidBold          &\XSolidBold     &\makecell[c]{\CheckmarkBold \footnote{In our paper, $P(\bm{W}\mid \bm{U})$ only needs be partially bounded.}\\   } \\
		\hline
	\end{tabular}  } 
	\label{table_literature}
\end{table}

\begin{itemize}
    \item We propose a novel analytical framework of seeking the valid bound of causal effect $f(y \mid do(x))$ via the partial observability of $P(\bm{W}\mid \bm{U})$, and provide a sufficient and necessary condition to justify whether the bound is tight or not.
    \item We develop a global optimization strategy called Partial Identification via Sum-of-ratios Fractional programming (PI-SFP). We theoretically prove that PI-SFP algorithm globally converges and can achieve the valid bound of $f(y \mid do(x))$ in an exponential rate. 
    \item We analyze the rationality and generalizability of PI-SFP via extended discussions, such as 1) motivation of partial observability assumption, 2) acceleration of PI-SFP, 3) graph structure extension and 4) generalization to the continuous confoundings.
\end{itemize}

\section{A fractional programming framework for partial identification}\label{framework}

\subsection{Definitions and assumptions}

Based on preliminaries, we reorganize all definitions and assumptions as follows.

\begin{definition}
$Y,Y_{0},Y_{1}\in [Y^L,Y^U], Z \in [Z^L, Z^U], X \in [X^{L}, X^{U}]$, $W \in [W^L, W^U]$, $U\in [U^{L}, U^{U}]$. Moreover, we use $dim(\bm{\cdot})$ to denote the dimension of variables. $d := dim(\bm{U})<+\infty$, and set of confoundings $\bm{U}$ is $\{u_1,u_2,...u_{d}\}$.\label{def_bound}
\end{definition}

In our paper, we consider the case $\bm{Z}$, $\bm{X}$, $\bm{W}$ are all discrete with dimensions $dim(\bm{Z})$, $dim(\bm{X})$, $dim(\bm{W})$\footnote{If $\bm{Z}, \bm{X}, \bm{W}$ are continuous in the real-world, we will do segmentation in their corresponding continuous intervals}. $Y$ can be discrete or continuous.

\begin{definition}
$Y_x$ is the value of $\bm{Y}$ when $X$ is forced to be $x$. On this basis, $ACE_{\bm{X} \rightarrow \bm{Y}}$ denotes the average causal effect (ACE) from $\bm{X}$ to $\bm{Y}$, namely that 
\begin{equation}
    \begin{aligned}
    ACE_{\bm{X} \rightarrow \bm{Y}} = \int_{x} E(Y_x)\pi(x) dx = \int_{x} \int_y f(Y_x=y)\pi(x) dx dy,
    \end{aligned}
\end{equation}

where $\pi(x)$ is a weight function of $\bm{X}$.

\label{def_ACE}
\end{definition}

It is called as generalized average causal effect in~\cite{kallus2021causal}. Moreover, it can degenerate to the traditional form~\cite{pearl2013testability} as
$ACE_{\bm{X}\rightarrow \bm{Y}} = E(Y_1)-E(Y_0)$, if we choose $dim(X)=2, X=\{0,1\}$, and $\pi(x) = \bm{sgn}(x)
$, where $\bm{sgn}(\cdot)$ is the sign function.

\begin{assumption}{\textbf{(partial observability assumption)}} $P(\bm{W} \mid \bm{U}) \in \mathscr{P}$. \label{ass_partial_bounded}
\end{assumption}
Here the set $\mathscr{P}$ is identified in Formulation.~\ref{original_ass_wu}.
$\underline{P^{}({\bm{W}} \mid {\bm{U})}}$ and $\overline{P^{}(\bm{W} \mid \bm{U})}$ are two a priori known matrices. According to Ass.~\ref{ass_partial_bounded}, we can derive that

\begin{equation}
    \begin{aligned}
    \left[\begin{matrix}
     +\overline{P^{}(\bm{W} \mid \bm{U})} f(y,\bm{U},X=x) - f(y,\bm{W},X=x)\\
    -\underline{P^{}(\bm{W} \mid \bm{U})} f(y,\bm{U},X=x) +f(y,\bm{W},X=x)
    \end{matrix}\right]
     \geq 0. 
    \end{aligned}\label{partial_observe}
\end{equation}
Here we use $\bm{S}\geq 0$ to denote $\bm{S}$ is a non-negative matrix. These inequalities is for preparation of the construction of the relaxed feasible region of $f(y, \bm{U}, \bm{W}, \bm{X})$. In the following part, we construct the model framework of searching the valid bound of $f(Y_x = y)$, and then extend it to the ACE case.

\subsection{{Objective function}}\label{sec_obj_fun}

This section aims to formalize the single proxy control under Ass.~\ref{ass_partial_bounded} into an optimization problem. In this process, we confront and address the two challenges illustrated in the preliminaries.

Recalling the Construction~(\ref{simple_formulation}), we only need to consider the minimum case, and the maximum case is symmetric. Our original goal is:
\begin{equation}
    \begin{aligned}
    &\text{min~} f(y,X=x)  +  \sum_{i=1}^{d} \frac{f(y,u_i,X=x) f(u_i, X\neq x)}{f(u_i,X=x)}\\  &\text{subject to: $f(y,\bm{U},\bm{W}, \bm{X}) \in \mathcal{F}$}.
    \label{eqn_basic_bound}
    \end{aligned}
\end{equation}
$\mathcal{F} = \{f(y, \bm{U}, \bm{W}, \bm{X}): f(y, \bm{U}, \bm{W}, \bm{X})$ \text{~is compatible with} \text{Ass.~\ref{ass_partial_bounded}} and observed $f(y, \bm{W}, \bm{X})$ \}. 
 
As we suggested in the preliminaries, the first challenge is the nonexistence of closed-form expression of $\mathcal{F}$. To solve it, we introduce the new symbol $\mathcal{\widetilde{F}}$ to formally describe the relaxation of the identification region of $f(y,\bm{U},\bm{W}, \bm{X})$. For preparation, we introduce the symbol $\bm{\theta}, \bm{\psi}, \bm{\omega}$ and follow the previous notation $\bm{\phi}$ in Formulation~(\ref{transformation}): 
\begin{equation}
    \begin{aligned}
    \begin{matrix}
    \theta_i &= f(y,u_i,X=x)\\
    \psi_i &= f(u_i,X=x)\\
    \omega_i &= f(u_i,X\neq x)\\
    \end{matrix},~
    \begin{matrix}
    \bm{\theta} &= (\theta_1, \theta_2,...\theta_d)^T\\
    \bm{\psi} &= (\psi_1, \psi_2,...\psi_d)^T\\
    \bm{\omega} &= (\omega_1, \omega_2,...\omega_d)^T\\
    \end{matrix},
    ~\bm{\phi} = \left(\begin{matrix}
    \bm{\theta}~  \bm{\psi}~ \bm{\omega}
    \end{matrix} \right).
    \end{aligned}
\end{equation}

Then we construct a broader set $\mathcal{\widetilde{F}}$ as follows:  
\begin{equation}
    \begin{aligned}
    \mathcal{\widetilde{F}}= \left\{f(y,\bm{U},\bm{W}, \bm{X}): \bm{\phi} \in IR_{\bm{\Phi}}, IR_{\bm{\Phi}} = IR^{1}_{\bm{\Phi}}\cap IR^{2}_{\bm{\Phi}} \right\},
    \end{aligned}\label{construction_F}
\end{equation}

where the set $IR^{1}_{\bm{\Phi}}$ is constructed by Formulation~(\ref{partial_observe}):

\begin{equation}
\begin{aligned}
   &IR^{1}_{\bm{\Phi}} = \{\bm{\phi}: \left[\begin{matrix}
   -\bm{I_{d*d}} \\ \bm{I_{d*d}}
   \end{matrix}\right] \left[\begin{matrix}
   &f(y,\bm{W},X=x)^{T} \\ &f(\bm{W},X=x)^{T} \\ &f(\bm{W},X\neq x)^{T}
   \end{matrix} \right]^{T}  -  \left[ \begin{matrix}
   &-\overline{P(\bm{W}\mid \bm{U})} \\  &\underline{P(\bm{W} \mid \bm{U})}
   \end{matrix} \right] \bm{\phi} \geq \bm{0}\}.\\
   \end{aligned}\label{IR_1}
\end{equation}

Here $\bm{I_{d*d}}$ denotes the $d*d$ identity matrix. Moreover, the set $IR^{2}_{\bm{\Phi}}$ indicates the natural constraints by default:

\begin{equation}
    \begin{aligned}
    IR_{\bm{\Phi}}^{2} = \left\{ \bm{\phi}:
    \left[\begin{matrix}
    &\bm{1_{1*d}} \bm{\theta} \\ &\bm{1_{1*d}} \bm{\phi} \\ &\bm{1_{1*d}} \bm{\omega} 
    \end{matrix}\right] = \left[ \begin{matrix}
    &f(y,X=x)\\ &f(X=x)\\ &f(X\neq x) \}
    \end{matrix}\right], \forall i,
\left\{\begin{matrix}
    \theta_i \in [0,f(y,X=x)] \\
    \phi_i \in (0,f(X=x)] \\
    \omega_i \in [0,f(X\neq x)]
\end{matrix} \right\}. \right\}    .
    \end{aligned}
\end{equation}
Here $\bm{1_{1*d}}$ denotes the $1*d$ all-ones vector. By this construction, the enclosure property $\mathcal{F} \subseteq \widetilde{\mathcal{F}}$ is guaranteed as follows.
\begin{proposition} $\mathcal{F}$ is enclosed by $\widetilde{\mathcal{F}}$, namely that $\mathcal{F} \subseteq \mathcal{\widetilde{F}}$. \label{basic_IR}
\end{proposition}

The proof is shown in the Appendix.~\ref{proof_basic_IR}. Proposition.~\ref{basic_IR} provides the extension of the feasible region of $f(y, \bm{U}, \bm{W}, \bm{X})$ from $\mathcal{F}$ to $\mathcal{\widetilde{F}}$. On this basis, Formulation.~\ref{eqn_basic_bound} is relaxed as follows:

\begin{equation}
    \begin{aligned}
   \underline{f(Y_x = y)} = ~&  \text{min~} f(y,X=x) +  \sum_{i=1}^{d}\frac{1}{\psi_{i}}\theta_{i}\omega_{i}\\
    &\text{subject to:~} f(y,\bm{U},\bm{W}, \bm{X}) \in \mathcal{\widetilde{F}}, i.e., \bm{\phi} \in IR_{\bm{\Phi}}.
    \end{aligned} \label{formulation}
\end{equation}
Symmetrically, the optimal value is denoted as $\overline{f(Y_x = y)}$ for the maximum case. Moreover, the corresponding set of optimal solutions are denoted as $\bm{\Phi_{opt}}$. The following proposition discuss the validity and tightness of $\underline{f(Y_x = y)}$:
\begin{proposition}
The outcome $\underline{f(Y_x = y)}$ serves as the lower bound of $f(Y_x = y)$. Moreover, this bound is \textbf{tight} if and only if the following set is not empty:
\begin{equation}
    \begin{aligned}
       \{f(y,\bm{U},\bm{W},\bm{X}) : f(y,\bm{U},\bm{W},\bm{X})\in \mathcal{F} \text{~and is compatible with some~} \bm{\phi_{opt}}\in \bm{\Phi_{opt}}\} \neq \emptyset ,
    \end{aligned}\label{constraint_prove_tight}
\end{equation}
where $\bm{\phi_{opt}}$ is an element of the set $\bm{\Phi_{opt}}$. The maximum case $\overline{f(Y_x = y)}$ is symmetric. \label{proposition_tight}
\end{proposition}
Details are deduced in Appendix.~\ref{equal_reform_algorithm1}. 

\begin{remark}
As illustrated in the preliminaries, it is hard to theoretically guarantee the tightness of $\underline{f(Y_x = y)}$. However in practice, it is tight in many cases, since the Constraint~\eqref{constraint_prove_tight} is not hard to be satisfied. For instance, if we obtain the following observations and the partial observability:
\begin{equation}
    \begin{aligned}
        \left[ \begin{matrix}f(Y = y, \bm{W}, X=x)^T \\ f(Y = y, \bm{W}, X\neq x)^T  \\
        f(Y \neq y, \bm{W}, X=x)^T \\ f(Y \neq y, \bm{W}, X\neq x)^T
        \end{matrix}\right]
        = \left[\begin{matrix}
0.08 & 0.12\\ 0.15 & 0.1 \\ 0.18 & 0.12\\ 0.15 & 0.1 \\
\end{matrix}\right], \left[\begin{matrix}
\overline{P(\bm{W}\mid \bm{U})} \\ \underline{P(\bm{W}\mid \bm{U})}
\end{matrix}\right] = \left[ \begin{matrix}
0.6\bm{I_{2*2}} + 0.4\bm{J_{2*2}} \\ 0.6\bm{I_{2*2}} 
\end{matrix}\right].
    \end{aligned} \label{tight_example}
\end{equation}
Here $\bm{W}, \bm{U}, \bm{X}$ are all binary, and $\bm{I_{n*n}}, \bm{J_{n*n}}$ denote the $n-$dimensional identity matrix and all-ones matrix respectively. we can verify one of the optimal solutions $\bm{\phi_{opt}} = [0~ 0.2~ 0.3~ 0.2~ 0.5~ 0]^T$. The corresponding $f(y,\bm{U},\bm{W},\bm{X})$ satisfying Constraints~(\ref{constraint_prove_tight}) exists, whose explicit form is detailed in Appendix.~\ref{equal_reform_algorithm1} due to space limitation. Thus in this case the bound $\underline{f(Y_x = y)}$ is tight.
\end{remark}

We now aim to address the second challenge. That is, this fractional programming problem is still non-trivial since the invalidation of the convex-concave condition. With this reason, we adopt the \emph{difference-in-convex (DC)} decomposition strategy to formally describe how we reduce (\ref{formulation}) into a relaxed linear programming problem. 

For preparation, we do transformation of this fractional form. We define the knockoff $\bm{\psi^o}$ to replace the denominator and then introduce the $4d-$ dimensional vector $\bm{\gamma}$ :
\begin{equation}
    \begin{aligned}
    \bm{\psi^o} = (\psi_{1}^{o}, \psi_{2}^o,... \psi_{d})^T,~\bm{\gamma} = \left(\begin{matrix}
    (\bm{\psi^o})^T, \bm{\theta}^T, \bm{\psi}^T, \bm{\omega}^T
    \end{matrix} \right)^T, \text{where~} (\bm{\theta}, \bm{\psi}, \bm{\omega}) \text{~is copied from~} \bm{\phi}.
    \end{aligned}\label{definition_gamma}
\end{equation}
Then our objective function is equivalently transformed to
\begin{equation}
    \begin{aligned}
    \underline{f(Y_x = y)} = &\text{~min~} f(y ,X=x) + \sum_{i=1}^{d} \psi_{i}^{o} \theta_{i} \omega_{i} \\
    &\text{~subject to}: \bm{\gamma} \in IR_{{\Gamma}}, \text{where~} IR_{\Gamma} = \{\bm{\gamma}: \bm{\phi} \in IR_{{\Phi}},
    \psi_{i}^{o} \psi_{i}  = 1, i=1,...d.\}.\\
    \end{aligned}   \label{re-formulation}
\end{equation}

This is the final goal. However, it is still hard in practice in spite of implementing the knock-off trick. On the one hand, it is non convex and nonlinear both for the objective function and the constraints. On the other hand, we cannot just look for local optimal solutions, or else the bound $\underline{f(Y_x = y)}$ can not be guaranteed to be valid. By this motivation, we attempt to construct a weaker linear programming form to approximate the global optimal value of (\ref{re-formulation}). Our core idea is to apply the \emph{difference-in-convex (DC)} decomposition :
\begin{equation}
    \begin{aligned}
     \forall \bm{\gamma}, &\sum_{i=1}^{d}\psi_{i}^o \theta_{i} \omega_{i} = C_1(\bm{\gamma}) -C_2(\bm{\gamma}), \\
     & \psi_{i}^{o}\psi_{i} = D_{i1}(\bm{\gamma}) - D_{i2}(\bm{\gamma}), i=1, 2, \cdots, d, \label{DC_decomposition}
    \end{aligned}
\end{equation}
where $C_1(\bm{\gamma}), C_2(\bm{\gamma}), D_{i1}(\bm{\gamma}), D_{i2}(\bm{\gamma})$ \footnote{Note that the sub-script $cyc$ in~\eqref{dc_decomposition} is an abbreviation of cyclic sum following~\cite{du2012note}, which cycles through $\{\psi^o_i,\theta_{i},\omega_{i}\}$ in the corresponding function and take the sum. Taking $D_{i2}(\bm{\gamma})$ for instance, we have $\sum_{cyc}[\psi^o_i+\theta_i^2]^2 = [\psi^o_i+\theta_i^2]^2 + [\theta_{i}+\omega_{i}^2]^2 + [\omega_{i}+(\psi^o_{i})^2]^2$.} are all convex functions (see Appendix.~\ref{app_main_result_1}) satisfying that
{\begin{equation}
    \begin{aligned}
      &C_1(\bm{\gamma}) =\sum_{i=1}^{d} \frac{1}{6}(\sum\limits_{cyc} \psi^o_i )^3+\frac{1}{2}\sum\limits_{cyc}(\psi^o_i)^4 + \frac{1}{2}\sum\limits_{cyc}(\psi^o_i)^2,\\
    &C_2(\bm{\gamma}) = \sum_{i=1}^{d}  \frac{1}{6}\sum\limits_{cyc} (\psi^o_i)^3+\frac{1}{4}\sum_{cyc}[(\psi^o_i)^2+\theta_i]^2 + \frac{1}{4}\sum_{cyc}[\psi_i^o+\theta_i^2]^2,\\
    &D_{i1}(\bm{\gamma}) = \frac{1}{2}(\psi_i^o+\psi_i)^2,~
    D_{i2}(\bm{\gamma}) = \frac{1}{2}[(\psi^o_i)^2 + (\psi_i)^2].
    \end{aligned}\label{dc_decomposition}
\end{equation}}

Exploiting their convexity, we bound them by the following linear functions, which are constructed by secants and tangents of the original function: 
\begin{equation}
\begin{aligned}
C_1(\bm{\gamma}) -C_2(\bm{\gamma}) & \geq  C_1^{\text{tan}}(\bm{\gamma}) - C_2^{\text{sec}}(\bm{\gamma}) \\
 D_{i1}(\bm{\gamma}) - D_{i2}(\bm{\gamma}) & \in [D_{i1}^{\text{tan}}(\bm{\gamma}) - D_{i2}^{\text{sec}}(\bm{\gamma}), D_{i1}^{\text{sec}}(\bm{\gamma}) - D_{i2}^{\text{tan}}(\bm{\gamma})]
\end{aligned}
\end{equation}
For their explicit form solutions, we refer the readers to \eqref{tan_sec}.
This allows us to relax the original problem in~(\ref{re-formulation}) into the following linear program:
\begin{equation}
    \begin{aligned}
    &\text{min~}f(y ,X=x) + C_1^{\text{tan}}(\bm{\gamma}) - C_2^{\text{sec}}(\bm{\gamma}) \\
    &\text{subject to}: \bm{\phi} \in IR_{{\Phi}}, D_{i1}^{\text{tan}}(\bm{\gamma}) - D_{i2}^{\text{sec}}(\bm{\gamma}) \leq 1 , D_{i1}^{\text{sec}}(\bm{\gamma}) - D_{i2}^{\text{tan}}(\bm{\gamma}) \geq 1, i=1, 2, \cdots, d.
    \end{aligned}   \label{re-formulation_linear_weaker}
\end{equation}
It is clear that this shift causes the estimation error. In order to eliminate it in practice, we iteratively do DC within simplicial partitioned feasible regions. Details will be shown in the following section. 

In conclusion, we already address these two challenges in the preliminaries.

\subsection{Valid bound of ACE}
The identification region of $f(Y_x)$ is constructed as follows. 
\begin{equation}
    \begin{aligned}
    IR_{F(Y_x)} = \{f(Y_x): \int_{Y^L}^{Y^U}f({Y_x=y})dy =1, \forall y \in [Y^{L}, Y^{U}], f(y,\bm{U}, \bm{W}, \bm{X}) \in \mathcal{F} \}.
    \end{aligned}
\end{equation}
Then the valid bound of $ACE_{\bm{X} \rightarrow \bm{Y}}$ can be denoted as $[\underline{ ACE_{\bm{X} \rightarrow \bm{Y}}}, \overline{ ACE_{\bm{X} \rightarrow \bm{Y}}}]$:
\begin{equation}
    \begin{aligned}
    &\underline{ ACE_{\bm{X} \rightarrow \bm{Y}}} \leq \min\{ ACE_{\bm{X} \rightarrow \bm{Y}} = \int_{{X}^L}^{{X}^U} \int_{Y^L}^{Y^U} f(Y_x=y)\pi(x) dx dy: f({Y_x}) \in IR_{F(Y_x)}\},~\\
    &\overline{ ACE_{\bm{X} \rightarrow \bm{Y}}} \geq \max\{ ACE_{\bm{X} \rightarrow \bm{Y}} = \int_{X^{L}}^{X^{U}} \int_{Y^L}^{Y^U} f(Y_x=y)\pi(x) dx dy: f({Y_x}) \in IR_{F(Y_x)}\}.
    \end{aligned}\label{tight_bound_ace}
\end{equation}

$[\underline{ ACE_{\bm{X} \rightarrow \bm{Y}}}, \overline{ ACE_{\bm{X} \rightarrow \bm{Y}}}]$ is the valid bound of ACE. In our paper, we aim to design an algorithm to seek the valid bound of $f(Y_x = y)$, and then extend our strategy from bounding $f(Y_x = y)$ to bounding ACE. Homoplastically, we only need to consider the optimization technique on the minimum case, and the maximum case will be symmetric.

\section{Algorithm}\label{section_algorithm}
In this section, we showcase how to compute $\underline{f(Y_x = y)}$ in~\eqref{re-formulation} in practice. As illustrated above, since it corresponds to optimizing a non-convex function, new optimization techniques needs to be derived in order to find the global optimum. On this basis, we propose Partial Identification via Sum-of-ratios Fractional Programming (PI-SFP), which is a fractional programming based method that optimizes the objective via iterative approximation. More specifically, we first construct a simplex $S_0$ that encloses the feasible region of~\eqref{re-formulation}, then we use the simplex $S_0$ as an assistance to identify a lower bound of $\underline{f(Y_x = y)}$ via difference-in-convex (DC) decomposition strategy. Then in each iteration, we partition the original simplex $S_{0}$ into multiple simplices to help us fine tune the lower bound constructed in the initial step. The rest of the section is structured as follows. In Section~\ref{sec_framework}, we introduce the main framework of our algorithm. In particular, we divided the entire algorithm into four modules: 1) \textbf{Initialization()}, 2)~\textbf{Bisection()}, 3) \textbf{Bounding()}, and 4) \textbf{Global\textbf{\_error}()}. Then in Section~\ref{initialization_section}, we elaborate these modules in detail. For notational simplicity, we introduce the following symbols for algorithm description:
\begin{itemize}
\item For a simplex $S$, $dia(S) := \max_{s_1,s_2 \in S_{}}\|s_1 - s_2\|_2$ denotes its diameter\footnote{For simplicity, we use $\|\cdot \|$ to denote $\|\cdot \|_2$.}, and $S^i$ denotes its $i-$th supporting vector, $i=\{0,1,...4d\}$.
\item $\underline{f_{S}(Y_{x} = y)}$ denotes the optimal value of (\ref{re-formulation}) when its feasible region is strengthened to $\bm{\gamma} \in IR_{{\Gamma}} \cap S$.

\end{itemize}

\subsection{Framework of PI-SFP}\label{sec_framework}

The framework of PI-SFP to solve (\ref{re-formulation}) is as follows.

%
%

\begin{algorithm}[t]

\caption{Partial Identification via Sum-of-ratios Fractional Programming (PI-SFP).}
\LinesNumbered 
\KwIn{Observational distribution $f(y, \bm{W}, \bm{X})$, $\underline{P(\bm{W} \mid \bm{U})}$, $\overline{P(\bm{W} \mid \bm{U})}$, a prespecified error bound $\delta >0$.}
\KwOut{A lower bound estimate ${\underline{f^k_{opt}(Y_x = y)}}$. }

Let $k = 0$, construct an original simplex
\[S_0 = \textbf{Initialization} (f(y, \bm{W}, \bm{X}), \underline{P(\bm{W} \mid \bm{U})}, \overline{P(\bm{W} \mid \bm{U})}); \]

Calculate a lower bound of $\underline{f_{S_{0}}(Y_x = y)}$ via the $\textbf{Bounding}$ function:
\[\underline{\underline{f_{S_{0}}(Y_x = y)}} = \textbf{Bounding}(S_{0});\]

Set the collection of simplices at the $0$-th iteration as $\cS_0 = \{S_{0}\}$;

\While{${\text{PI-SFP}_{\text{error}}} \leq \delta$}{
	
Let \[\tilde{S}_k = \argmin_{S \in \cS_k} \underline{\underline{f_{S}(Y_x = y)}},\] where $\underline{\underline{f_{S}(Y_x = y)}}$ denotes the output of $\textbf{Bounding}(S)$ with input $S$;

Split  $\tilde{S}_k$ into two simplicies $\tilde{S}_{k1}$ and $\tilde{S}_{k2}$ via the $\textbf{Bisection}$ function:
\[
\tilde{S}_{k1}, \tilde{S}_{k2} = \textbf{Bisection} (\tilde{S}_k)
\]
and set
\[
\cS_{k + 1} = \left(\cS_k \setminus \tilde{S}_k \right) \cup \{\tilde{S}_{k1}, \tilde{S}_{k2}\};
\]

Calculate the estimation error bound via
\[{\text{\emph{PI-SFP}}_{\text{\emph{error}}}} = \textbf{Global\textbf{\_error}}(\tilde{S}_0, \tilde{S}_1, \cdots, \tilde{S}_{k+1});\]

Set $k = k + 1$;
} 

Return $\underline{f^k_{opt}(Y_x = y)} = \max\limits_{i\in \{0,1,...k\}} \underline{\underline{f_{\tilde{S}_i}(Y_x = y)}}$.\label{main_alg}
\end{algorithm}

\textbf{Step~1} is for pre-processing. Using function $\textbf{Initialization()}$, a baseline simplex $S_0$ is constructed to enclose the original feasible region. i.e., $IR_{\Gamma} \subseteq S_0$, and thus $\underline{f(Y_x = y)} = \underline{f_{S_0}(Y_x = y)}$ (see lemma.~\ref{lemma_enclose} in Appendix.~\ref{app_main_result_1}). This equivalent transformation allows us to compute $\underline{f(Y_x = y)}$ via $\underline{f_{S_0}(Y_x = y)}$. Adopting the DC decomposition strategy as in~\eqref{DC_decomposition}-\eqref{re-formulation_linear_weaker}, in \textbf{Step~2} we find a lower bound of $\underline{f_{S_0}(Y_x = y)}$, namely $\underline{\underline{f_{S_0}(Y_x = y)}}$. 

Then in \textbf{Steps 4-9}, we apply a bisection like approach to iteratively partition $S_0$ into a set of simplices $\cS_k$ (in $k$-th iteration) and then reapply the above DC decompsition strategy to new simplices for more accurate estimate. Once  bounding error calculated by \textbf{Global\_error()} reaches the prespecified threshold $\delta$, we stop and return the lower bound estimate in \textbf{Step 10}. Otherwise, we reiterate this step and make more delicate partitions.

\subsection{Internal functions of PI-SFP}
\label{initialization_section}

In this section, the above four functions are illustrated in detail.
\\ \quad \\
\noindent $\large{\textbf{1) Initialization()}}$: This function is to construct an original simplex $S_0$ to enclose the feasible region $IR_{\Gamma}$. Motivated by~\cite{horst2000introduction, pei2013global}, we construct $S_0$ using that:
\begin{equation}
    \begin{aligned}
     S_0 = \{ \bm{\gamma} = (\gamma_1, \cdots, \gamma_{4d}): \sum_{i=1}^{4d} \gamma_i \leq \alpha, \gamma_i \geq \gamma_i^l, \forall i = 1,2, \cdots, 4d. \}, 
    \end{aligned}\label{initialization_S00}
\end{equation}
where $\gamma^{l}_i := \min\limits_{\bm\gamma \in IR_{\Gamma}} \gamma_i$; and $\alpha$ is set as
\begin{equation}
	\begin{aligned}
		\alpha :=  1 + f(y, X=x) + 
		\frac{d^2 (\psi^{l}+ \psi^{u})^2 }{4f(X=x) \psi^{l} \psi^{u}  },
	\end{aligned}
\end{equation}
with
\begin{equation}
\left[\begin{matrix}
	\psi^{l}, \psi^{u}
\end{matrix}\right] := \left[  \begin{matrix}
	\min\limits_{i \in [2d+1,3d]} \gamma_{i}^l,  \max\limits_{i \in [2d+1,3d]} \gamma_{i}^u 
\end{matrix}\right] \qquad\text{and}\qquad \gamma^{u}_i := \max\limits_{\bm\gamma \in IR_{\Gamma}} \gamma_i.
\label{identify_psi}
\end{equation}

The justification of such construction is given in lemma.~\ref{lemma_enclose}.

\quad \\
 \noindent $\large{\textbf{2) Bisection()}}$: This function is to partition an input simplex $S$ into two simplices $S_1, S_2$. Motivated by~\cite{rivara1984mesh}, we adopt the longest-edge (LE) bisection strategy to create the partition. The details are given in Algorithm.~\ref{alg_bisection}.

\begin{algorithm}[t]
\caption{Recursive procedure to split the simplicial partitions (\textbf{Bisection()}).}
\LinesNumbered
\KwIn{Simplex $S$ with vertices $\{S^0, \cdots, S^{4d}\}$.}
\KwOut{Two new simplices $S_1, S_2$.}

Set $S^{t_1}, S^{t_2}$ as the vertices incident to the longest edge of the simplex $S$, i.e., 
\[
\{t_1, t_2\} = \argmax_{\{a, b\} \in \{0, 1, \cdots, 4d\}} \|S^a - S^b\|_2;
\]

Construct $S_1, S_2$ based on the following two sets of vertices:
\[
\{S^{0}, \cdots, S^{t_1 - 1}, v, S^{t_1 + 1}, \cdots, S^{4d}\}, \quad \{S^{0}, \cdots, S^{t_2 - 1}, v, S^{t_2 + 1}, \cdots, S^{4d}\},
\]
where $v$ corresponds to the midpoint the longest edge. Return $S_1,  S_2$.\label{alg_bisection}
\end{algorithm}


\quad~\\
\noindent $\large{\textbf{3) Bounding()}}$: This function aims to derive a lower bound of $\underline{f_S(Y_x = y)}$ with the input $S$, which is also the most important component of this algorithm. Recall that $\underline{f_S(Y_x = y)}$ can be expressed as the solution of the optimization program~\eqref{re-formulation} with an additional constraint $\bm{\gamma} \in S$. Then using the derivations (\ref{DC_decomposition})-(\ref{re-formulation_linear_weaker}), it is straightforward that we can derive a lower bound of $\underline{f_S(Y_x = y)}$ via solving the following optimization problem\footnote{Notice that we should first verify $IR_{\Gamma} \cap S \neq \emptyset$. Otherwise we can directly set $\underline{\underline{f_S(Y_x = y)}} = +\infty$. }:
\begin{equation}
    \begin{aligned}
     \underline{\underline{f_{S_{}}(Y_{x} = y)}}= &\text{~min~} f(y,X=X) + C_1^{\text{tan}}(\bm{\gamma}) - C_2^{\text{sec}}(\bm{\gamma}) \\
&\text{~subject to:~} \bm{\phi} \in IR_{{\Phi}},\bm{\gamma} \in S_{}; \\
&~~~~~~~~~~~~~~~~~~D_{i1}^{\text{tan}}(\bm{\gamma}) - D_{i2}^{\text{sec}}(\bm{\gamma}) \leq 1, D_{i1}^{\text{sec}}(\bm{\gamma}) - D_{i2}^{\text{tan}}(\bm{\gamma}) \geq 1, i=1,2,...d,
\label{lemma_sub_linear}
    \end{aligned}
\end{equation}
where as demonstrated in lemma.~\ref{tan_sec_bounded} of Appendix.~\ref{app_main_result_1}, the functions $C_1^{\text{tan}}(\bm{\gamma})$, $C_2^{\text{sec}}(\bm{\gamma})$, $D_{i1}^{\text{tan}}(\bm{\gamma})$, $D_{i2}^{\text{sec}}(\bm{\gamma})$, $D_{i1}^{\text{sec}}(\bm{\gamma})$ and $D_{i2}^{\text{tan}}(\bm{\gamma})$ are constructed from $C_1(\bm{\gamma}), C_2(\bm{\gamma})$, $D_{i1}(\bm{\gamma}), D_{i2}(\bm{\gamma})$'s in~\eqref{dc_decomposition} based on \textbf{sec}ants and \textbf{tan}gents within the simplex $S$: 
\hspace{-2cm}\begin{equation} 
    \begin{aligned}
    \left[\begin{matrix}C_k^{\text{tan}}(\bm{\gamma})\\ D^{\text{tan}}_{ik}(\bm{\gamma}) \end{matrix} \right] :&= {\left[ \begin{matrix}
   C_k(\bm{\gamma_0}) \\ D_{ik}(\bm{\gamma_0})
   \end{matrix} \right] + \left[\begin{matrix}
   \frac{\partial C_k (\bm{\gamma})}{\partial \bm{\gamma}}\mid_{\bm{\gamma} = \bm{\gamma_0}} \\ \frac{\partial D_{ik} (\bm{\gamma})}{\partial \bm{\gamma}}\mid_{\bm{\gamma} = \bm{\gamma_0}}
   \end{matrix}\right](\bm{\gamma}-\bm{\gamma_0})},~k=1,2, \forall \bm{\gamma_0} \in S,\\
   \left[\begin{matrix}
   C_k^{\text{sec}}(\bm{\gamma}) \\ D^{\text{sec}}_{ik}(\bm{\gamma})
   \end{matrix}\right] :&= {\left[\begin{matrix}
   C_k(S^0),... C_k(S^{4d}) \\
   D_{ik}(S^0), ... D_{ik}(S^{4d})
   \end{matrix}\right] \left[\begin{matrix} S^0, ... , S^{4d} \\ 1,...,1 \end{matrix}\right]^{-1}\left[\begin{matrix}\bm{\gamma} \\ 1\end{matrix}\right]},~k=1,2. \\ \label{tan_sec}
   \end{aligned}
   \end{equation}
From above,~\eqref{lemma_sub_linear} is fully a linear programming problem, which can be solved by a wide variety of solutions, e.g. simplex algorithm~\cite{klee1972good}, interior algorithm~\cite{kojima1989primal, nesterov1994interior}.

~\\ \quad \\
\noindent $\large{\textbf{4) Global\_error()}}$: This function is to terminate PI-SFP via estimating the order of the error with respect to $n$. Recall that in \textbf{Step~5} of Algorithm.~1, we always select the $\tilde{S}_k$ with the lowest $\underline{\underline{f_S(Y_x = y)}}$ in the $k-$th iteration. This strategy guarantees (see Appendix.~\ref{app_main_result_1} for more details)
\begin{equation}
    \begin{aligned}
        \underline{\underline{f_{\tilde{S}_k}(Y_x = y)}} \leq \min_{S \in \cS_k}\underline{f_S(Y_x = y)} = \underline{f(Y_x = y)},
    \end{aligned}\label{bound_each_iteration}
\end{equation}
i.e., all the $\underline{\underline{f_{\tilde{S}_k}(Y_x = y)}}$'s are lower bounds of $\underline{f(Y_x = y)}$, and thus $\underline{f^n_{opt}(Y_x = y)} \leq \underline{f(Y_x = y)}$. From this, we further have that, in the $n$-th iteration, for any $k \in \{0, \cdots, n\}$,
\begin{equation}
    \begin{aligned}
       0 \leq \underline{f(Y_x = y)} - \underline{f^n_{opt}(Y_x = y)}
       \leq \min_{S \in \cS_k} \underline{f_S(Y_x = y)} - \underline{\underline{f_{\tilde{S}_k}(Y_x = y)}} \leq 
        {\underline{f^{{}}_{\tilde{S}_k}(Y_x = y)}} - \underline{\underline{f_{\tilde{S}_k}(Y_x = y)}}. 
    \end{aligned}\label{shrinking_diameter}
\end{equation}
Also see Appendix.~\ref{app_main_result_1} for details. This allows us to calculate an error bound via targeting  
\begin{equation}
\min_{0 \leq k \leq n}\left\{\underline{f^{{}}_{\tilde{S}_k}(Y_x = y)} - \underline{\underline{f_{\tilde{S}_k}(Y_x = y)}}\right\}.\label{eq:tildesbnd}
\end{equation}
Since the bound of $\underline{f^{{}}_{\tilde{S}_k}(Y_x = y)} - \underline{\underline{f_{\tilde{S}_k}(Y_x = y)}}$ is dominated by the diameter of the simplex $\tilde{S}_k$, i.e., $dia(\tilde{S}_k)$, we aim to get an order of \eqref{eq:tildesbnd} based on the order of the smallest $dia(\tilde{S}_k)$ with respect to $n$. As shown in Eqn.~(\ref{second_bounded_result}) in Appendix.~\ref{app_main_result_1}, this order is controlled by the length $L_n$ of the longest nested subsequence of $\{\tilde{S}_k\}_{k=0}^{n}$, which gives us Algorithm.~3.

\begin{algorithm}[t]
\caption{Procedure to estimate the current convergence (\textbf{Global\textbf{\_error}()}).}
\LinesNumbered
\KwIn{Collections of simplex partitions in each iteration till $n-$th iteration: $\tilde{S}_0, \cdots, \tilde{S}_n$.}
\KwOut{An estimate of the global error.}

Let $\{\tilde{S}_{i_k}\}_{k=1}^{L_n}$ be the (longest) subsequence of $\{\tilde{S}_k\}_{k=0}^n$ such that each $\tilde{S}_{i_{j + 1}}$ is partitioned from $\tilde{S}_{i_j}$ for $j=0, 1, \cdots, L_n-1$, where $L_n$ is the length of this subsequence;

Return the global error estimate $(\frac{\sqrt{3}}{2})^{\lfloor \frac{L_n}{4d}\rfloor}$.
\end{algorithm}

\section{Theoretical analysis}

This section investigates the theoretical property of PI-SFP. We first explore the general converging rate of PI-SFP with respect to $L_n$ (Theorem.~\ref{convergence_theorem}). Then we show that PI-SFP can be extended from calculating $\underline{f(Y_x = y)}$ to the general ACE case.

For preparation, we should ensure that $IR_{\Gamma}$ is bounded, so that $dia(S_0)<+\infty$, and then we can split it into sufficient small partitions for further estimation. For this goal, we introduce the following positive definite assumption:

\begin{assumption}{\textbf{(boundedness)}}\label{positive definite assumption}
$\mathscr{P}$ is a set of $P(\bm{W} \mid \bm{U})$ guaranteeing each compatible solution $P(\bm{U}, \bm{X}=x)$ to be positive definite. Namely, $\exists \delta > 0$, such that $\forall \bm{\phi}=(\bm{\theta}, \bm{\psi}, \bm{\omega}) \in IR_{\bm{\phi}}$, we have $\bm{\psi} \geq \delta * \bm{1}_{1*d}>\bm{0}_{1*d}$.
\end{assumption}

\begin{remark}
Note that it is a fairly broad and reasonable assumption in practice, just in order to ensure that the denominator in \eqref{formulation} is not too small to facilitate the calculation. Under this assumption, we have $\psi_i^o < \frac{1}{\delta}$ in \eqref{re-formulation} and $\psi^l >\delta$ in \eqref{identify_psi}. Hence we have $sup_{\gamma \in IR_{\Gamma}}\|\bm{\gamma}\|_{+\infty} < +\infty$ and $dia(S_0)<+\infty$ respectively. 

Moreover, this assumption is proposed only to facilitate further elaboration of the fundamental properties of PI-SFP. It does not strictly limit the scope of its application. Indeed, we can generalise PI-SFP so that it is applicable to situations where this assumption does not hold. We refer the readers to Appendix.~\ref{discussion_ass} for more details.

\end{remark}

On this basis, we formally collate the previous analysis as our first main result:

\begin{theorem}

    Under Ass.~\ref{ass_partial_bounded}--\ref{positive definite assumption}, PI-SFP concentrates around the target value $\underline{f(Y_x = y)}$ at the $O((\frac{3}{4})^{\lfloor\frac{L_n}{4d}}\rfloor)$ rate. Specifically, 
    \begin{equation}
        \begin{aligned}
            &\mid \underline{f^{n}_{opt}(Y_x = y)} - \underline{f(Y_x = y)} \mid  \leq  A   ( \frac{3}{4})^{\lfloor \frac{L_n}{4d} \rfloor}  dia(S_0)^2,\\
             \end{aligned}
    \end{equation}
   $A = \max\limits_{\bm{\gamma} \in S_{0}} \frac{2(\sqrt{2}+1)\sqrt{d}}{\delta} \|\frac{\partial{(C_1(\bm{\gamma}) - C_2(\bm{\gamma}))}}{\partial{\bm{\gamma}}}\|  + \max\limits_{\bm{\gamma} \in S_{0}} \|\frac{\partial^2 C_1(\bm{\gamma)}}{\partial \bm{\gamma}^2}\|_F +  \frac{1}{2} \max\limits_{\bm{\gamma} \in S_{0}} \|\frac{\partial^2 C_2(\bm{\gamma)}}{\partial \bm{\gamma}^2} \|_F < +\infty$. Here $\|\cdot \|$ denotes the Euclidean norm, and $\|\cdot \|_F$ denotes the Frobenius norm. $L_n\in [\lfloor log(n)\rfloor+1, n]$ is the length of the longest nested sequence till $n-$th iteration. Moreover, $\lim\limits_{n\rightarrow +\infty} \underline{f_{opt}^{n}(Y_x = y)} = \underline{f(Y_x = y)}$.
\label{convergence_theorem}
\end{theorem}
Theorem.~\ref{convergence_theorem} states that PI-SFP converges to $\underline{f(Y_x = y)}$ with the growing length of the longest nested sequence, and will approach it in the infinite case. We relegate the proof to Appendix.~\ref{app_main_result_1} and reserve a brief summary. First, $\underline{f(Y_x = y)}$ is equal to $\underline{f_{S_{0}}(Y_x = y)}$ via
constructing an original enclosure $S_0$ in \eqref{initialization_S00}. Second, $\underline{f_{S_{0}}(Y_x = y)}$ is substituted with $\min_{S\in \cS_k}\underline{f_{S}(Y_x = y)}$ in the $k$-th iteration by bisection. Third, each $\underline{f_{S}(Y_x = y)}$ is lower bounded by \eqref{lemma_sub_linear}, namely we have $\forall S \in \cS_{k}, {\underline{f_S(Y_x = y)}} \geq \underline{\underline{f_S(Y_x = y)}}$. Finally,
$\tilde{S}_k$ with the lowest bound $\min\limits_{S\in \cS_k}\underline{\underline{f_S(Y_x = y)}}$ is gathered as $\{\tilde{S}_k\}_{k=0}^n$ in order to formulate $\underline{f^n_{opt}(Y_x = y)}$ (see \textbf{Step~10} in Algorithm.~1). The asymptotic error can be bounded by \eqref{shrinking_diameter}-\eqref{eq:tildesbnd}. In conclusion, these four steps correspond to the four functions in the above section in order.

\begin{remark}{(Discussion on the growing rate of $L_n$ w.r.t. $n$)} The worst case is $L_n = \lfloor log(n) \rfloor +1$. In this scenario, PI-SFP will be equivalent to the method of exhaustion which shows a rather slow polynomial convergence as $O(n^{-\alpha})$ by Theorem.~\ref{convergence_theorem}, where $\alpha = \frac{1}{4d}log( \frac{{2}}{\sqrt{3}})$. However, empirically, this case is rare. In the simulation part, the convergence rate is faster than $O(n^{-\alpha})$. It can also be enhanced by some pruning strategy which will be discussed in the Section.~\ref{section_dis_ext}.

Despite this empirical observation, it is well beyond the scope of this paper to theoretically estimate $L_n$ w.r.t $n$. During iteration, each optimal solution (converging point) may be covered by increasing number of nested sequences\footnote{To guarantee each converging point is covered by finite partitions, we should resort to the regularity condition of simplices (identified in~\cite{ciarlet2002finite}). However, whether LE bisection can promise a family of regular partitions is still an open problem (\cite{korotov2016longest}) to be solved. }. These sequences possess different lengths and are difficult to estimate. More seriously, the number of optimal solutions is not necessarily finite either, namely we do not guarantee $|\bm{\Phi}_{opt}| < +\infty$.

\end{remark}

\begin{remark}{(Extension of bounding the ACE)} Taking advantage of PI-SFP, we can further achieve the valid bound of $ACE_{\bm{X} \rightarrow \bm{Y}}$. The above PI-SFP algorithm is to seek $\underline{f(Y_x = y)}$ when $X=x$ is fixed. We do further extension to consider all values of $\bm{X}$ simultaneously. In this sense, we reorganize (\ref{eqn_basic_bound}) to bound ACE (Definition.~\ref{def_ACE}) as follows:
\begin{equation}
    \begin{aligned}
    &\text{min~} \sum_x \pi(x) \int_{Y^L}^{Y^U} yf(y,X=x) dy +  \sum_x \pi(x) \sum_{i=1}^{d} \frac{ \left(\int_{Y^L}^{Y^U}yf(y,u_i,X=x)dy\right) f(u_i, X\neq x)}{f(u_i,X=x)}\\  &\text{subject to: $f(y,\bm{U},\bm{W}, \bm{X}) \in \mathcal{F}$}.
    \label{eqn_basic_bound_ace}
    \end{aligned}
\end{equation}
Using the same strategy as in Section.~\ref{framework}-\ref{section_algorithm}, we can achieve the valid bound of ${ ACE_{\bm{X} \rightarrow \bm{Y}}}$ in~(\ref{tight_bound_ace}). Due to the space limitation, we summarize it in the following corollary and deduce these details in Appendix.~\ref{corollary_ACE_appendix}.
\end{remark}

\section{Simulations}
In this section, we do simulations to illustrate the effectiveness of PI-SFP.

We focus on fig.~\ref{Fig.sub.1} and generalise the case \eqref{tight_example} in the introduction part. We consider an interesting and general situation called 'information leakage', namely the information of $\bm{U}$ is regularly retained by $\bm{W}$ but suffers loss in transmission. Formally, we claim 
\begin{equation}
    \begin{aligned}
    P(W=w_i \mid U=u_i)\geq 1-\epsilon, \epsilon \in [0,0.5]. 
    \end{aligned}
\end{equation}

To make the experiment simple and representative, we consider the binary cases of $\bm{W},\bm{U},\bm{X}$. On this basis, the construction is as follows:
\begin{equation}
    \begin{aligned}
         \left[\begin{matrix}
\overline{P(\bm{W}\mid \bm{U})} \\ \underline{P(\bm{W}\mid \bm{U})}
\end{matrix}\right] = \left[ \begin{matrix}
(1-\epsilon)\bm{I_{2*2}} + \epsilon \bm{J_{2*2}} \\ (1-\epsilon) \bm{I_{2*2}} 
\end{matrix}\right], \epsilon \in [0,0.5].
    \end{aligned} \label{initial_observe}
\end{equation}

The construction of $f(\bm{Y},\bm{W},\bm{X})$ still follows \eqref{tight_example}. In order to avoid the ill-conditioned case for PYTHON $3.8.5$, we make a rather broad restriction that elements of $P(\bm{U}, X=x)$ are at least $1e^{-2}$ in all cases. Moreover, we set the iteration number as $1000$.
\begin{figure}[h]
    \centering
    \includegraphics[width = 6cm, height = 4cm]{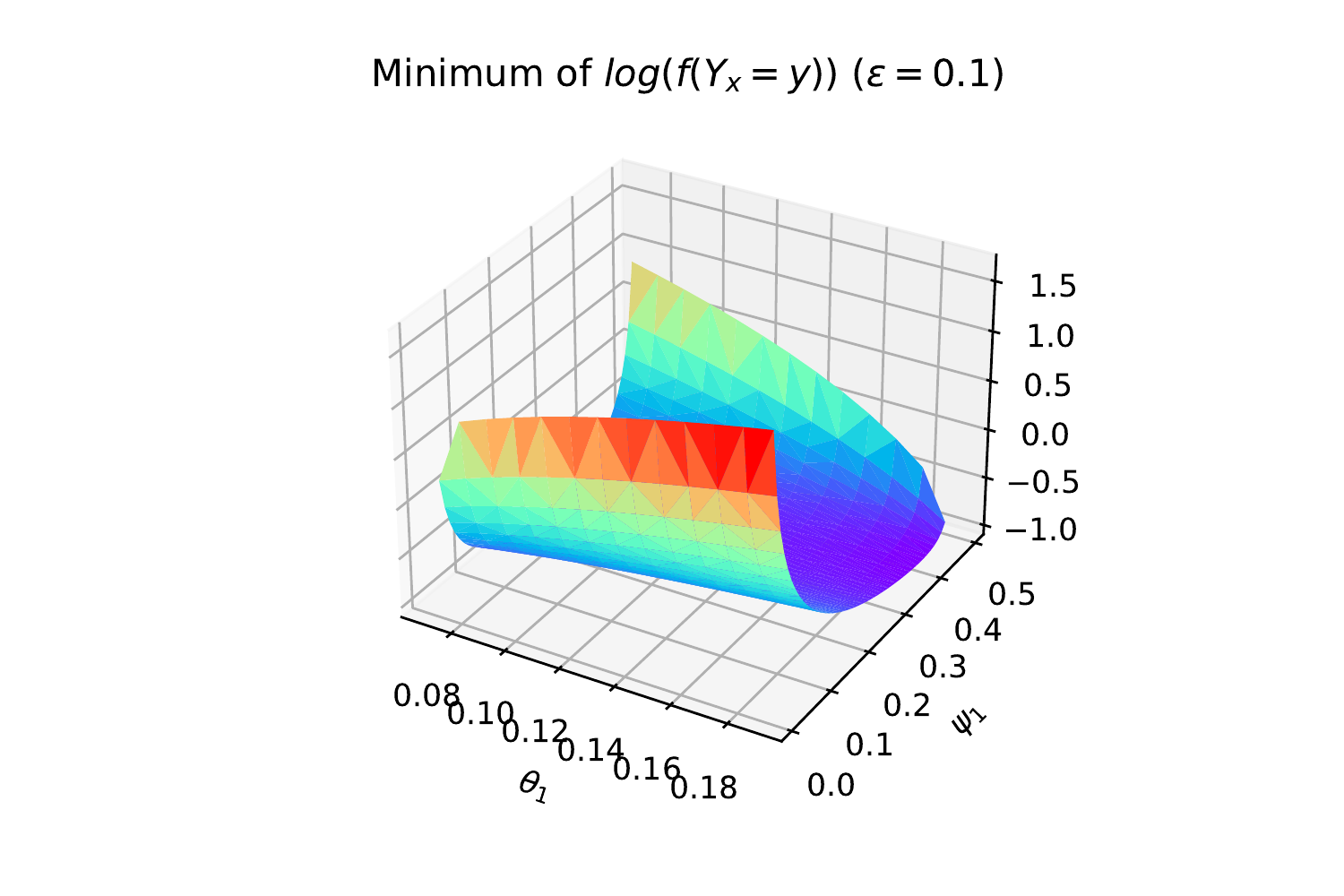}
    \includegraphics[width = 6cm, height = 4cm]{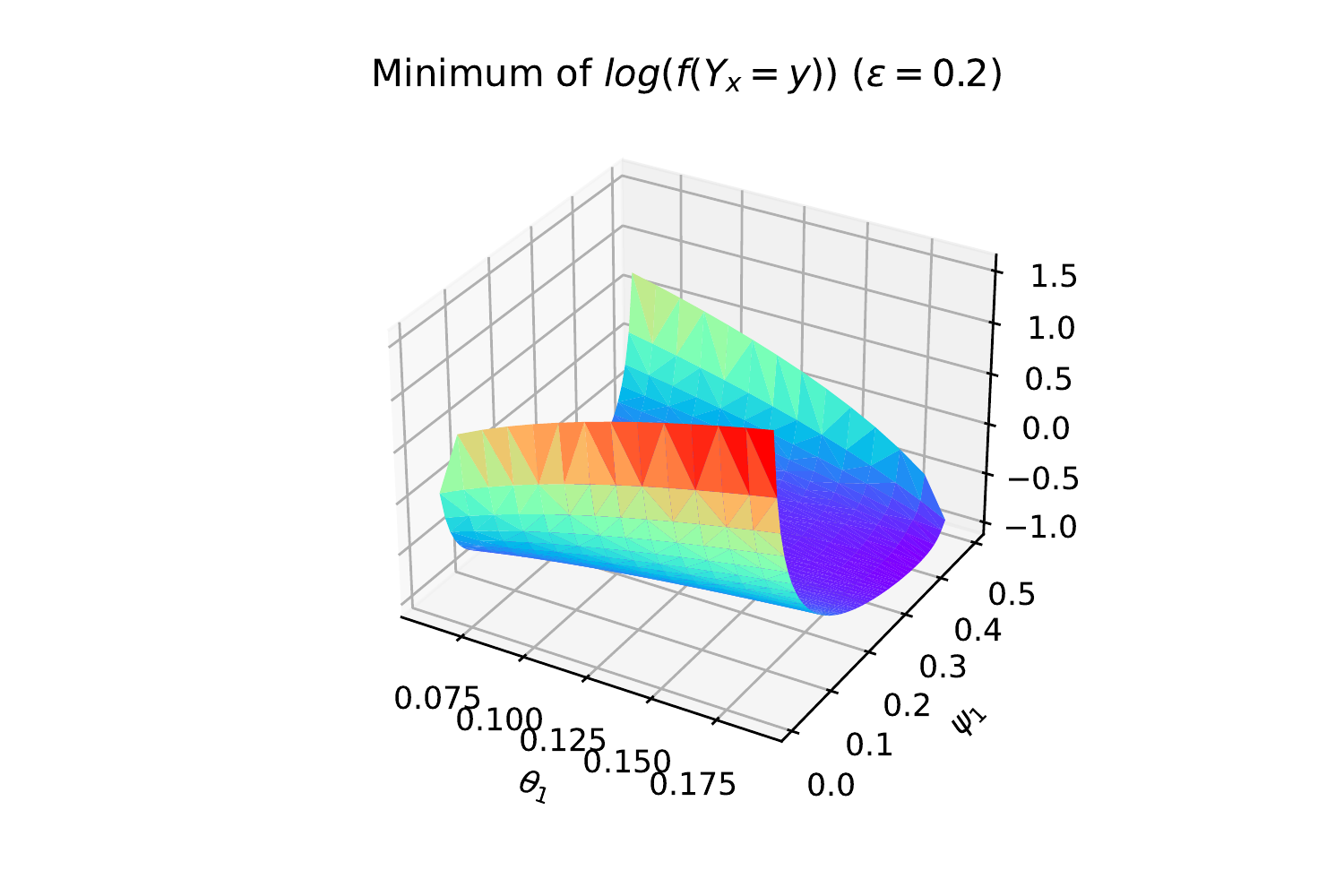}
    \includegraphics[width = 6cm, height = 4cm]{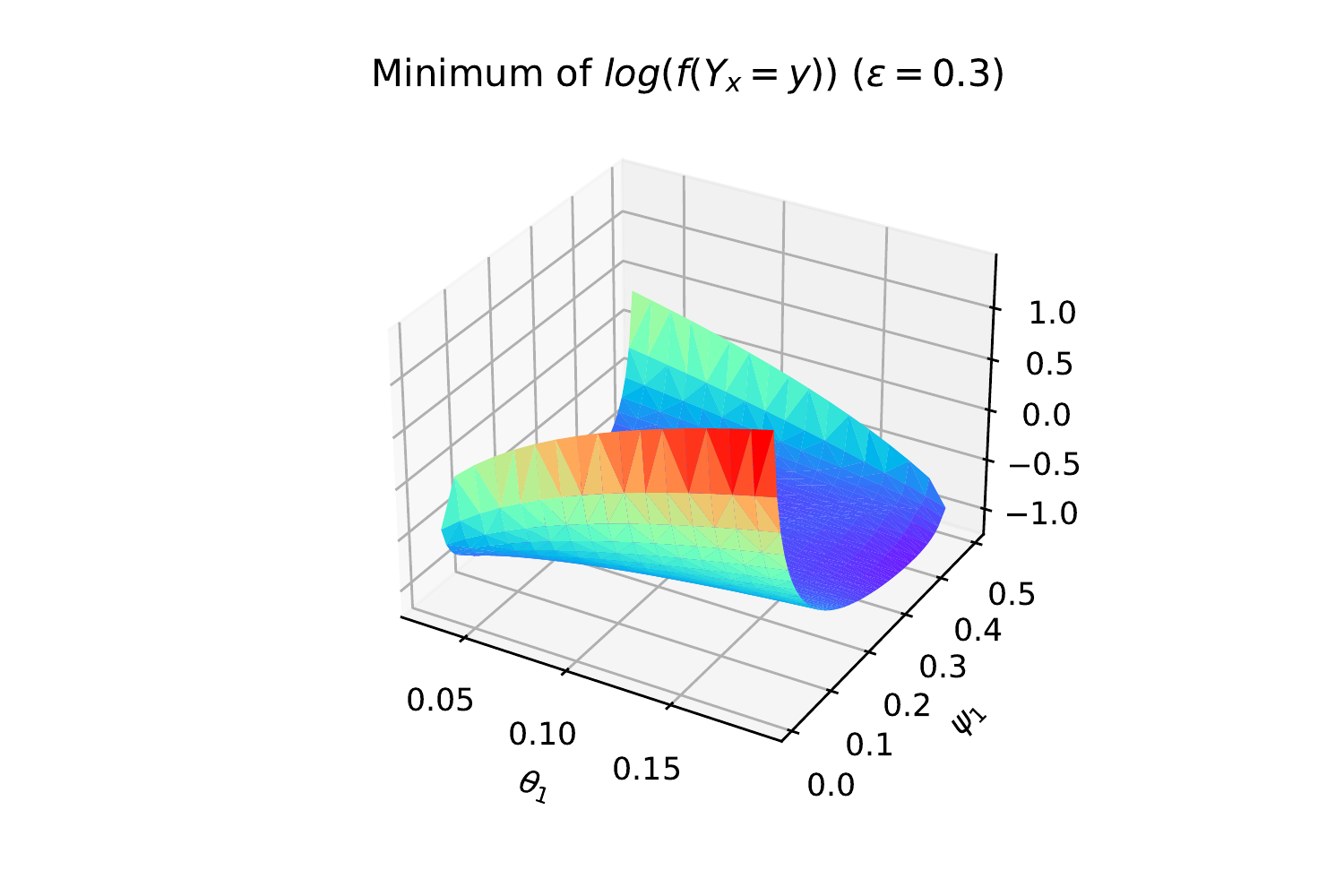}
    \includegraphics[width = 6cm, height = 4cm]{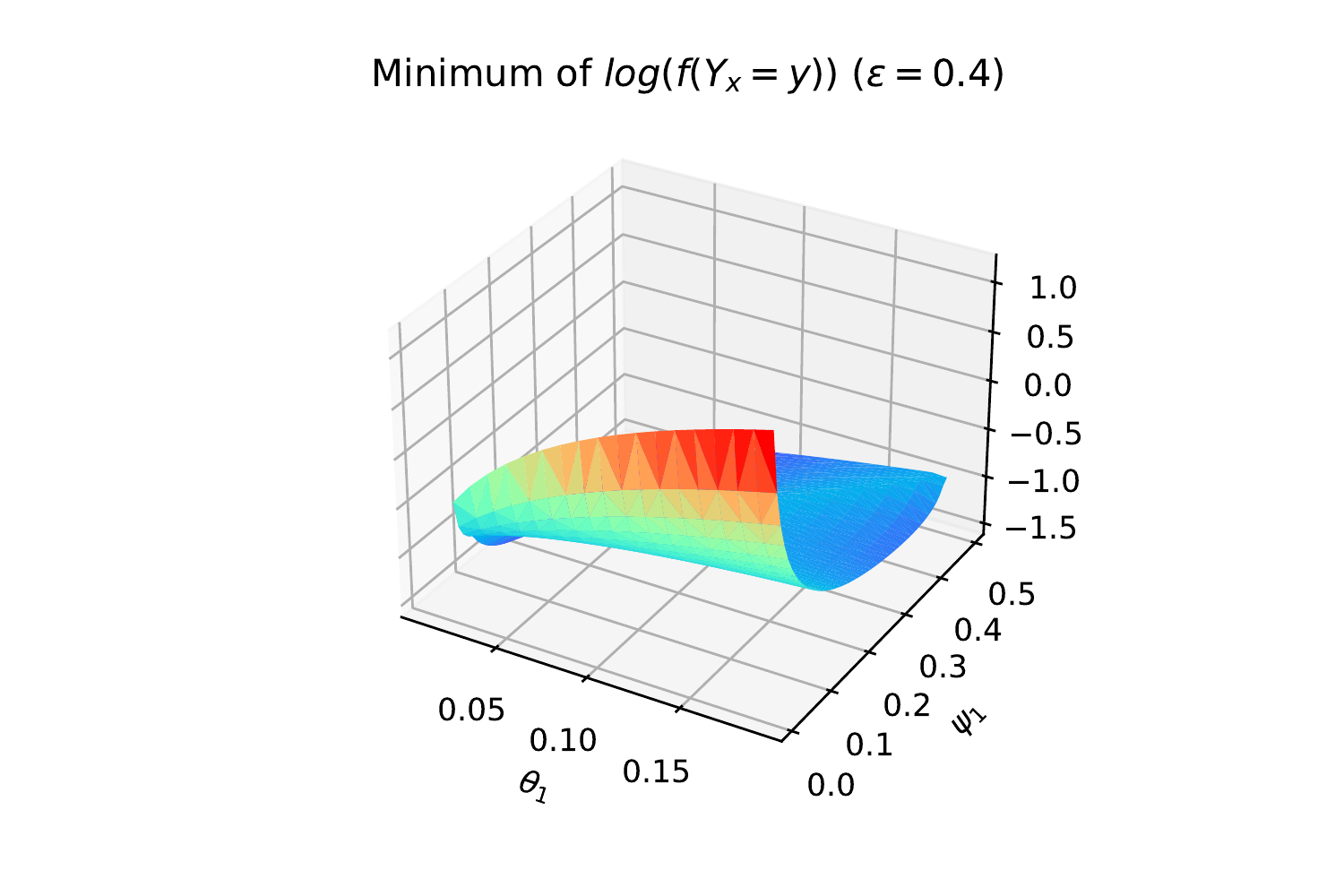}
    \caption{We search the minimum of $f(Y_x = y)$ in the binary case, by fixing each possible $\theta_1, \theta_2$ and $\psi_1, \psi_2$ in \eqref{re-formulation}. In this way \eqref{re-formulation} degenerates to be a set of linear programming problems. We can detect that the number of optimal solutions are finite. Moreover, we will show that PI-SFP converges to these optimal solutions as shown in tab.~\ref{simulation_result}.}
    \label{minimum_graph}
\end{figure}

The simulation result is in Tab.~\ref{simulation_result} and Fig.~\ref{PI-SFP_0.5}. We can find that PI-SFP successfully find the optimal solutions and optimal values in a fast convergence rate. Although the initial error increases as $\epsilon$ increases, they are always kept under control by the theoretical error, which is guaranteed by Theorem.~\ref{convergence_theorem}. Notice that the converging error performs an fast decrease in practice..

\begin{table}[h]
	\caption{Simulation results. The middle column denotes the optimal solution with iteration $1000$. The approximation of $\underline{f(Y_x = y)}$ decreases monotonically with the increasing $\epsilon$, since the feasible region of latent variables $\Phi$ is gradually enlarged by \eqref{initial_observe}.}
	\begin{tabular}{c| ccc ccc  | c}
		\hline
		\multirow{2}{*}{$\epsilon$}    & \multicolumn{6}{c|}{$\Phi$} & \multirow{2}{*}{$\underline {f(Y_x = y)} $ }      \\
		\cline{2-7} 
		& \makecell[c]{$\theta_1$ }        & \makecell[c]{$\theta_2$ }   &\makecell[c]{$\psi_1$ }  & \makecell[c]{$\psi_2$ }        & \makecell[c]{$\omega_1$ }   &\makecell[c]{$\omega_2$ }\\
		\hline
		0.1    & 0.067 &0.133 & 0.261 & 0.239 &0.333 & 0.167 & 0.370\\
		0.2    & 0.050 &0.150 & 0.262 & 0.238 &0.375 & 0.125 & 0.350 \\
		0.3    & 0.029 &0.171 & 0.264 & 0.236 &0.429 & 0.072 & 0.298 \\
		{0.4}    & 0.001 & 0.199 & 0.310 & 0.190 &0.500 & 0.000& 0.200 \\
	\end{tabular}
	\label{simulation_result}
\end{table}

\begin{figure}[h]
    \centering
    \includegraphics[width = 6cm, height = 4cm]{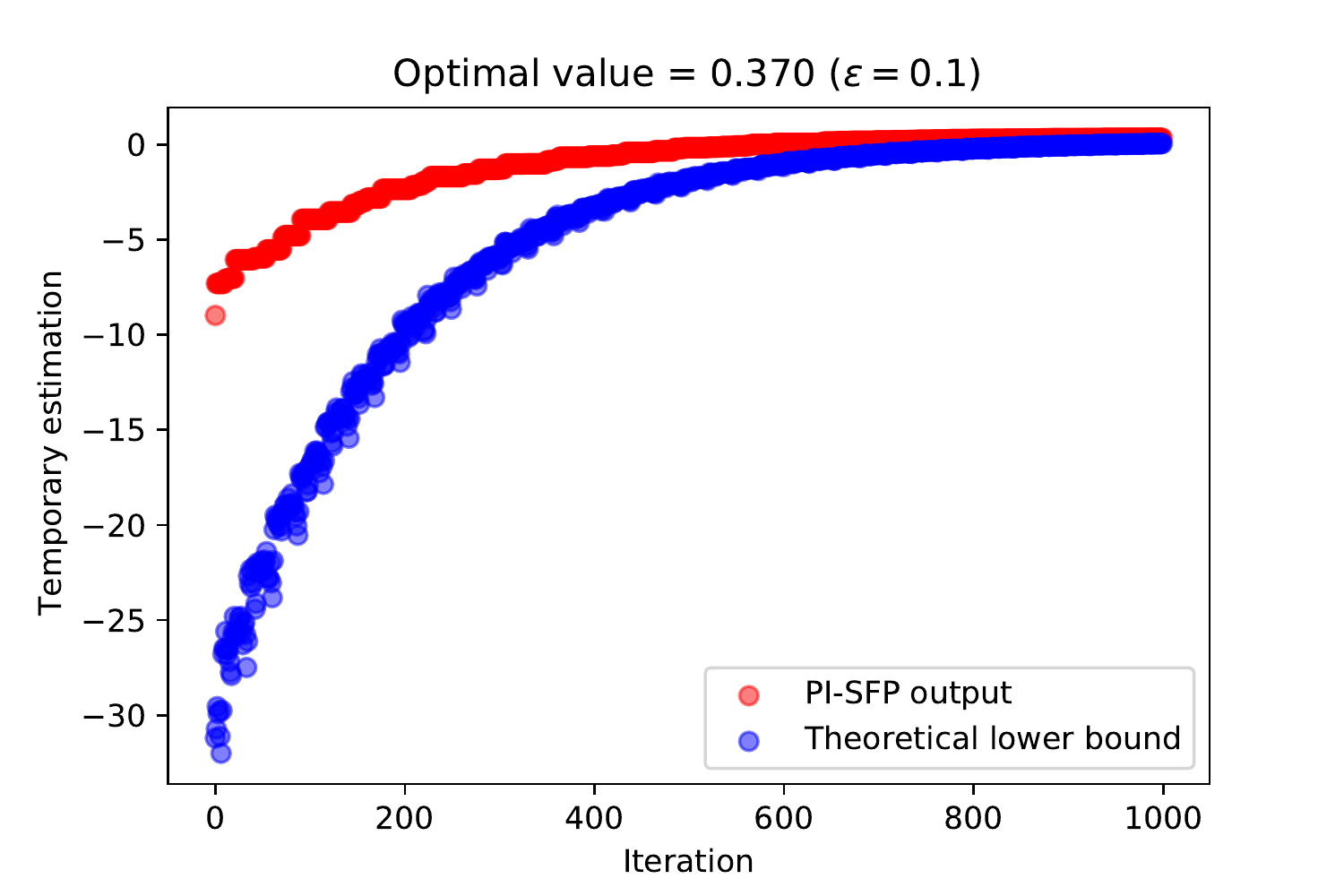}
    \includegraphics[width = 6cm, height = 4cm]{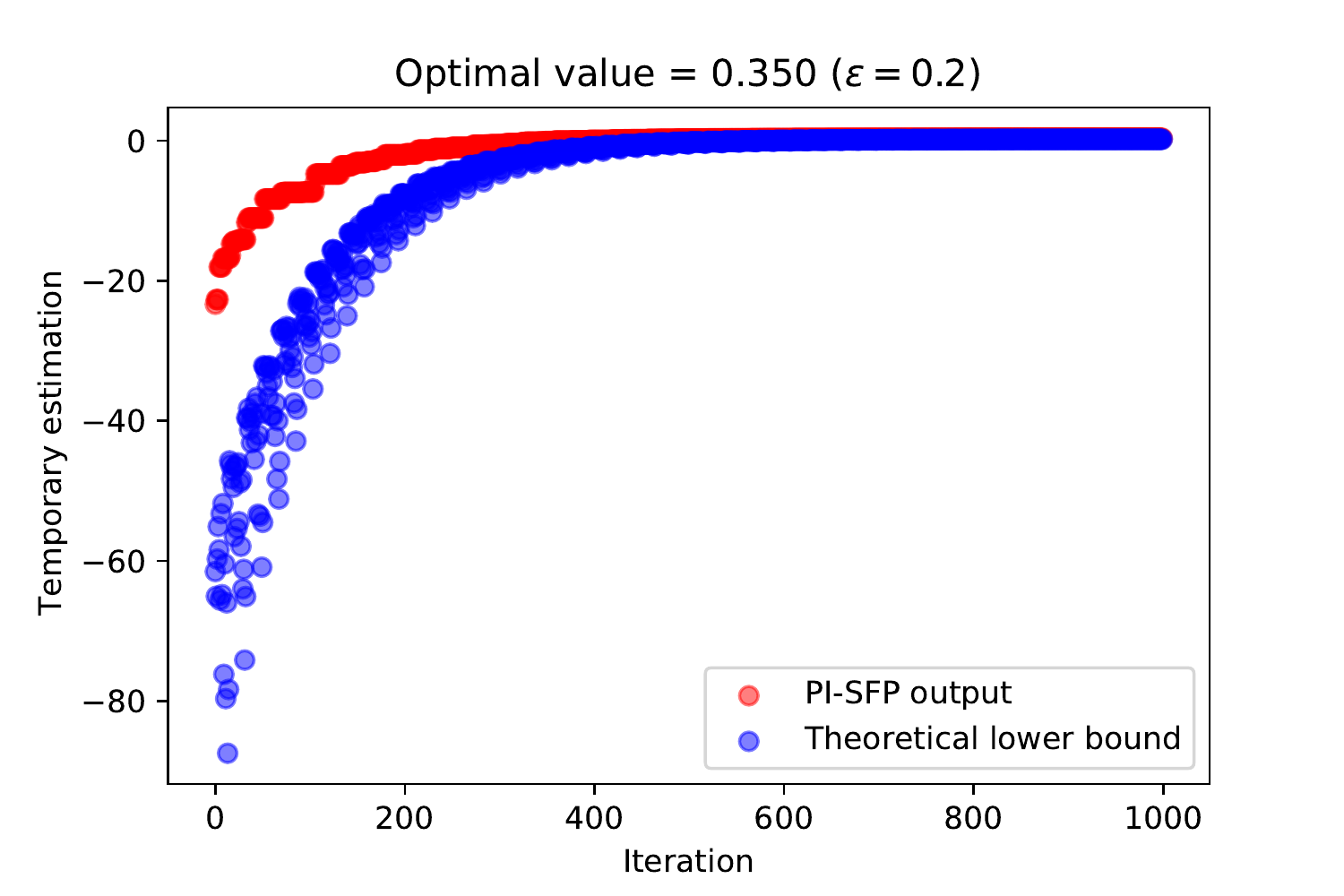}
    \includegraphics[width = 6cm, height = 4cm]{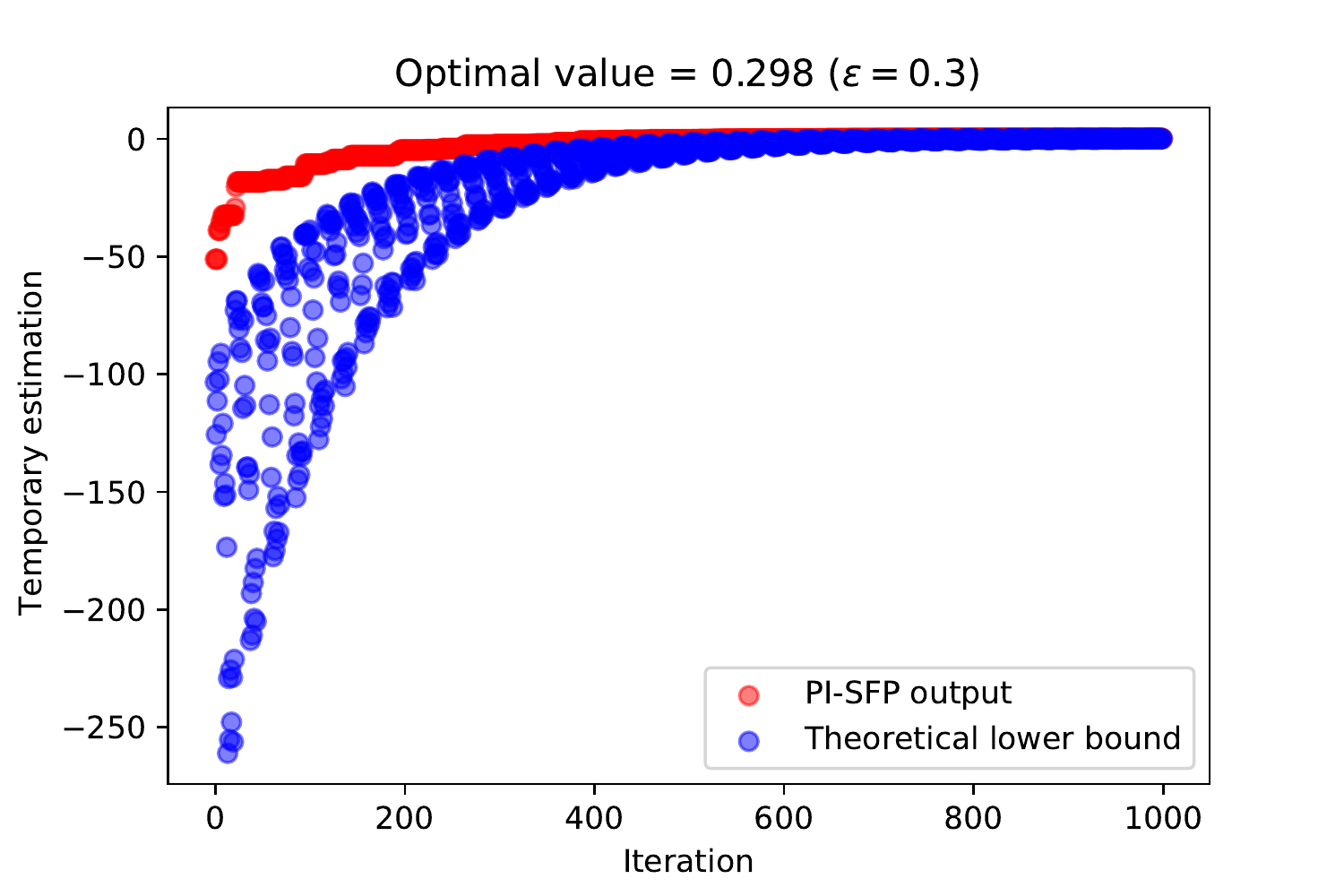}
    \includegraphics[width = 6cm, height = 4cm]{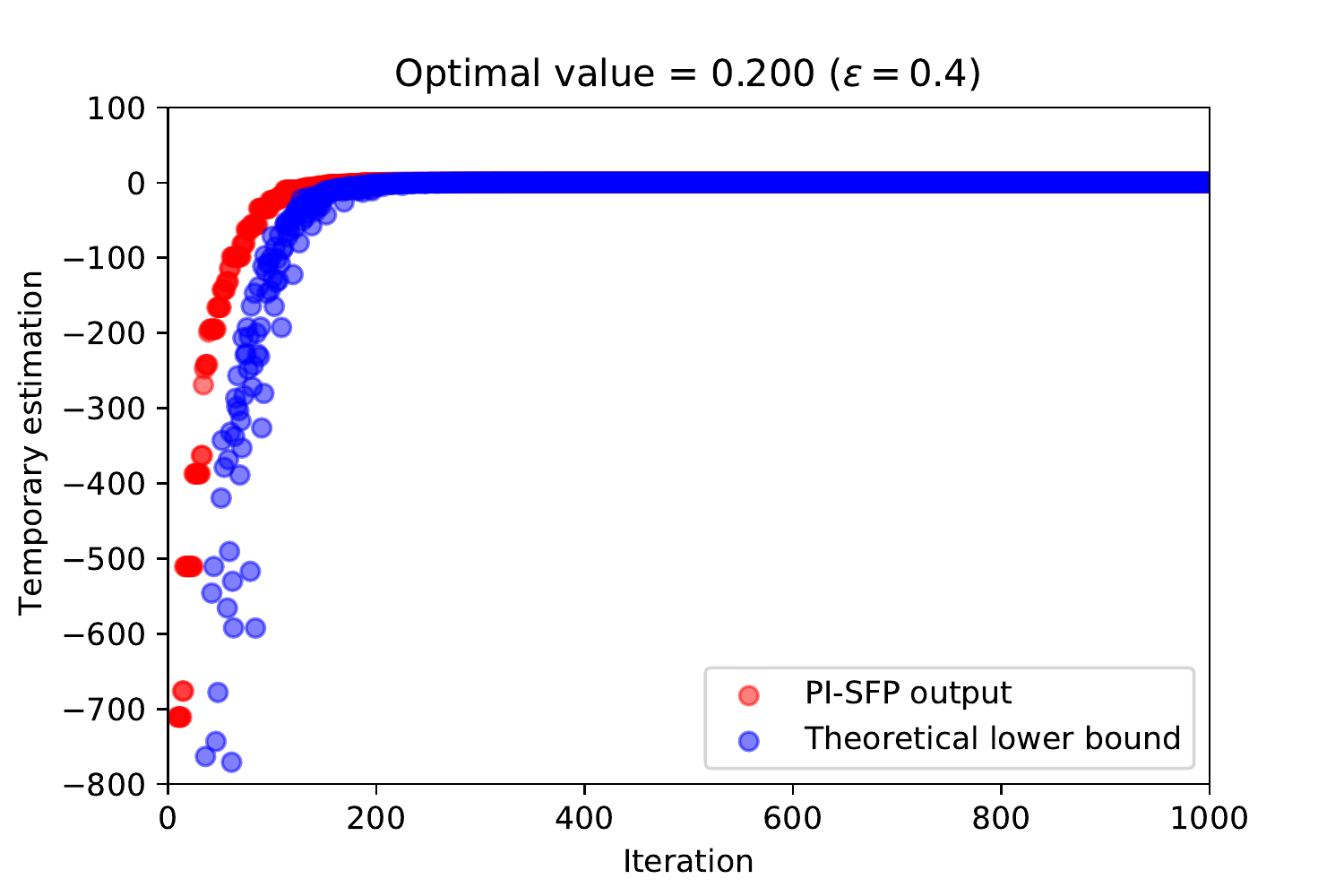}
    \caption{Results of PI-SFP. PI-SFP (red) converges to the optimal value of $\underline{f(Y_x = y)}$ when $\epsilon$ changes from $0.1$ to $0.4$ with a theoretical convergence rate guarantee (blue).}
    \label{PI-SFP_0.5}
\end{figure}

\section{Further discussions and extensions}\label{section_dis_ext}
In this section, we provide further discussions and extensions on assumptions, algorithm, graphical structure, and moreover, continuous confoundings.

\subsection{Discussions on assumptions}\label{discuss_ass}
We focus on the core partial observability assumption (Ass.~\ref{ass_partial_bounded}), and analyze its necessity, generalisability, and verifiability respectively. First, we show that the partial observability assumption is more helpful to achieve a better bound of $f(Y_x = y)$, instead of reversibility, although the latter one serves as an important hypothesis for calculation in previous papers. Second, we show that our Ass.~\ref{ass_partial_bounded} is weaker and more general, addressing kinds of cases that previous assumptions fail to work. At last, we illustrate that Ass.~\ref{ass_partial_bounded} is verifiable.

\noindent \textbf{Necessity} In this part, we investigate the necessity of the partial observability assumption. We consider the case when there is no knowledge on $P(\bm{W} \mid \bm{U})$, i.e., without Ass.~(\ref{ass_partial_bounded}). In fact, even in this case, it is still hard to achieve the tight lower bound of $f(Y_x = y)$ as we illustrated in the introduction. Hence, analogously to Section 3.2, we also resort to a relatively broader feasible region of $f(y,\bm{W},\bm{U},\bm{X})$ than $\mathcal{F}$, within which we show the tight lower bound will degenerate to be trivial ($f(y, X=x)$) without Ass.~(\ref{ass_partial_bounded}). 

Notice that a precise portrayal of the dynamic relationship between partial observability and tight bound is not realistic currently. However, the following lemma can at least serve as an initial exploration in single proxy control, to illustrate the role of Ass.~(\ref{ass_partial_bounded}).

\begin{lemma}
Assume that $[\underline{P(\bm{W} \mid \bm{U})}, \overline{P(\bm{W} \mid \bm{U})}] = [\bm{0}_{\dim(\bm{W}) * d}, \bm{1}_{\dim(\bm{W}) * d} ]$, and $f(\bm{U},X=x) > \bm{0}$. We consider the whole set of $f(y,\bm{W},\bm{U},\bm{X})$ which is within $\mathcal{\widetilde{F}}$ and is additionally compatible with two observed distributions $f(\bm{W}, X\neq x)> \bm{0}_{dim(\bm{W}) * 1}$, $f(y, \bm{W}, X= x)>\bm{0}_{dim(\bm{W}) * 1}$ by an unknown $P(\bm{W}\mid \bm{U})$. Then
\begin{itemize}
\item The tight lower bound of ${f(Y_x = y)}$ is $f(y, X=x)$.
\item If $P(\bm{W} \mid \bm{U})$ is restricted to be left-reversible and $ f(\bm{W} \mid X\neq x) \neq  f(\bm{W} \mid X=x,y)$, then the tight lower bound of ${f(Y_x = y)}$ is still $f(y, X=x)$.
\item If $P(\bm{W} \mid \bm{U})$ is restricted to be left-reversible and $ f(\bm{W} \mid X\neq x) =  f(\bm{W} \mid X=x,y)$, then ${f(Y_x = y)}$ is lower bounded by $f(y\mid X=x)$.
\end{itemize}
\label{lemma_no_assumption}
\end{lemma}

The proof is in Appendix.~\ref{pro_lemma_no_assumption}. This lemma extends the well-known inequality $f(Y_x = y) \geq f(y,X=x)$ \cite{reason:Pearl09a} to single proxy control. $f(y, \bm{W}, X= x)$, $f(\bm{W}, X\neq x)$ are to control $f(y, \bm{U}, X= x)$, $f(\bm{U}, X\neq x)$ respectively, where we aim to construct $f(y, \bm{U}, X= x)\circ f(\bm{U}, X\neq x) = \bm{0}_{d*1}$\footnote{$\circ$ denotes the Hadamard product.}. 

According to the first two cases in lemma.~\ref{lemma_no_assumption}, we can claim the reversibility can not directly help produce a non-trivial tight lower bound in all cases. In fact, reversibility is only for ease of matrix calculation. Moreover, we could ignore the third scenario in most cases, since this equality is fairly rare in practice and without theoretical guarantee. In conclusion, in order to enhance the tight bound, 1) the knowledge on $P(\bm{W}\mid \bm{U})$ is necessary, and 2) reversibility assumption on $P(\bm{W} \mid \bm{U})$ may not be necessary. Removing or weakening the reversibility is rational and worthy of being explored.

\noindent \textbf{Generalisability} Our assumption relaxes the assumptions of the matrix adjustment method~\cite{rothman2008modern,kuroki2014measurement} (seen as tab.~\ref{table_literature}), where authors assumed that the transition operator $P(\bm{W}\mid \bm{U})$ is totally explicit and reversible. Due to this relaxation, our Ass.~\ref{ass_partial_bounded} covers a few new problems in practice. That is, if the reversibility and total observability do not both exist (see our simulation part), then the recent literature on single-proxy control fail to work, just except for our PI-SFP. On the other hand, we also do not need an auxiliary $\bm{Z}$ to adopt double negative control such as \cite{miao2018identifying, cui2020semiparametric, tchetgen2020introduction, deaner2018proxy,shi2020multiply,singh2020kernel, nagasawa2018identification,kallus2021causal}. This helps us to get rid of a large number of assumptions such as completeness, and bridge function in the previous literature.

\noindent \textbf{Verifiability} The previous work has suggested the feasibility of Ass.~\ref{ass_partial_bounded}. \cite{kuroki2014measurement} claimed that if we want to find the bounds $\underline{P(\bm{W}\mid \bm{U})}$ and $\overline{P(\bm{W}\mid \bm{U})}$ a priori, the Bayesian strategy \cite{greenland2005multiple} and some re-calibration methods~\cite{rothman2008modern, selen1986adjusting} can be resorted. To show this, they provide their estimation of $P(\bm{W}\mid \bm{U})$ in detail in their ``Head Start Program''. 

\subsection{Discussion on algorithms}

In this section, two more optimization methods are discussed to illustrate the diversity of solving \eqref{re-formulation}, and then we state that PI-SFP performs better than them. In addition, we propose a prunning strategy for acceleration supported by a local optimization method.

\noindent\textbf{Algorithm comparison} In \cite{shen2017solving}, the author derived an $\epsilon-$ approximation method, which can be adopted and its result falls in $[\underline{ f(Y_x = y)}, (1+\epsilon)\underline{f(Y_x = y)}]$. However, this algorithm maintains an exponential time complexity for the dimension $dim(\bm{U})$, making it impossible to operate properly in high-dimensional confoundings. Moreover, \cite{le2014dc} designed an iterative algorithm to search the karush-kuhn-tucker~(KKT) point of difference-in-convex (DC) problem, which can be applied to~\eqref{re-formulation}. However, KKT theory can not guarantee the global optimality, compared with our PI-SFP.

\noindent\textbf{Algorithm acceleration} In order to accelerate PI-SFP, we aim to set a sufficient criteria to evaluate whether the current partition contains the optimal solution. If not, we can delete the branch online and narrow our search. By this motivation, we propose an auxiliary algorithm to search the local minimum of $f(Y_x = y)$, which serves as an upper-bound of $\underline{f(Y_x = y)}$. Specifically, in sub simplex $S$, if the optimal value $\underline{\underline{f_S(Y_x = y)}}$ is even larger than the local minimum, then it will be larger than $\underline{f(Y_x = y)}$. Hence we claim this partition must not include the optimal solutions, and this partition can be deleted forever. This auxiliary algorithm is by local optimization, whose details are shown in Appendix.~\ref{pre_train}.

\subsection{Discussion on graphical structure}

\noindent \textbf{Fig.~\ref{Fig.sub.1}} Our algorithm PI-SFP mainly focuses on Fig.~\ref{Fig.sub.1}. Moreover, if $\bm{W} \rightarrow \bm{Y}$ is added, the optimization problem will be transferred as follows under Ass.~\ref{ass_partial_bounded}.
 \begin{equation}
    \begin{aligned}
    &  \min f(y \mid \bm{U},\bm{X} = x)f(\bm{U}),\\ &\text{~subject to~}
     f(y \mid \bm{U},\bm{X} = x)f(\bm{U}\mid \bm{X} = x) = f(y \mid x),~\text{where~}f(\bm{U} \mid \bm{X}=x) \text{~satisfies}\\
    & \left[\begin{matrix}
    \overline{f(\bm{W} \mid \bm{U})}f(\bm{U} \mid \bm{X}=x) - f(\bm{W} \mid \bm{X} = x)\\ f(\bm{W} \mid \bm{X} = x) - \underline{f(\bm{W} \mid \bm{U})}f(\bm{U} \mid \bm{X}=x)
    \end{matrix} \right] \geq 0, 
    \left[\begin{matrix}
    \overline{f(\bm{W} \mid \bm{U})}f(\bm{U}) - f(\bm{W}) \\ f(\bm{W}) - \underline{f(\bm{W} \mid \bm{U})}f(\bm{U})
    \end{matrix}\right] \geq 0.
    \end{aligned}  \label{wy_formulation}
\end{equation}
Notice that the feasible region of $f(y \mid \bm{U},\bm{X} = x)$ and $f(\bm{U})$ is even more irregular than in \eqref{re-formulation}. Nevertheless, we can still adopt a similar strategy to PI-SFP to approximate its optimal value. Analogously, we construct a simplex to enclose the original feasible region. Then do bisetion to generate sub simplices, and reduce our optimization problem to linear programming in the set of sub space. We will discuss in detail in the future work.

\noindent \textbf{Fig.~\ref{Fig.sub.2} and \ref{Fig.sub.3}} We extend PI-SFP on Fig.~\ref{Fig.sub.2} and \ref{Fig.sub.3}. We illustrate that the negative exposure control $\bm{Z}$ can enhance our estimation. Due to the fact $f(y\mid u,X=x) = f(y \mid u,X=x, Z)$, our original model~(\ref{eqn_basic_bound}) can be transformed as:

\begin{equation}
    \begin{aligned}
     \underline{f(Y_x = y)} := f(y,X=x)  + \max_{\mathcal{Z}\subseteq Z}\min_{f(y,\bm{W},\bm{U}, \bm{X}, \mathcal{Z}) \in \mathcal{\widetilde{F} }_{Z}}  \sum_{u=1}^{d} \frac{f(y,u_i,X=x, z \in \mathcal{Z}) f(u_i, X\neq x)}{f(u_i,X=x,z \in \mathcal{Z})}.
    \end{aligned}
\end{equation}
where $\mathcal{\widetilde{F} }_{Z}$ is the feasible region of $f(y,\bm{W},\bm{U}, \bm{X}, \mathcal{Z})$. It is constructed by an analogous way to that of constructing $\mathcal{\widetilde{F} }$ in \eqref{construction_F}.

That is to say, for each $\mathcal{Z} \subseteq Z$, we can adopt PI-SFP, and choose the maximum of which as the best lower bound $\underline{f(Y_x = y)}$.

\subsection{Discussions on extensions to the continuous confoundings}
In this section, we further consider the continuous case of $U$. We only have a priori upper/lower envelop on $\{P(\bm{W}\mid u \in [u_{i-1}, u_{i}]), i=1,2,...d \}$ as in Ass.~\ref{ass_partial_bounded}, and the optimal value is still denoted as $\underline{f(Y_x = y)}$. We re-use the branch-and-bound strategy in our main text based on discretization. The approximation error decreases with the sampling length.

\begin{assumption}{(Lipschitz condition)} $\forall y \in \bm{Y}, \forall \{u^{'},u^{''}\} \in \bm{U} \backslash \{u_0,u_1,...u_d\}$, 
\begin{equation}
    \begin{aligned}
      \left| \frac{f(y, u^{'}, X=x) - f(y, u^{''}, X=x)}{f(u^{'}, X=x) - f(u^{''}, X=x)} \right| \leq C_1, \left| \frac{f(u^{'}, X=x) - f(u^{''}, X=x)}{u^{'} - u^{''}} \right| \leq C_2,
    \end{aligned}
\end{equation}\label{ass_lipschitz}
where $ C_1, C_2$ are positive constants.
\end{assumption}

Then we have the following theorem. 

\begin{corollary}
Suppose that Ass.~\ref{ass_partial_bounded}-\ref{ass_lipschitz} holds. When $U$ is continuous, and $\max\limits_{i \in \{1,2,...d\}} |u_i - u_{i-1}| < \eta \delta$. Then
 \begin{equation}
        \begin{aligned}
          \underline{f(Y_x = y)} \leq \lim \limits_{n\rightarrow +\infty}\underline{\underline{f_{opt}^n (Y_x = y)}} &\leq \frac{1}{1-\frac{1}{2}C_2 \eta}\underline{f(Y_x = y)} + \frac{C_1f(X\neq x) -  f(y, X=x)}{2 - C_2 \eta} C_2 \eta.
         \end{aligned}
    \end{equation}
   \label{theorem_continuous}
\end{corollary}

The proof is in Appendix.~\ref{proof_continuous}.

\section{Conclusions}
In this paper, we first list the traditional settings of assumptions, such as total observability, reversibility, completeness, bridge function in the negative control problem and analyze their limitations, then we propose a fairly broad 'partial boundedness' assumption. On this basis, we develop a branch-and-bound global optimization method called PI-SFP to achieve the valid bound of ACE. In the future, we will further extend the PI-SFP approach to a wider range of graph structures, as well as to the case of continuous confoundings.

\section{Acknowledgement}
I sincerely thank Professor Yuhao Wang for his suggestions for the first four parts.

\normalem
\bibliographystyle{unsrt}  
\bibliography{references}

\newpage 
\appendices


\setcounter{equation}{0}
\setcounter{section}{0}
\setcounter{subsection}{0}
\renewcommand{\theequation}{A.\arabic{equation}}
\renewcommand{\thesubsection}{A.\arabic{subsection}}

In appendices, we provide the supplementary material and proofs for our main text. 

Appendix.~\ref{proof_basic_IR}-\ref{equal_reform_algorithm1} are for propositions. In Appendix.~\ref{proof_basic_IR}, we prove that $\mathcal{F} \subseteq \mathcal{\widetilde{F}}$. In Appendix.~\ref{equal_reform_algorithm1}, we demonstrate the bound will be tight under certain cases. 

Appendix.~\ref{discussion_ass} is for the assumption. We discuss the case when Ass.~\ref{positive definite assumption} does not hold.

Appendix.~\ref{app_main_result_1} is for the main results. First, we show that our original simplex $S_{0}$ encloses our identification region. Second, we prove that the original optimization can be transformed to the set of sub-problems in the reduced space. Third, we show our construction to transfer the original nonlinear optimization problem to the weaker linear case. Finally, we demonstrate our algorithm converges to the global optimal solution at an exponential rate.

Appendix.~\ref{app_corollary} is for the corollary, in which we extend our result from $f(Y_x=y)$ to the more general ACE.

Appendix.~\ref{app_discussion} is for extensions. We additionally discuss 1) the previous assumptions in the original literature, 2) auxiliary optimization algorithm, 3) acceleration strategy, and 4) extension to the continuous confoundings.

\subsection{The proof of proposition.~\ref{basic_IR}}\label{proof_basic_IR}
According to Ass.~\ref{ass_partial_bounded}, by integration, we can also directly claim that if $f(y,\bm{U}, \bm{W}, \bm{X}) \in \mathcal{F}$, then 
\begin{equation}
    \begin{aligned}
     &\underline{P^{}(\bm{W} \mid \bm{U})} \bm{\theta} \leq P(y,\bm{W},X=x) \leq  \overline{P^{}(\bm{W} \mid \bm{U})} \bm{\theta}, \forall x \in X.\\
     &\underline{P^{}(\bm{W} \mid \bm{U})} \bm{\psi} \leq P(\bm{W},X=x) \leq  \overline{P^{}(\bm{W} \mid \bm{U})} \bm{\psi}, \forall x \in X.\\
     &\underline{P^{}(\bm{W} \mid \bm{U})} \bm{\omega} \leq f(\bm{W},X\neq x) \leq  \overline{P^{}(\bm{W} \mid \bm{U})}\bm{\omega}, \forall x \in X.
    \end{aligned}
\end{equation}
Thus
\begin{equation}
    \begin{aligned}
        \left[\begin{matrix}
   -\bm{I_{d*d}} \\ \bm{I_{d*d}}
   \end{matrix}\right] \left[\begin{matrix}
   &f(y,\bm{W},X=x)^{T} \\ &f(\bm{W},X=x)^{T} \\ &f(\bm{W},X\neq x)^{T}
   \end{matrix} \right]^{T}  -  \left[ \begin{matrix}
   &-\overline{P(\bm{W}\mid \bm{U})} \\  &\underline{P(\bm{W} \mid \bm{U})}
   \end{matrix} \right] \bm{\phi} \geq \bm{0}.
    \end{aligned}
\end{equation}
Combined with the natural that $\theta_i \in [0, P(y,X=x)]$, $\psi_i \in [0,P(X=x)], \omega_i \in [0, P(X\neq x)], i=1,2,...d$, we have $f(y,\bm{U}, \bm{W}, \bm{X}) \in \mathcal{\widetilde{F}}$. In conclusion, we claim $\mathcal{F} \subseteq \mathcal{\widetilde{F}}$.

\subsection{The proof of proposition.~\ref{proposition_tight}} \label{equal_reform_algorithm1}

\begin{proof}
As the optimal solution $\bm{\phi_{opt}}$ satisfies the constraint \eqref{constraint_prove_tight} in Proposition.~\ref{proposition_tight}, we can equivalently claim that $\bm{\phi_{opt}}$ is compatible with some $f(y,\bm{W}, \bm{U}, \bm{X})$ which satisfies $f(y,\bm{W}, \bm{U}, \bm{X}) \in \mathcal{F}$. On this basis, Formulation.~\ref{formulation} can be transformed with stricter constraints but equal minimum optimal value, namely that from

\begin{equation}
    \begin{aligned}
    &  \text{min~} f(y,X=x) +  \sum_{i=1}^{d}\frac{1}{\psi_{i}}\theta_{i}\omega_{i}\\
    &\text{subject to:~} f(y,\bm{U},\bm{W}, \bm{X}) \in \mathcal{\widetilde{F}}
    \end{aligned} 
\end{equation}

to 

\begin{equation}
    \begin{aligned}
    &  \text{min~} f(y,X=x) +  \sum_{i=1}^{d}\frac{1}{\psi_{i}}\theta_{i}\omega_{i}\\
    &\text{subject to:~} f(y,\bm{U},\bm{W}, \bm{X}) \in \mathcal{\widetilde{F}} \cap \mathcal{{F}} = \mathcal{{F}}.
    \end{aligned} 
\end{equation}

This is equal to the original \eqref{eqn_basic_bound}. Hence $f(Y_x = y)$ is the tight lower bound of $f(Y_x = y)$ under constraint.~\ref{constraint_prove_tight}. 

By contrast, if the constraint \eqref{constraint_prove_tight} does not hold, then any $f(y,\bm{U},\bm{W}, \bm{X})$ compatible with $\bm{\phi_{opt}}$ will be within $\mathcal{F}^c \cap \widetilde{\mathcal{F}}$. In another word, the minimum value of Formulation.~\ref{formulation} will be lower than that of Formulation.~\ref{eqn_basic_bound}, and the bound $\underline{f(Y_x = y)}$ is not tight. Proved.

\end{proof}
The construction of $f(y, \bm{W}, \bm{U}, \bm{X})$ is 
\begin{equation}
    \begin{aligned}
        \left[ \begin{matrix}
        f(Y=y, W=w_1, \bm{U}, \bm{X}) & f(Y\neq y, W=w_1, \bm{U}, \bm{X}) \\ f(Y=y, W=w_2, \bm{U}, \bm{X}) & f(Y\neq y, W=w_2, \bm{U}, \bm{X})
        \end{matrix}
        \right] = \left[\begin{matrix} 0 & 0.15 & 0.18 & 0.15 \\ 0.08 & 0 & 0 & 0 \\ 0 & 0.1 & 0.12 & 0.1 \\ 0.12 & 0 & 0 & 0 \end{matrix} \right]
    \end{aligned}
\end{equation}

\subsection{Further discussion on Ass.~\ref{positive definite assumption}}\label{discussion_ass}

In this section, we consider the case when Ass.~\ref{positive definite assumption} does not hold. We propose a new version of PI-SFP. Recall that our objective function is:
\begin{equation}
    \begin{aligned}
   \underline{f(Y_x = y)} = ~&  \text{min~} f(y,X=x) +  \sum_{i=1}^{d}\frac{1}{\psi_{i}}\theta_{i}\omega_{i}\\
    &\text{subject to:~} f(y,\bm{U},\bm{W}, \bm{X}) \in \mathcal{\widetilde{F}}, i.e., \bm{\phi} \in IR_{\bm{\Phi}}.
    \end{aligned}
\end{equation}
In our main text, we let $\psi_i^o = \frac{1}{\psi_i}$. However, when we can not guarantee that $\exists \delta$, $\forall i \in \{1,2,...d\}, \psi_i \geq \delta$ (without Ass.~\ref{positive definite assumption},), then $\psi_i^o$ may turn to infinity. On this basis, we introduce another algebraic distortion $\psi_i^o = \frac{\theta_i \omega_i }{\psi_i}$. Then the above programming can be transformed to:
\begin{equation}
    \begin{aligned}
    \underline{f(Y_x = y)} = ~&  \text{min~} f(y,X=x) +  \sum_{i=1}^{d} \psi_i^o \\
    &\text{subject to:~} \bm{\gamma} \in IR_{{\Gamma}}, \text{where~} IR_{\Gamma} = \{\bm{\gamma}: \bm{\phi} \in IR_{{\Phi}},
    \psi_{i}^{o} \psi_{i}  = \theta_i \omega_i, \psi_i^o \leq C, i=1,...d.\}.\\
    \end{aligned}\label{new_PI_SFP}
\end{equation}
where $C$ is a local optimal value (a priori computed) of $\underline{f(Y_x = y)}$. On this basis, we can adopt the analogous strategy as in the traditional PI-SFP. Here the original $S_0$ is easy to be constructed since $\| \bm{\gamma} \|_{+\infty} < +\infty$. 

Programming \eqref{new_PI_SFP} can also be adopted under Ass.~\ref{positive definite assumption}. Compared with the traditional PI-SFP, firstly, programming \eqref{new_PI_SFP} needs an a priori computed $C$ to upper bound $\psi_i^o$. Secondly, we will do linearization on $\psi_{i}^{o} \psi_{i}  = \theta_i \omega_i$ instead of $\psi_{i}^{o} \psi_{i}  = 1$, which is more complex. There is no guarantee of which version is better and we will explore it in the future work.

\subsection{The proof of Theorem.~\ref{convergence_theorem}}\label{app_main_result_1}

\noindent {\textbf{The sketch of proof}} This is the main result of our paper. The main procedure are as follows:
\begin{equation}
    \begin{aligned}
    & \mid \underline{f(Y_x = y)}  - \underline{f_{opt}^n (Y_x=y)} \mid \\
    \overset{}{=} &\mid \underline{f(Y_x = y)}  - \max\limits_{k \in \{0,1,...n\}} \underline{\underline{f_{\tilde{S}_{k}}(Y_x = y)}} \mid  &\text{Definition of~} \underline{f_{opt}^n (Y_x=y)}\\
     =& \min\limits_{k \in \{0,1,...n\}} \mid \underline{f(Y_x = y)}  -  \underline{\underline{f_{\tilde{S}_{k}}(Y_x = y)}} \mid \\
     \overset{\textbf{(1)}}{=} & \min\limits_{k \in \{0,1,...n\}} \mid \underline{f_{S_{0}}(Y_x = y)}  -  \underline{\underline{f_{\tilde{S}_{k}}(Y_x = y)}} \mid &{{\textbf{Initialization()}}}\\
      \overset{\textbf{(2)}}{=} & \min\limits_{k \in \{0,1,...n\}} \mid \min\limits_{S\in \cS_k} {\underline{f_{S}(Y_x = y)}}  -  \underline{\underline{f_{{\tilde{S}_{k}}}(Y_x = y)}} \mid &{{\textbf{Bisection()}}}\\
    \overset{*}{\leq} & \min\limits_{k \in \{0,1,...n\}} \mid \underline{f_{{\tilde{S}_{k}}}(Y_x = y)}  - \underline{{\underline{f_{{\tilde{S}_{k} }}(Y_x = y)}} } \mid \\
    {\leq}& \mid {{\underline{f_{\tilde{S}_{i_{L_n}  }}(Y_x = y)}} } - \underline{{\underline{f_{\tilde{S}_{i_{L_n}}}(Y_x = y)}} }\mid  \\
    \overset{\textbf{(3)}}{=} & O(dia(\tilde{S}_{i_{L_n}  })) &{{\textbf{Bounding()}}}\\
    \overset{\textbf{(4)}}{=} & O((\frac{\sqrt{3}}{2})^{\lfloor\frac{L_n}{4d}\rfloor} ). &{{\textbf{Global\_error()}}}
    \end{aligned}\label{final_conclusion}
\end{equation}
 
$*$ is directly by \textbf{(2)} and we have previously mentioned it in Formulation.~\ref{shrinking_diameter}. In the following demonstration, we mainly focus on procedure $\textbf{(1)(2)(3)(4)}$, corresponding to the algorithm part \textbf{Initialization()}, \textbf{Bisection()}, \textbf{Bounding()}, \textbf{Global\_error()} in order. 

~\quad \\
\noindent{\textbf{The proof of (1)}} 
\noindent We claim that $IR_{\Gamma} \subseteq S_{0}$:
\begin{lemma}
The original $S_{0}$ satisfies $IR_{\Gamma} \subseteq S_{0}$, and thus $\underline{f(Y_{x} = y)} = \underline{f_{S_{0}}(Y_{x} = y)}$.
\label{lemma_enclose}
\end{lemma}
\begin{proof}
\noindent 
The simplex construction is as follows. $S_{0}$ is spanned by $\{S_{0}^0, S_{0}^1, ...S_{0}^{4d}\}$, where
 \begin{equation}
     \begin{aligned}
      S_{0}^{i} =  \begin{cases}
 \bm{\gamma^l},~i=0 \\ \bm{\gamma^l}+ (\alpha - \bm{1_{1*4d}}\bm{\gamma^l})*\bm{\vec{e_i}} ,~i \in \{1,2,...4d\}
 \end{cases}, \text{~where~} \bm{\gamma^{l}} = (\gamma^{l}_1, \gamma^{l}_2,...\gamma^{l}_{4d})^T,
     \end{aligned}\label{span_S00}
 \end{equation}

where $S_{0}^{i}$ is the supporting vertices set described in our main text. For each $ \bm{\gamma} \in IR_{\Gamma}$, we attempt to provide a direct construction as follows:
\begin{equation}
    \begin{aligned}
    \forall \bm{\gamma} \in IR_{\Gamma}, \text{we~have~} \bm{\gamma} \overset{*}{=} \sum_{i=0}^{4d}\beta_{i}S_{0}^i, ~\beta_i = \begin{cases}
1-\sum_{i=1}^{4d}\beta_i, i=0 \\ \frac{\bm{\gamma}\bm{\vec{e_i}}-\gamma_i^l}{\alpha - \bm{1_{1*4d}}\bm{\gamma^l}}, i=1,2,...4d \\
 \end{cases}, \beta_i \in [0,1],
    \end{aligned} \label{original_beta}
\end{equation}
To prove \eqref{original_beta}, we only need to prove the correctness of the equality $*$ and the fact $\beta_i \in [0,1], \forall i =0,1,...4d.$
 
First, we demonstrate the correctness of this construction. 
\begin{equation}
    \begin{aligned}
      \sum_{i=0}^{4d}\beta_{i}S_{0}^i &= \beta_0 S_{0} + \sum_{i=1}^{4d}\beta_{i}S_{0}^i \\
      &= \beta_0 \bm{\gamma^l} + \sum_{i=1}^{4d}\beta_{i} \left(\bm{\gamma^l}+ (\alpha - \bm{1_{1*4d} }\bm{\gamma^l} )\bm{\vec{e_i}}\right) &\text{(definition of } S_{0}^i\text{)}\\
      &= (1-\sum_{i=1}^{4d}\beta_i) \bm{\gamma^l} + \sum_{i=1}^{4d}\beta_{i} \left(\bm{\gamma^l}+ (\alpha - \bm{1_{1*4d} }\bm{\gamma^l} )\bm{\vec{e_i}}\right) &\text{(definition of } \beta_i \text{)}\\
      &= \bm{\gamma^l} + \sum_{i=1}^{4d}\frac{\bm{\gamma}\bm{\vec{e_i}}-\gamma_i^l}{\alpha - \bm{1_{1*4d}}\bm{\gamma^l}} \left((\alpha - \bm{1_{1*4d} }\bm{\gamma^l} )\bm{\vec{e_i}}\right) &\text{(definition of } \beta_i \text{)} \\
      &= \bm{\gamma^l} + \left(\sum_{i=1}^{4d} \bm{\gamma}\bm{\vec{e_i}}-\gamma_i^l \right)\bm{\vec{e_i}} = \bm{\gamma}.
    \end{aligned}
\end{equation}

Second, we claim $\forall i \in \{1,...4d\}, \beta_i \in [0,1]$. Since we already have $\beta_i>0, i=1,2,...4d$ according to the construction of $\{\alpha, \bm{\gamma^l}\}$, we only need to prove the left: $\beta_0 >0$. Notice that
\begin{equation}
    \begin{aligned}
    \sum_{i=1}^{4d} \beta_i &= \sum_{i=1}^{4d}\frac{\bm{\gamma}\bm{\vec{e_i}}-\gamma_i^l}{\alpha - \bm{1_{1*4d}}\bm{\gamma^l}} = \frac{\bm{1_{1*4d}\bm{\gamma}} - \bm{1_{1*4d}{\bm{\gamma^l}}}}{\alpha - \bm{1_{1*4d}}\bm{\gamma^l}}.
    \end{aligned}
\end{equation}
Due to $\beta_0 = 1-  \sum_{i=1}^{4d} \beta_i$, it is equal to prove 
\begin{equation}
    \bm{1_{1*4d}{\bm{\gamma}}} \leq \alpha = 1 + f(y, X=x) + 
\frac{d^2 (\psi^{l}+ \psi^{u})^2 }{4f(X=x) \psi^{l} \psi^{u}  }  ,
\end{equation}
where $\psi^{l}, \psi^{u}$ are identified in the main text. It is equivalent to
\begin{equation}
    \begin{aligned}
        \sum_{i=1}^{d} \psi^o_i + \sum_{i=1}^d \theta_i +  \sum_{i=1}^d \psi_i +  \sum_{i=1}^d \omega_i  \leq 1 + f(y, X=x) + 
\frac{d^2 (\psi^{l}+ \psi^{u})^2 }{4f(X=x) \psi^{l} \psi^{u}  },
    \end{aligned}
\end{equation}
namely that
\begin{equation}
    \begin{aligned}
        \sum_{i=1}^{d} \psi_{i}^o \leq \frac{d^2 (\psi^{l}+ \psi^{u})^2 }{4f(X=x) \psi^{l} \psi^{u}  }.
    \end{aligned}\label{inverse_chauchy}
\end{equation}
We only need prove the inequality~\eqref{inverse_chauchy}. It is due to the fact $(\psi_{i} -  \psi^{l})(\frac{1}{\psi_{i}} - \frac{1}{{\psi^{u}}}) \geq 0$, namely $1+\frac{ \psi^{l}}{ {\psi^{u}}} \geq \frac{ \psi^{l}}{\psi_{i}} + \frac{\psi_{i}}{\psi^{u}}$. By which we have
\begin{equation}
    \begin{aligned}
& (1+\frac{ \psi^{l}}{ {\psi^{u}}})d \geq  \psi^{l} \sum_{i=1}^{d} \frac{1}{\psi_{i}} + \frac{1}{ {\psi^{u}}} \sum_{i=1}^{d} \psi_{i} \geq 2\sqrt{\frac{ \psi^{l}}{{\psi^{u}}}} \sqrt{\sum_{i=1}^{d} \frac{1}{\psi_{i}}} \sqrt{f(X=x)}.
    \end{aligned}
\end{equation}
It is equal to
\begin{equation}
    \begin{aligned}
    \sum_{i=1}^{d} \psi_{i}^o = \sum_{i=1}^{d} \frac{1}{\psi_{i}}  \leq \frac{( \psi^{u} +  \psi^{l})^2 d^2}{4 f(X=x) \psi^{u} \psi^{l}}, \text{and thus} \sum_{i=1}^{4d}  \beta_i \in [0,1].
    \end{aligned}
\end{equation}

On this basis, $\beta_0 = 1-\sum_{i=1}^{4d}\beta_i \in [0,1]$. Combining with $\beta_i \geq 0, i\in\{0,1,...4d\}$ and Eqn.~\eqref{original_beta}, we claim that $\forall \bm{\gamma} \in IR_{\Gamma}$, we have $\bm{\gamma} \in S_{0}$. Due to the arbitrary of $\bm{\gamma}$, we have $IR_{\Gamma} \subseteq S_{0}$, and thus $\underline{f(Y_{x} = y)} = \underline{f_{S_{0}}(Y_{x} = y)}$.
\end{proof}

~\quad \\
\noindent{\textbf{The proof of (2)}}
\noindent We introduce the following lemma:
\begin{lemma}\label{lemma_ObjS0_equal_ObjSk}
The partitioning set $\cS_{k}$ satisfies $\underline{f_{S_{0}}(Y_{x} = y)} = \min\limits_{S\in \cS_k}  \underline{f_{S}(Y_{x} = y)}$.

\end{lemma}

\begin{proof}
By definition of bisection process, $\tilde{S}_k$ is bisectioned into $\tilde{S}_{k1}, \tilde{S}_{k2}$. Then
\begin{equation}
    \begin{aligned}
        \cS_{k + 1} := \left(\cS_k \setminus \tilde{S_k}\right) \cup \{\tilde{S}_{k1}, \tilde{S}_{k2}\}
    \end{aligned}
\end{equation}
Hence we have $\cup_{S\in \cS_{k} } S = \cup_{S\in \cS_{k+1} } S, \forall k=0,1,...$ Thus $S_{0} = \cup_{S\in \cS_{k} } S$, and we have
\begin{equation}
\begin{aligned}
     \underline{f_{S_{0}}(Y_{x} = y)}
        = \underline{f_{(\cup_{S\in \cS_{k} } S)}(Y_x = y)}
        = \min\limits_{S\in \cS_k}  \underline{f_{S}(Y_{x} = y)} .
    \end{aligned}
\end{equation}
Hence we have proved.
\end{proof}

~\quad \\
\noindent{\textbf{The proof of (3)}}
\noindent We first introduce lemma.~\ref{dc_decom_lemma} and lemma.~\ref{tan_sec_bounded} for preparation, then the procedure \textbf{(3)} is proved by lemma.~\ref{second_bounded}.

\begin{lemma}
The decomposition of \eqref{re-formulation} can be established as Eqn.~\eqref{dc_decomposition}.
\label{dc_decom_lemma}
\end{lemma}
\begin{proof}
Specifically, we give the explicit decomposition as follows, and the sub-script $cyc$ means the cycle of symbol set $[\psi_{i}^o, \theta_{i}, \omega_{i}]$:
\begin{equation}
    \begin{aligned}
    & \psi_{i}^o \theta_{i} \omega_{i} \\
    = & \left[ \frac{1}{2}\sum\limits_{cyc} (\psi_{i}^o)^2 \theta_{i} + \frac{1}{2}\sum\limits_{cyc} (\psi_{i}^o) \theta_{i} ^2 + \sum\limits_{i=1}^d \psi_{i}^o \theta_{i} \omega_{i} \right] - \frac{1}{2}\sum\limits_{cyc} (\psi_{i}^o)^2 \theta_{i} - \frac{1}{2}\sum\limits_{cyc} (\psi_{i}^o) \theta_{i} ^2 \\
    =& \left[\frac{1}{6}(\sum\limits_{cyc} \psi_{i}^o )^3 -  \frac{1}{6}(\sum\limits_{cyc} (\psi_{i}^o)^3) \right] -  \frac{1}{2}\sum\limits_{cyc} (\psi_{i}^o)^2 \theta_{i} - \frac{1}{2}\sum\limits_{cyc} (\psi_{i}^o) \theta_{i} ^2 \\
    =& \left[\frac{1}{6}(\sum\limits_{cyc} \psi_{i}^o )^3 -  \frac{1}{6}(\sum\limits_{cyc} (\psi_{i}^o)^3) \right] + \frac{1}{2}\sum\limits_{cyc}(\psi_{i}^o)^4 - \frac{1}{4}\sum\limits_{cyc}(\psi_{i}^o)^4 - \frac{1}{4}\sum\limits_{cyc}(\theta_{i})^2 -  \frac{1}{2}\sum\limits_{cyc} (\psi_{i}^o)^2 \theta_{i} \\ &+  \frac{1}{2}\sum\limits_{cyc}(\psi_{i}^o)^2 - \frac{1}{4}\sum\limits_{cyc}(\psi_{i}^o)^2 - \frac{1}{4}\sum\limits_{cyc}(\theta_{i})^4 -  \frac{1}{2}\sum\limits_{cyc} \psi_{i}^o \theta_{i}^2 \\
    =& \left[\frac{1}{6}(\sum\limits_{cyc} \psi_{i}^o )^3 -  \frac{1}{6}(\sum\limits_{cyc} (\psi_{i}^o)^3) \right] + \frac{1}{2}\sum\limits_{cyc}(\psi_{i}^o)^4 - \frac{1}{4}\sum\limits_{cyc}((\psi_{i}^o)^2+\theta_{i})^2 + \frac{1}{2}\sum\limits_{cyc}(\psi_{i}^o)^2 - \\ &\frac{1}{4}\sum\limits_{cyc}(\psi_{i}^o+\theta_{i}^2)^2.
    \end{aligned}
\end{equation}
On this basis, if we choose 
\begin{equation}
    \begin{aligned}
    &C_1(\bm{\gamma}) =\sum_{i=1}^{d} \left[ \frac{1}{6}(\sum\limits_{cyc} \psi_{i}^o )^3+\frac{1}{2}\sum\limits_{cyc}(\psi_{i}^o)^4 + \frac{1}{2}\sum\limits_{cyc}(\psi_{i}^o)^2 \right],\\
    &C_2(\bm{\gamma}) = \sum_{i=1}^{d} \left[ \frac{1}{6}\sum\limits_{cyc} (\psi_{i}^o)^3 + \frac{1}{4}\sum_{cyc}[(\psi_{i}^o)^2+\theta_{i}]^2 + \frac{1}{4}\sum_{cyc}[\psi_{i}^o+\theta_{i}^2]^2\right],
    \end{aligned}
\end{equation}
then we have 
\begin{equation}
    \begin{aligned}
    \sum_{i=1}^{d}\psi_{i}^o \theta_{i} \omega_{i} = C_1(\bm{\gamma}) - C_2(\bm{\gamma}). 
    \end{aligned}
\end{equation}
Here the Hessian matrix $\frac{\partial^{2}C_1(\bm{\gamma})}{\partial^2( \bm{\gamma})}$ and $\frac{\partial^{2}C_2(\bm{\gamma})}{\partial^2( \bm{\gamma})}$~is~positive~semi-definite:
\begin{equation}
    \begin{aligned}
       \frac{\partial^{2}C_1(\bm{\gamma})}{\partial^2( \bm{\gamma})} =\frac{\partial^{2}C_2(\bm{\gamma})}{\partial^2( \bm{\gamma})} = \left[\bm{\gamma} + 6 \bm{\gamma} \circ \bm{\gamma} + \bm{1_{4d*1}}\right] \circ \left[\begin{matrix}
       \bm{1_{1*2d}}& \bm{0_{1*d}} &\bm{1_{1*d}}
       \end{matrix}\right]^T \geq \bm{0_{4d*1}},
    \end{aligned}
\end{equation}

where $\circ$ denotes the Hadamard product. Moreover,
\begin{equation}
    \begin{aligned}
    \frac{\partial^{2}D_{i1}(\bm{\gamma})}{\partial^2( \bm{\gamma})} = \frac{\partial^{2}D_{i2}(\bm{\gamma})}{\partial^2( \bm{\gamma})} = \left[\begin{matrix}
       \bm{1_{1*d}}& \bm{0_{1*d}} &\bm{1_{1*d}} & \bm{0_{1*d}}
       \end{matrix}\right]^T \geq \bm{0_{4d*1}}.
    \end{aligned}
\end{equation}
$D_{i1}(\bm{\gamma}), D_{i2}(\bm{\gamma})$ are also positive semi-definite.
\end{proof}
On this basis, we further give the upper and lower bound of the convex function as follows:
\begin{lemma}
If function $F(\bm{\gamma})$ is differential and convex restricted by any simplex $S$, then
\begin{equation} 
    \begin{aligned}
    F^{\text{tan}}(\bm{\gamma}) \leq F(\bm{\gamma}) \leq F^{\text{sec}}(\bm{\gamma}),
    \end{aligned}
\end{equation}
where $\bm{\gamma_0} \in S$. In our paper, function $F(\cdot)$ can be chosen as $C_1(\cdot),C_2(\cdot),D_{i1}(\cdot),D_{i2}(\cdot)$, and $F^{\text{tan}}(\cdot), F^{\text{sec}}(\cdot)$ hold the same construction as in Formulation~\ref{tan_sec}.\label{up_low_bound_convex}  \footnote{The matrix of the starting simplex $\left[\begin{matrix} S_{0}^0, ... , S_{0}^{4d} \\ 1,...,1 \end{matrix}\right]^{}$ is reversible by the construction in lemma.~\ref{lemma_enclose}. Moreover, the reversibility of $\left[\begin{matrix} S_{}^0, ... , S_{}^{4d} \\ 1,...,1 \end{matrix}\right]^{}, S\in \cS_{k}, k=0,1,...$ still holds during bisection, since each bisection can be seen as a linear transformation between different columns.} \label{tan_sec_bounded}
\end{lemma}

\begin{proof}
The left part is intuitive. It is the tangent line equation of $F(\bm{\gamma})$. We only consider the right part by the convex property of $F(\bm{\gamma})$, whose construction is motivated by~\cite{pei2013global}. We use $\bm{\gamma}_i, i=1,2,...4d$ to denote the value of $\bm{\gamma}$ on each dimension ($\lambda_i \in [0,1],~\sum_{i=0}^{4d}\lambda_i = 1$):
\begin{equation}
    \begin{aligned}
    F(\bm{\gamma}) &= F(\sum_{i=0}^{4d} \lambda_i S^i) \leq \sum_{i=0}^{4d} \lambda_{i} F(S^i) \\
    &= \sum_{i=0}^{4d} \lambda_i [F(S^{0}), F(S^1), ..., F(S^{4d})] \left[\begin{matrix} S^0, ... , S^{4d} \\ 1,...,1 \end{matrix}\right]^{-1}[\begin{matrix}S^i \\ 1\end{matrix}] \\
    & = [F(S^{0}), F(S^2), ..., F(S^{4d})] \left[\begin{matrix} S^0, ... , S^{4d} \\ 1,...,1 \end{matrix}\right]^{-1}[\begin{matrix}\bm{\gamma} \\ 1\end{matrix}] = F^{\text{sec}}(\bm{\gamma}).
    \end{aligned}
\end{equation}
Hence we have proved our lemma.
\end{proof}

On this basis, we can claim \eqref{lemma_sub_linear} provides the lower bound of $\underline{f_{S}(Y_x = y)}$, namely $\underline{\underline{f_{S}(Y_x = y)}} \leq \underline{f_{S}(Y_x = y)}$.

~\\ \quad \\
After the above difference-in-convex linear construction, we introduce the following lemma to approximate $\underline{f_{S_{}}(Y_x = y)}$ by $\underline{\underline{f_{S_{}}(Y_x = y)} } $:

\begin{lemma}
\begin{equation}
    \begin{aligned}
    \forall S, \mid \underline{f_{S}(Y_x = y)}  - \underline{\underline{f_{S}(Y_x = y)} }      \mid \leq A * dia(S_{})^2,
    \end{aligned}
\end{equation}
 $A = \max\limits_{\bm{\gamma} \in S_{0}} \|\frac{\partial{(C_1(\bm{\gamma}) - C_2(\bm{\gamma}))}}{\partial{\bm{\gamma}}}\|\frac{2(\sqrt{2}+1)\sqrt{d}}{\delta}  + \max\limits_{\bm{\gamma} \in S_{0}} \|\frac{\partial^2 C_1(\bm{\gamma)}}{\partial \bm{\gamma}^2}\|_F +  \frac{1}{2} \max\limits_{\bm{\gamma} \in S_{0}} \|\frac{\partial^2 C_2(\bm{\gamma)}}{\partial \bm{\gamma}^2} \|_F < +\infty$.
\label{second_bounded}
\end{lemma} 
\begin{proof}
Since $dia(S_{0}) < +\infty$, we have that each element of $\gamma \in S_{0}$ can be bounded, namely $\| \gamma \|_{+\infty}<+\infty $. Then $\|\frac{\partial C_1(\bm{\gamma)}}{\partial \bm{\gamma}}\|$,  $\|\frac{\partial C_2(\bm{\gamma)}}{\partial \bm{\gamma}}\|$, $\|\frac{\partial^2 C_1(\bm{\gamma)}}{\partial \bm{\gamma}^2}\|_F$, $\|\frac{\partial^2 C_2(\bm{\gamma)}}{\partial \bm{\gamma}^2}\|_F$ are all finite. Here $\|\cdot \|$ denotes the Euclidean norm, and $\|\cdot \|_F$ denotes the Frobenius norm. 

If the corresponding optimal solution of 
$\underline{f_S(Y_x = y)}$ and $\underline{\underline{f_S(Y_x = y)}}$ are denoted as $\underline{\bm{\gamma}}$ and $\underline{\underline{{\bm{\gamma}}}}$ ($\underline{\bm{\gamma}}, \underline{\underline{{\bm{\gamma}}}} \in S$). Then according to lemma.~\ref{tan_sec_bounded}, $\mid \underline{f_{S_{}}(Y_x = y)}  - \underline{\underline{f_{S_{}}(Y_x = y)} }\mid$ can be bounded as follows:
\begin{equation}
\begin{aligned}
       0 \leq &{\underline{f_{S_{}}(Y_x = y)}}  - \underline{\underline{f_{S_{}}(Y_x = y)} } \\ = &\mid C_1({\underline{{\bm{\gamma}}}}) - C^{\text{tan}}_1(\underline{\underline{{\bm{\gamma}}}})  - C_2({\underline{{\bm{\gamma}}}}) + C^{\text{sec}}_2(\underline{\underline{{\bm{\gamma}}}}) \mid\\
       \leq &\mid C_1(\underline{\underline{{\bm{\gamma}}}}) -  C^{\text{tan}}_1(\underline{\underline{{\bm{\gamma}}}})  -
       C_2(\underline{\underline{{\bm{\gamma}}}}) + C^{\text{sec}}_2(\underline{\underline{{\bm{\gamma}}}}) \mid + 
       \mid C^{\text{}}_1({\underline{{\bm{\gamma}}}}) -  C^{\text{}}_1(\underline{\underline{{\bm{\gamma}}}}) - C^{\text{}}_2({\underline{{\bm{\gamma}}}}) +  C^{\text{}}_2(\underline{\underline{{\bm{\gamma}}}}) \mid \\
       \overset{*}{\leq} & {\underbrace{\mid C_1(\underline{\underline{{\bm{\gamma}}}}) -  C^{\text{tan}}_1(\underline{\underline{{\bm{\gamma}}}}) \mid }_{(1)}} + \underbrace{\mid  C_2(\underline{\underline{{\bm{\gamma}}}}) - C^{\text{sec}}_2(\underline{\underline{{\bm{\gamma}}}}) \mid}_{(2)} +
       \underbrace{\mid (C^{\text{}}_1({\underline{{\bm{\gamma}}}}) -  C^{\text{}}_2({\underline{{\bm{\gamma}}}})) - ( C^{\text{}}_1({\underline{\underline{{\bm{\gamma}}}}}) -  C^{\text{}}_2(\underline{\underline{{\bm{\gamma}}}}))  \mid}_{(3)} \\
       \end{aligned}\label{linear_bound_main}
\end{equation}

\noindent \noindent \textbf{item~(1)}:We consider the last line. The tangent line equation satisfies the following bound by Taylor expansion:
\begin{equation}
    \mid C_1(\bm{\gamma}) - C^{\text{tan}}_1(\bm{\gamma}) \mid = O(\max\limits_{\bm{\gamma} \in S_{0}} \|\frac{\partial^2 C_1(\bm{\gamma)}}{\partial \bm{\gamma}^2} \|_F (dia(S_{}))^2) = O( dia(S_{})^2), \label{bound_1}
\end{equation}

\noindent \textbf{item~(2)}: On the other hand, note that $\bm{\gamma} = \sum\limits_{j=0}^{4d} \lambda_j S_{}^j$, here $\sum\limits_{j=0}^{4d} \lambda_j = 1, \lambda_j \geq 0$:
\begin{equation}
    \begin{aligned}
    &\mid C_2(\bm{\gamma}) - C^{\text{sec}}_2(\bm{\gamma}) \mid \\ =& -[C_2(S_{}^0), C_2(S_{}^1),... C_{2}(S_{}^{4d})] \left[\begin{matrix} S_{}^0, ... , S_{}^{4d} \\ 1,...,1 \end{matrix}\right]^{-1}\left[\begin{matrix}{\sum\limits_{i=0}^{4d} \lambda_i S_{}^{i}} \\ 1\end{matrix}\right] + C_2(\sum_{j=0}^{4d} \lambda_j S_{}^{j}) \\
    = &  \sum_{j=0}^{4d} \lambda_j C_2(S_{}^j) -C_2(\sum_{j=0}^{4d} \lambda_j S_{}^{j}).
    \label{bound_2}
    \end{aligned}
\end{equation}

We now aim to bound Eqn.~\eqref{bound_2}, inspired by \cite{budimir2001further}. For simplicity, we use $\triangledown$ to denote the derivative of a vector. Notice that the convex function has the property:
\begin{equation}
    \begin{aligned}
    C_2(\sum\limits_{i=0}^{4d}\lambda_j S^j ) - C_2(S^j) \geq \langle  \nabla C_2(S^j), \sum\limits_{j=0}^{4d} \lambda_j S^j - S^j\rangle 
    \end{aligned}
\end{equation}
By summation, we have
\begin{equation}
    \begin{aligned}
    \eqref{bound_2} &= \sum\limits_{j=0}^{4d} \lambda_j C_2(S^j) - C_2(\sum\limits_{i=0}^{4d}\lambda_j S^j ) \\ &\leq \sum_{j=0}^{4d} \lambda_j \langle  \nabla C_2(S^j), -\sum\limits_{j=0}^{4d} \lambda_j S^j + S^j\rangle \\
    &= \sum_{j=0}^{4d} \lambda_j \langle \nabla C_2(S^j), S^j \rangle - \langle \sum_{j=0}^{4d} \lambda_j S^j, \sum_{j=0}^{4d} \lambda_j \nabla C_2(S^j) \rangle
    \end{aligned}\label{inverse_qinsen}
\end{equation}

\eqref{inverse_qinsen} equals to
\begin{equation}
    \begin{aligned}
     &\frac{1}{2} \sum_{i=0}^{4d} \sum_{j=0}^{4d} \lambda_i \lambda_j \left[\left[ \langle \nabla C_2(S^j), S^j \rangle + \langle \nabla C_2(S^i), S^i  \rangle \right]- \left[\langle \nabla C_2(S^j), S^i \rangle + \langle \nabla C_2(S^i), S^j \rangle\right] \right] \\
     = & \frac{1}{2} \sum_{i=0}^{4d} \sum_{j=0}^{4d} \lambda_i \lambda_j \langle S^i - S^j, \nabla C_2(S^i) - \nabla C_2(S^j) \rangle \\
     \leq & \frac{1}{2} \sum_{i=0}^{4d} \sum_{j=0}^{4d} \lambda_i \lambda_j \| S^i - S^j\| \|\nabla C_2(S^i) - \nabla C_2(S^j) \| \\
     \leq & \frac{1}{2} \sum_{i=0}^{4d} \sum_{j=0}^{4d} \lambda_i \lambda_j (\max\limits_{\bm{\gamma} \in S} \|\frac{\partial^2 C_2(\bm{\gamma})}{\partial \bm{\gamma}^2}\|_F) \| S^i - S^j\|^2 \\
     \leq & \frac{1}{2} \sum_{i=0}^{4d} \sum_{j=0}^{4d} \lambda_i \lambda_j (\max\limits_{\bm{\gamma} \in S}\|\frac{\partial^2 C_2(\bm{\gamma})}{\partial \bm{\gamma}^2}\|_F) dia(S)^2 \\
     \leq & \frac{1}{2}(\max\limits_{\bm{\gamma} \in S_0}\|\frac{\partial^2 C_2(\bm{\gamma})}{\partial \bm{\gamma}^2}\|_F) dia(S)^2 .
    \end{aligned}\label{bound_chauchy}
\end{equation}

We have
\begin{equation}
    \begin{aligned}
         0 \leq \sum_{j=0}^{4d} \lambda_j C_2(S_{}^j) -C_2(\sum_{j=0}^{4d} \lambda_j S_{}^{j}) \leq \frac{1}{2} \max\limits_{\bm{\gamma} \in S_{}} \|\frac{\partial^2 C_2(\bm{\gamma)}}{\partial \bm{\gamma}^2} \|_F (dia(S_{}))^2 = O((dia(S_{}))^2).
    \end{aligned}
\end{equation}
Thus 
\begin{equation}
    \begin{aligned}
    \text{Eqn.~\eqref{bound_2}} &= \sum_{j=0}^{4d} \lambda_j C_2(S_{}^j) - C_2(\sum_{j=0}^{4d} \lambda_j S_{}^{j}) 
    = O ((dia(S_{}))^2). \label{bound_2_final}
    \end{aligned}
\end{equation}

\noindent \textbf{item~(3)} We introduce an auxiliary optimization problem as follows:
\begin{equation}
    \begin{aligned}
    &\text{min~}f(y ,X=x) + C_1^{\text{}}(\bm{\gamma}) - C_2^{\text{}}(\bm{\gamma}) \\
    &\text{subject to}: \bm{\phi} \in IR_{{\Phi}}, D_{i1}^{\text{}}(\bm{\gamma}) - D_{i2}^{\text{}}(\bm{\gamma}) = 1 , \\ &~~~~~~~~~~~~~~~~~~~\bm{\gamma} \in \{\bm{\gamma}^{'}: \text{} \exists \bm{\gamma}^{''}\in S, \|\bm{\gamma}^{'} - \bm{\gamma}^{''}\|\leq \frac{(\sqrt{2}+1)\sqrt{d}}{\delta} dia(S)^2 \} \cap S_0.
    \end{aligned}   \label{auxiliary-formulation}
\end{equation}

Compared with the optimization problem of $\underline{ f_S(Y_x = y)}$ (by \eqref{auxiliary-formulation} with an additional constraint $\gamma \in S$), \eqref{auxiliary-formulation} provides a relaxed constraint on $\bm{\gamma}$. We denote the optimal solution of \eqref{auxiliary-formulation} as $\uwave{\bm{\gamma}}$, and the optimal value as $\uwave{f_S(Y_x = y)}$.

On the one hand, \eqref{auxiliary-formulation} slightly relaxes the constraint $\bm{\gamma} \in S_0$. Namely for each $\uwave{\bm{\gamma}}$, there exists a corresponding $\bm{\gamma}^{''} \in S$ with a distance less than $\frac{Q}{\delta}dia(S)^2$. Hence
\begin{equation}
    \begin{aligned}
    &[C_1(\underline{\bm{\gamma}}) -  C_2(\underline{\bm{\gamma}})] - [C_1(\uwave{\bm{\gamma}}) -  C_2(\uwave{\bm{\gamma}})] \\
    \leq& [C_1({\bm{\gamma}}^{''}) -  C_2({\bm{\gamma}}^{''})] - [C_1(\uwave{\bm{\gamma}}) -  C_2(\uwave{\bm{\gamma}})] \\
    \leq& \max\limits_{\bm{\gamma} \in S_0} \|\frac{\partial (C_1(\bm{\gamma}) - C_2(\bm{\gamma}))}{\partial \bm{\gamma}}\| \left(\frac{(\sqrt{2}+1)\sqrt{d}}{\delta}dia(S)^2\right).
    \end{aligned}\label{claim_1}
\end{equation}

On the other hand, we consider the optimal solution $\underline{\underline{\bm{\gamma}}}$ of $\underline{\underline{f_S(Y_x = y)}}$. We identify the elements $\underline{\underline{\bm{\gamma}}}^{T} = ((\underline{\underline{\bm{\psi^o}}})^T, \underline{\underline{\bm{\theta}}}^T, \underline{\underline{\bm{\psi}}}^T, \underline{\underline{\bm{\omega}}^T})$. Then we introduce an auxiliary solution as follows:
\begin{equation}
    \begin{aligned}
    &{\uwave{\psi^o_i}} = \frac{1}{\underline{\underline{\psi_i}}}, {\uwave{\bm{\psi^o}}} = ({\uwave{\psi^o_1}},... {\uwave{\psi^o_{d}}}) \\
    &\uwave{\bm{\gamma}}^{'} =  (({\uwave{\bm{\psi^o}}})^T, \underline{\underline{\bm{\theta}}}^T, \underline{\underline{\bm{\psi}}}^T, \underline{\underline{\bm{\omega}}}^T).
    \end{aligned}\label{identification_new_gamma}
\end{equation}
We will show that $\uwave{\bm{\gamma}}^{'}$ is within the feasible region of \eqref{auxiliary-formulation}. By identification in \eqref{identification_new_gamma}, the first row of constraints in \eqref{auxiliary-formulation} can be directly satisfied. Moreover, by Ass.~\ref{positive definite assumption}, we have 

\begin{equation}
    \begin{aligned}
    \|\uwave{\bm{\gamma}}^{'} - \underline{\underline{\bm{\gamma}}}\| =&  \left(\sum\limits_{i=1}^{d} (\frac{1}{\underline{\underline{\psi_i}}} - \underline{\underline{\psi_i^o}} )^2\right)^{\frac{1}{2}} \\
    \leq& \frac{1}{\delta} (\sum\limits_{i=1}^{d}( \underline{\underline{\psi_i}}\underline{\underline{\psi_i^o}} - 1)^2)^{\frac{1}{2}}\\
    =& \frac{1}{\delta} (\sum\limits_{i=1}^{d}( D_{i1}(\bm{\underline{\underline{\gamma}}}) - D_{i2}(\bm{\underline{\underline{\gamma}}}) - 1)^2)^{\frac{1}{2}}\\
    \leq& \frac{\sqrt{d}}{\delta} \max\limits_{i=1,...d} \left(1- \left(D_{i1}(\bm{\underline{\underline{\gamma}}}) - D_{i2}(\bm{\underline{\underline{\gamma}}})\right) \right) \\
    \leq& \frac{\sqrt{d}}{\delta} \max\limits_{i=1,...d}\left[ \left( D_{i1}^{sec}(\bm{\underline{\underline{\gamma}}}) - D_{i2}^{tan}(\bm{\underline{\underline{\gamma}}})\right) - \left( D_{i1}^{}(\bm{\underline{\underline{\gamma}}}) - D_{i2}^{}(\bm{\underline{\underline{\gamma}}})\right)\right].
    \end{aligned}\label{gamma_difference_1}
\end{equation}

Symmetrically, we have 
\begin{equation}
    \begin{aligned}
    \|\uwave{\bm{\gamma}}^{'} - \underline{\underline{\bm{\gamma}}}\| \leq \frac{\sqrt{d}}{\delta} \max\limits_{i=1,...d}\left[ \left( D_{i1}^{}(\bm{\underline{\underline{\gamma}}}) - D_{i2}^{}(\bm{\underline{\underline{\gamma}}})\right) - \left( D_{i1}^{tan}(\bm{\underline{\underline{\gamma}}}) - D_{i2}^{sec}(\bm{\underline{\underline{\gamma}}})\right) \right].
    \end{aligned}\label{gamma_difference_2}
\end{equation}

By the same strategy in \textbf{item(1)-(2)}, and noticing the fact that 
\begin{equation}
    \begin{aligned}
    \max\limits_{\bm{\gamma} \in S_{}} \|\frac{\partial^2 D_{i1}(\bm{\gamma)}}{\partial \bm{\gamma}^2} \|_F = 2, \max\limits_{\bm{\gamma} \in S_{}} \|\frac{\partial^2 D_{i2}(\bm{\gamma)}}{\partial \bm{\gamma}^2} \|_F = \sqrt{2},
    \end{aligned}
\end{equation}
\eqref{gamma_difference_1} and \eqref{gamma_difference_2} can be combined as
\begin{equation}
    \begin{aligned}
    &\|\uwave{\bm{\gamma}}^{'} - \underline{\underline{\bm{\gamma}}}\|  \leq   \frac{\sqrt{d}}{\delta} \min \{\frac{1}{2}*2+\sqrt{2}, 2 + \frac{1}{2}\sqrt{2}\} dia(S)^2 = \frac{\sqrt{d}}{\delta} (\sqrt{2}+1)dia(S)^2.
    \end{aligned}
\end{equation}

Hence we claim this $\uwave{\bm{\gamma}}^{'}$ is within the feasible region of \eqref{auxiliary-formulation}. Then
\begin{equation}
    \begin{aligned}
     \left[C_1(\uwave{\bm{\gamma}}) - C_2(\uwave{\bm{\gamma}})\right] - \left[C_1(\underline{\underline{\bm{\gamma}}}^{}) - C_2(\underline{\underline{\bm{\gamma}}}^{})\right]  &\leq \left[ C_1(\uwave{\bm{\gamma}}^{'}) - C_2(\uwave{\bm{\gamma}}^{'}) \right] - \left[C_1(\underline{\underline{\bm{\gamma}}}^{}) - C_2(\underline{\underline{\bm{\gamma}}}^{})\right] \\
     &\leq \max\limits_{\bm{\gamma} \in S_{}} \|\frac{\partial{(C_1(\bm{\gamma}) - C_2(\bm{\gamma}))}}{\partial{\bm{\gamma}}}\|  \frac{\sqrt{d}}{\delta} (\sqrt{2}+1)dia(S)^2.
     \end{aligned}\label{claim_2}
\end{equation}

Combining \eqref{claim_1} and \eqref{claim_2}, we have
\begin{equation}
    \begin{aligned}
    \mid (C^{\text{}}_1({\underline{{\bm{\gamma}}}}) -  C^{\text{}}_2({\underline{{\bm{\gamma}}}})) - ( C^{\text{}}_1({\underline{\underline{{\bm{\gamma}}}}}) -  C^{\text{}}_2(\underline{\underline{{\bm{\gamma}}}}))  \mid \leq \max\limits_{\bm{\gamma} \in S_{0}} \|\frac{\partial{(C_1(\bm{\gamma}) - C_2(\bm{\gamma}))}}{\partial{\bm{\gamma}}}\|  \frac{2\sqrt{d}}{\delta} (\sqrt{2}+1)dia(S)^2.
    \end{aligned}
\end{equation}

\noindent \textbf{Combination of} \textbf{item(1)-(3)} Combining with Eqn.~\eqref{bound_1} and Eqn.~\eqref{bound_2_final}, and recalling the bound in (\ref{linear_bound_main}), we have:
\begin{equation}
    \underline{\underline{f_{S_{}}^{}(Y_x = y)}} \leq \underline{f_{S_{}}(Y_x = y)} \leq   \underline{\underline{f_{S_{}}^{}(Y_x = y)}} +A * dia(S_{})^2.
\end{equation}
In brief, we have $\mid  \underline{\underline{f_{S_{}}^{}(Y_x = y)}} -  {\underline{f_{S_{}}^{}(Y_x = y)}}\mid = O(dia(S)^2)$. Thus we have proved our lemma.
\end{proof}

\begin{remark}
We can do enhancement in $\textbf{Bounding()}$ as follows. It is through taking advantage of the information from the parent simplex $\text{pa}(S)$ \footnote{$S_1 = \text{pa}(S_2)$ denotes $S_2$ is bisectioned from $S_1$.} and encapsulating the above bounding strategy into a recursive form during partitioning.
\begin{equation}
    \begin{aligned}
        \underline{\underline{f_{S_{}}(Y_{x} = y)}} =  &\max \{ \textbf{Bounding}(\text{pa}(S_{})), \underline{\underline{f_{S_{}}^{}(Y_{x} =  y)}\}}
    \end{aligned}
\end{equation}
\end{remark}

\noindent{\textbf{The proof of (4)}}
For the final preparation, we introduce the bisection theorem:
\begin{theorem}{(\cite{kearfott1978proof}, Theorem~3.1)}
When $\tilde{S}_{i_{k}}$ is bisectioned from $S_{0}$ by $k$ times, we have
\begin{equation}
    \begin{aligned}
    dia(\tilde{S}_{i_{k}} ) \leq  (\frac{\sqrt{3}}{2})^{\lfloor\frac{k}{4d}\rfloor } dia(S_{0}). 
    \end{aligned}
\end{equation}
\end{theorem}

On this basis, notice that lemma.~\ref{second_bounded} holds on each iteration, and $dia(S_{0})<+\infty$, then we have 
\begin{equation}
    \begin{aligned}
    \mid {\underline{f_{\tilde{S}_{i_{L_n} }}(Y_x = y)}}  -    \underline{\underline{f_{\tilde{S}_{i_{L_n} }}(Y_x = y)}} \mid  \leq A ((\frac{{3}}{4})^{\lfloor\frac{L_n}{4d}\rfloor} )  = O((\frac{{3}}{4})^{\frac{L_n}{4d}}),
    \end{aligned}\label{second_bounded_result}
\end{equation}
where $A$ is identified in our main text.
~\quad \\

Until here we have proved procedure \textbf{(1)-(4)}, thus the main part of Theorem.~\ref{convergence_theorem} has been proved.
~\quad \\

Additionally, consider the infinite case. Due to $L_n \geq log(n)$ (the worst case is that simplices set is bisectioned like a complete binary tree), we have $L\rightarrow +\infty$ when $n\rightarrow +\infty$, thus 
\begin{equation}
    \begin{aligned}
    \lim_{n\rightarrow +\infty} \mid \underline{f(Y_x = y)}  - \underline{f_{opt}^n (f(Y_x=y))} \mid = 0.
    \end{aligned}
\end{equation} Done.

\subsection{Extension to the ACE case}\label{app_corollary}
\label{corollary_ACE_appendix}

We first illustrate the construction of (\ref{eqn_basic_bound_ace}):
\begin{equation}
\begin{aligned}
     &\int_{X^L}^{X^U} \int_{Y^L}^{Y^U} f(Y_x=y)\pi(x) dx dy \\
     =&\int_{X^L}^{X^U} \int_{Y^L}^{Y^U} \sum_{i=1}^{d} \left(\frac{f(y,u_i,X=x) f(u_i,X\neq x)}{f(u_i,X=x)}+f(y,X=x)\right)\pi(x) dx dy \\
     =& \sum_x \pi(x) \int_{Y^L}^{Y^U} yf(y,X=x) dy +  \sum_x \pi(x) \sum_{i=1}^{d} \frac{ \left(\int_{Y^L}^{Y^U}yf(y,u_i,X=x)dy\right) f(u_i, X\neq x)}{f(u_i,X=x)}.
\end{aligned}
\end{equation}

Let $X=\{x_{1},x_{2},...x_{dim(\bm{X})}\}$. In this section, we extend PI-SFP method from bounding $f(Y_x = y)$ to bounding $ACE.$ For simplicity, we extend the denotations in our main text as follows:
\begin{equation}
    \begin{aligned}
    &\theta_{i \mid x} = {\int_{Y^{L}}^{Y^U}yf(y,U=u_i,X=x)dy}, & \bm{\theta_x} = (\theta_{1\mid x}, \theta_{2\mid x}, ...\theta_{d \mid x})^T\\ 
    &\psi_{i \mid x} = {f(U=u_i,X=x)}, & \bm{\psi_x} = (\psi_{1\mid x}, \psi_{2\mid x}, ...\psi_{d \mid x})^T\\ 
    &\omega_{i \mid x} = {f(U=u_i,X\neq x)}, & \bm{\omega_x} = (\omega_{1\mid x}, \omega_{2\mid x}, ...\omega_{d \mid x})^T\\
    &\psi_{i \mid x} \psi_{i \mid x}^o = 1, & \bm{\psi_x^o} = (\psi^o_{1\mid x}, \psi^o_{2\mid x}, ...\psi^o_{d \mid x})^T\\
    &\bm{\phi_x} = (\bm{\theta_{x}}, \bm{\psi_{x}}, \bm{\omega_{x}}),
    &\bm{\gamma}_x= \left(\begin{matrix}
    (\bm{\psi_{x}^o})^T,  \bm{\theta_{x}}^T,  \bm{\psi_x}^T,  \bm{\omega_x}^T
    \end{matrix}\right)^T
    \end{aligned}
\end{equation}
On this basis, the independent variables are transformed to $\bm{\gamma} = (\bm{\gamma}_{x_1}, \bm{\gamma}_{x_2},...\bm{\gamma}_{x_{dim(\bm{X})}} )$. Following the same strategy as in Section.~\ref{framework} and Section.~\ref{section_algorithm}, we can relax the programming (\ref{eqn_basic_bound_ace}) in our main text as follows. It is a natural extension of (\ref{re-formulation}) in Section.~\ref{framework}, by which we seek the valid bound of ACE: 

\begin{equation}
    \begin{aligned}
    \underline{ACE_{\bm{X}\rightarrow \bm{Y}}} {=} &\min \sum_{x}\int_{Y^{L}}^{Y^U} \pi(x) yf(y ,X=x) dy +  \sum_{x}\sum_{i=1}^{d} \psi_{i\mid x}^{o} \theta_{i\mid x} \omega_{i\mid x}\pi(x), \\
    &\text{~subject to~} 
    \forall x \in X, \psi_{i\mid x}^{o} \psi_{i\mid x} = 1 , 
    \bm{\phi_x} \in IR_{\bm{\Phi_x}} =  IR^{1}_{\bm{\Phi_x}} \cap  IR^2_{\bm{\Phi_x}}, \\
    \end{aligned}   \label{PI-SFP_ACE}
\end{equation}

where the set $IR^{1}_{\bm{\Phi}}$ is constructed as

\begin{equation}
\begin{aligned}
   &IR^{1}_{\bm{\Phi_x}} = \{\bm{\phi_x}: \left[\begin{matrix}
   -\bm{I_{d*d}} \\ \bm{I_{d*d}}
   \end{matrix}\right] \left[\begin{matrix}
   &(\int_{Y^L}^{Y^U} yf(y,\bm{W},X=x) dy )^{T} \\ &f(\bm{W},X=x)^{T} \\ &f(\bm{W},X\neq x)^{T}
   \end{matrix} \right]^{T}  -  \left[ \begin{matrix}
   &-\overline{P(\bm{W}\mid \bm{U})} \\  &\underline{P(\bm{W} \mid \bm{U})}
   \end{matrix} \right] \bm{\phi_x} \geq \bm{0}\}.\\
   \end{aligned}
\end{equation}

$\bm{I_{d*d}}$ is the $d*d$ identity matrix. Moreover, the set $IR^{2}_{\bm{\Phi}}$ indicates the natural constraints by default:

\begin{equation}
    \begin{aligned}
    IR_{\bm{\Phi}_x}^{2} = \left\{ \bm{\phi}_x:
    \left[\begin{matrix}
    &\bm{1_{1*d}}^{} \bm{\theta_x} \\ &\bm{1_{1*d}}^{} \bm{\phi_x} \\ &\bm{1_{1*d}}^{} \bm{\omega_x} 
    \end{matrix}\right] = \left[ \begin{matrix}
    &\int^{Y^U}_{Y^L} yf(y,X=x)dy \\ &f(X=x)\\ &f(X\neq x) \}
    \end{matrix}\right], \forall i,
\left\{\begin{matrix}
    \theta_{i\mid x} \in [0,f(y,X=x)] \\
    \phi_{i\mid x} \in (0,f(X=x)] \\
    \omega_{i\mid x} \in [0,f(X\neq x)]
\end{matrix} \right\} \right\}    .
    \end{aligned}
\end{equation}
$\bm{1_{1*d}}$ is the $1*d$ all-ones vector. Then (\ref{re-formulation_linear_weaker}) in our main text is extended as
\begin{equation}
    \begin{aligned}
   &\min \sum_{x} \pi(x)\{[C_1^{\text{tan}}(\bm{\gamma}_{x}) - C_2^{\text{sec}}(\bm{\gamma}_{x})]\mathbbm{1}_{\pi(x)>0}+[C_1^{\text{sec}}(\bm{\gamma}_{x}) - C_2^{\text{tan}}(\bm{\gamma}_{x})]\mathbbm{1}_{\pi(x)<0}\},  \\
   &\text{~subject to~}D_{i}^{l}(\bm{\gamma}_{x}) \leq 1,  D_{i}^{u}(\bm{\gamma}_{x}) \geq 1, \forall i=1,2,...d, 
    \bm{\phi_x} \in IR_{\bm{\Phi_x}}.
    \end{aligned}
\end{equation}
Here the function $C_k^{\text{tan}}(\bm{\gamma}_{x}), C_k^{\text{sec}}(\bm{\gamma}_{x}), k=1,2, D_{i}^{l}(\bm{\gamma}_{x}), D_{i}^{u}(\bm{\gamma}_{x}), i=1,2,...d$ are all following (\ref{tan_sec}) in our main text. After this construction, we adopt the same simplicial partition strategy as in our main text.

\subsection{The proof of further discussions and extensions}\label{app_discussion}
\subsubsection{Discussion 1: the proof of lemma.~\ref{lemma_no_assumption}}\label{pro_lemma_no_assumption}

\begin{proof}
For simplification, the denotations $Y=y$, $X=x$ are simplified as $y$ and $x$, the denotation $X \neq x$ is simplified as $x^c$, and $dim(\bm{W})$ is simplified as $\mathscr{W}$. Samely, we use $\bm{E_{i*i}}$ to denote the $i*i$ identity matrix, $\bm{J_{i,j}}$ to denote the $i*j$ all-ones matrix, and $\bm{0_{i*j}}$ to denote the $i*j$ all-zero matrix.

\begin{itemize}
\item \textbf{Conclusion 1:} The tight lower bound of ${f(Y_x = y)}$ is $f(y, X=x)$.
\end{itemize}

We divide it into two parts. On the one hand, if $\mathscr{W} \geq d$, $P(\bm{W} \mid \bm{U})$ can be constructed as follows.
\begin{equation}
    \begin{aligned}   P(\bm{W}\mid \bm{U}) = \left[
\begin{BMAT}{c.c}{c}
    \begin{matrix} 
    \overbrace{\bm{P_{11}}}^{m*m} \\
    {\overbrace{\bm{P_{21}}}^{(\mathscr{W}-m)*m}}
    \end{matrix} 
    & 
    \begin{matrix} 
    \overbrace{\bm{P_{12}}}^{(d-m)*(d-m)}
    \\   
    \overbrace{\bm{P_{22}}}^{(\mathscr{W}-d+m)*(d-m)}
    \end{matrix} \\
\end{BMAT}
\right]
,
    \end{aligned} \label{trivial_construction}
\end{equation}  
where $\bm{P_{11}}, \bm{P_{12}}, \bm{P_{21}}, \bm{P_{22}}$ are matrices whose upper brackets indicate their rows and columns~($m\in [1,d-1]$). Specifically,
\begin{equation}
    \begin{aligned}
&\bm{P_{11}} = {\sum\limits_{i=1}^{m}f(W=w_i \mid y,x) }  \bm{E_{m*m}},~\bm{P_{12}} =  {\sum\limits_{i=1}^{d-m} f(W=w_i\mid x^c) } \bm{E_{(d-m)*(d-m)}}, \\
&\bm{P_{21}} = { \left[ \begin{matrix}  {f(W=w_{m+1}\mid y,x)}\\
    ... \\
    {f(W=w_{\mathscr{{W}}}\mid y,x)}\\  
    \end{matrix}\right]}  \bm{J_{1*m}},~\bm{P_{22}} = {\left[ \begin{matrix}  {f(W=w_{d-m+1}\mid x^c)}  \\
    ...  \\ 
    {f(W=w_{\mathscr{W}}\mid x^c)} \end{matrix}\right]}   \bm{J_{1*(d-m)}}.
\end{aligned} 
\end{equation} 

There is a solution for $f(y,\bm{U},x)$, $f(\bm{U},x^c)$ respectively as
\begin{equation}
\begin{aligned}
     \frac{1}{\sum\limits_{i=1}^{m}f(W=w_i\mid y,x)}\left[
     \begin{matrix}
     {f(y,W=w_1,x) }\\
     ...\\
     {f(y,W=w_m,x)}\\
       \bm{0_{(d-m)*1}}
     \end{matrix} \right],~~ \frac{1}{\sum\limits_{i=1}^{{d-m}}P(W=w_i\mid x^c)} \left[
     \begin{matrix}
          \bm{0_{m*1}}\\
     {P(W=w_1,x^c)}{}\\
     ...\\
     {P(W=w_{d-m},x^c)}\\
     \end{matrix}\right].
    \end{aligned}
\end{equation}

Due to $f(y,\bm{U},x) \circ f(\bm{U},x^c) = 0$ and the condition $f(\bm{U},x)>\bm{0}$, we have
\begin{equation}
    \begin{aligned}
     f(Y_x = y) =  f(y,x)+\sum_{i=1}^d\frac{f(y,u_i,x)}{P(u_i,x)}P(u_i, x^c) = f(y,x).
    \end{aligned}
\end{equation}

On the other hand, if $ \mathscr{W} < d $, we make adjustments on \eqref{trivial_construction} ($m_1+m_2 \leq \mathscr{W}$):
\begin{equation}
    \begin{aligned} \left[
\begin{BMAT}{c.c.c}{c}
    \begin{matrix} 
    \overbrace{\bm{P_{11}}}^{m_1*m_1} \\
    \overbrace{\bm{P_{21}}}^{(\mathscr{W}-m_1)*m_1}
    \end{matrix} 
    & 
    \begin{matrix} 
    \overbrace{\bm{P_{12}}}^{m_2 * m_2}
    \\   
    \overbrace{\bm{P_{22}}}^{(\mathscr{W}-m_2) * m_2}
    \end{matrix} 
    &
    \overbrace{\bm{P_3}}^{\mathscr{W} * (d-m_1-m_2)}
\end{BMAT}
\right] ,
    \end{aligned} \label{trivial_construction_extend}
\end{equation}
Specifically,
\begin{equation}
    \begin{aligned}
&\bm{P_{11}} = {\sum\limits_{i=1}^{m_1} P(W=w_i\mid y,x) }  \bm{E_{m_1 * m_1}},~\bm{P_{12}} = {\sum\limits_{i=1}^{m_2} P(W=w_i\mid x^c) }{}\bm{E_{m_2*m_2}}. \\
&\bm{P_{21}} = {\left[ \begin{matrix}  {f(W=w_{m_1+1}\mid y,x)}\\
    ... \\
    {f(W=w_{\mathscr{W}}\mid y,x)}\\  
    \end{matrix}\right]}^{} \bm{J_{1,m_1}},~\bm{P_{22}} = {\left[ \begin{matrix}  {f(W=w_{m_2+1} \mid x^c)} \\
    ...  \\ 
    {f(W=w_{\mathscr{W}}\mid x^c)}  \end{matrix}\right]} \bm{J_{1*m_2}}.\\
    &\bm{P_3} = \frac{1}{\mathscr{W}} {\bm{J_{\mathscr{W} * (d-m_1-m_2)}}}.
\end{aligned} 
\end{equation} 

Analogously, there is a solution for $f(y,\bm{U},x)$, $f(\bm{U},x^c)$ as follows respectively:

\begin{equation}
\begin{aligned}
     \frac{1}{\sum\limits_{i=1}^{m_1}f(W=w_i\mid y,x)}\left[
     \begin{matrix}
     {f(y,W=w_1,x)}\\
     ...\\
     {f(y,W=w_{m_1},x)}\\
       \bm{0_{(d-m_1)*1}}
     \end{matrix} \right],~~  \frac{1}{\sum\limits_{i=1}^{m_2}P(W=w_i\mid x^c)}\left[
     \begin{matrix}
         \bm{0_{m_1*1}} \\
     {P(W=w_1, x^c)}\\
     ...\\
     {P(W=w_{m_2}, x^c)}\\
    \bm{0_{(d-m_1-m_2)*1}}
     \end{matrix}\right].
    \end{aligned}
\end{equation}

In this case, we also have $f(Y_x = y) =  f(y,x)+\sum\limits_{i=1}^d\frac{f(y,u_i,x)}{f(u_i,x)}f(u_i,x^c) = f(y,x)$. In conclusion, if no assumptions are imposed, we have $\min {f(Y_x = y)} = f(y, X=x)$. Proved.\\

\begin{itemize}
\item \textbf{Conclusion 2:} If $P(\bm{W} \mid \bm{U})$ is restricted to be left-reversible and $ f(\bm{W} \mid X\neq x) \neq  f(\bm{W} \mid X=x,y)$, then the tight lower bound of ${f(Y_x = y)}$ is $f(y, X=x)$.
\end{itemize}

Without loss of generalization, we can assume that $\exists i_0 \in \{d,d+1,...\mathscr{W}\}$, such that $f(W=w_{i_0} \mid X\neq x) \neq f(W = w_{i_0} \mid x,y)$, or else we just need to relabel $\bm{W}$ in another order. 

On this basis, we still follow the Construction.~\ref{trivial_construction} in the first part. The tight lower bound has already been proved as $f(y,x)$, thus we only need demonstrate that with some choice of $m$, $P(\bm{W} \mid \bm{U})$ is left-reversible with the above assumption. In practice, we choose $m=d-1$. Then the $P(\bm{W} \mid \bm{U})$ is reformulated as
\begin{equation}
\begin{aligned}
    \left[
\begin{BMAT}{c.c}{c}
    \begin{matrix} 
    \overbrace{\bm{P_{11}}}^{(d-1)*(d-1)} \\
    {\overbrace{\bm{P_{21}}}^{(\mathscr{W}-d+1)*(d-1)}}
    \end{matrix} 
    & 
    \begin{matrix} 
    \overbrace{\bm{P_{12}}}^{1*1}
    \\   
    \overbrace{\bm{P_{22}}}^{(\mathscr{W}-1)*1}
    \end{matrix} \\
\end{BMAT}
\right]
:=
\left[
\begin{BMAT}{c.c}{c}
    \begin{matrix}
     \sum\limits_{i=1}^{d-1}P(W=w_i\mid y,x) \bm{E_{(d-1)*(d-1)}} \\
    \left[\begin{matrix}
     {f(W=w_d\mid y,x)} \\{f(W=w_{d+1}\mid y,x)} 
     \\...
     \\ {f(W=w_{\mathscr{W}}\mid y,x)}  
    \end{matrix} \right] \bm{J_{1*(d-1)}}
\end{matrix}
    & \left[
    \begin{matrix} 
   {f(W=w_1\mid x^c)} \\
   {f(W=w_2\mid x^c)} \\
   ...\\
   {f(W=w_{\mathscr{W}}\mid x^c)}
    \end{matrix} \right] \\
\end{BMAT}
\right].
\end{aligned}\label{left_reversible_wu_construct}
\end{equation}

We make equivalent denotations:
\begin{equation}
    \begin{aligned}
\left[
\begin{BMAT}{c}{c}
    \begin{matrix} 
    \overbrace{\bm{P^{'}_{12}}}^{(d-1)*1}
    \\   
    \overbrace{\bm{P^{'}_{22}}}^{(\mathscr{W}-d+1)*1}
    \end{matrix} \\
\end{BMAT}
\right] := 
 \left[
\begin{BMAT}{c}{c}
    \begin{matrix} 
    \overbrace{\bm{P_{12}}}^{1*1}
    \\   
    \overbrace{\bm{P_{22}}}^{(\mathscr{W}-1)*1}
    \end{matrix} \\
\end{BMAT}
\right]
    \end{aligned}
\end{equation}

In the following part, we claim that we only need to prove $\bm{P^{'}_{22} - P_{21} P_{11}^{-1} P^{'}_{12}} \neq \bm{0}$. We do the following algebraic distortion:
\begin{equation}
    \begin{aligned}
    \left[ \begin{matrix}
    \bm{E_{(d-1)*(d-1)}}& \bm{0_{(d-1)*(\mathscr{W}-d+1)}} \\ \bm{-P_{21} P_{11}^{-1}}& \bm{E_{(\mathscr{W}-d+1)*(\mathscr{W}-d+1)}}
    \end{matrix} 
     \right] *  \left[ \begin{matrix}
   \bm{P_{11}} & \bm{P^{'}_{12}} \\ \bm{P_{21}} & \bm{P^{'}_{22}}
    \end{matrix} 
     \right]  =  \left[ \begin{matrix}
   \bm{P_{11}} & \bm{P^{'}_{12}} \\ \bm{0_{(\mathscr{W}-d+1)*(d-1)}} & \bm{P^{'}_{22} - P_{21} P_{11}^{-1} P^{'}_{12}}
    \end{matrix} 
     \right].
    \end{aligned}
\end{equation}

According to the well-known Sylvester’s inequality~\cite{matsaglia1974equalities}: $\forall$$\bm{A_{m*n}},\bm{B_{n*p}}$, we have $\min \{rank(\bm{A}), rank(\bm{B})\} \geq rank(\bm{AB}) \geq rank(\bm{A})+rank(\bm{B})-n$. Then we have
\begin{equation}
    \begin{aligned}
        rank \left(\left[ \begin{matrix}
   \bm{P_{11}} & \bm{P^{'}_{12}} \\ \bm{P_{21}} & \bm{P^{'}_{22}}
    \end{matrix} 
     \right]\right) = rank\left(\left[ \begin{matrix}
   \bm{P_{11}} & \bm{P^{'}_{12}} \\ \bm{0_{(\mathscr{W}-d+1)*(d-1)}} & \bm{P^{'}_{22} - P_{21} P_{11}^{-1} P^{'}_{12}}
    \end{matrix} 
     \right]\right).
    \end{aligned}
\end{equation}

If $\bm{P^{'}_{22} - P_{21} P_{11}^{-1} P^{'}_{12}} = \bm{0_{(\mathscr{W}-d+1)*(d-1)}}$, then the right side of $rank()$ will be equal to $rank(\left[\bm{P_{11}},\bm{P_{12}^{'}}\right]) = d-1 <d$. On the other hand, if $\bm{P^{'}_{22} - P_{21} P_{11}^{-1} P^{'}_{12}} \neq \bm{0_{(\mathscr{W}-d+1)*(d-1)}}$, then it will turn to be $d$ (full column rank). In conclusion, to demonstrate the left-reversibility of $P(\bm{W}\mid\bm{U})$, $\bm{P^{'}_{22} - P_{21} P_{11}^{-1} P^{'}_{12}} \neq \bm{0}$ is all we need.

If we use $[\cdot]_{(i)}$ to denote the $i$-th element of vector $i=d,...\mathscr{W}$, then
\begin{equation}
    {[\bm{P^{'}_{22} - P_{21} P_{11}^{-1} P^{'}_{12}}]}_{(i)} = {\sum\limits_{i=1}^{d-1}f(W=w_i\mid x^c) }  \left[\frac{f(W=w_i,x^c)}{\sum\limits_{i=1}^{d-1}f(W=w_i,x^c)} - \frac{f(y,W=w_i,x)}{\sum\limits_{i=1}^{d-1}f(y,W=w_i,x)}\right]. \label{each_element}
\end{equation}
We make the contradiction. If we have ${\bm{P^{'}_{22} - P_{21} P_{11}^{-1} P^{'}_{12}}} = \bm{0_{(\mathscr{W}-d+1)*(d-1)}}$, then 
\begin{equation}
\begin{aligned}
\|{\bm{P^{'}_{22} - P_{21} P_{11}^{-1} P^{'}_{12}}}\|_1 &= {\sum\limits_{i=1}^{d-1}f(W=w_i\mid x^c) } \left[ \frac{\sum\limits_{i=d}^{\mathscr{W}}f(W=w_i,x^c)}{\sum\limits_{i=1}^{d-1}f(W=w_i,x^c)} - \frac{\sum\limits_{i=d}^{\mathscr{W}}f(y,W=w_i,x)}{\sum\limits_{i=1}^{d-1}f(y,W=w_i,x)}\right] \\
&= {\sum\limits_{i=1}^{d-1}f(W=w_i\mid x^c) } \left[ \frac{f(x^c)}{\sum\limits_{i=1}^{d-1}f(W=w_i,x^c)} - \frac{f(y,x)}{\sum\limits_{i=1}^{d-1}f(y,W=w_i,x)}\right]\\
&={\sum\limits_{i=1}^{d-1}f(W=w_i\mid x^c) }\left[\frac{1}{\sum\limits_{i=1}^{d-1}f(W=w_i \mid x^c)} - \frac{1}{\sum\limits_{i=1}^{d-1}f(W=w_i \mid y,x)}\right] = 0.
\end{aligned}
\end{equation}
Thus we have $\sum\limits_{i=1}^{d-1} f(W=w_i \mid x^c) = \sum\limits_{i=1}^{d-1} f(W=w_i \mid y,x)$. Then we substitute it into Eqn.~\eqref{each_element}, we have
\begin{equation}
    \begin{aligned}
        f(W=w_i \mid x^c) - f(W=w_i \mid x,y) = 0, \forall i\in\{d,...\mathscr{W}\}.
    \end{aligned}
\end{equation}
Contradiction! Hence we have $ {\bm{P_{22} - P_{21} P_{11}^{-1} P_{12}}} \neq \bm{0_{(\mathscr{W}-d+1)*(d-1)}}$, and then $P(\bm{W} \mid \bm{U})$ in Construction.~\ref{left_reversible_wu_construct} is left-reversible. Proved.

\begin{itemize}
\item \textbf{Conclusion 3:} If $P(\bm{W} \mid \bm{U})$ is restricted to be left-reversible and $ f(\bm{W} \mid X\neq x) =  f(\bm{W} \mid X=x,y)$, then the tight lower bound of ${f(Y_x = y)}$ is $f(y\mid X=x)$.
\end{itemize}
If this assumption holds, we will have $\bm{P^{'}_{22} - P_{21} P_{11}^{-1} P^{'}_{12}} = \bm{0}$ in the above construction, thus $P(\bm{W}\mid \bm{U})$ will be irreversible and validates the condintion here. Hence we need another way. 

According to the left-reversibility of $P(\bm{W} \mid \bm{U})$, we have
\begin{equation}
    \begin{aligned}
      f(\bm{U}\mid x,y) = P(\bm{W} \mid \bm{U})^{-1} f(\bm{W}\mid  x,y) = P(\bm{W} \mid \bm{U})^{-1} f(\bm{W}\mid x^c) = f(\bm{U} \mid x^c)
    \end{aligned}\label{totally_equal}
\end{equation}
Then we have
\begin{equation}
    \begin{aligned}
        f(Y_x = y) &= f(x,y) + \sum_{i=1}^{d}\frac{f(x,y,u_i)}{f(x,u_i)}f(u_i,x^c) \\
        &= f(x,y) + f(x,y)f(x^c) \sum_{i=1}^d \frac{f(u_i \mid x,y)}{f(x, u_i)} f(u_i \mid x^c)\\
        &= f(x,y) + f(x,y)f(x^c) \sum_{i=1}^d \frac{f(u_i \mid x,y)^2}{f(x, u_i)}\\
        & \overset{*}{\geq} f(x,y) + f(x,y)f(x^c) \frac{(\sum_{i=1}^d f(u_i\mid x,y))^2}{\sum_{i=1}^d f(x,u_i)}\\
        & = f(x,y)\left(1+\frac{f(x^c)}{f(x)}\right)\\
        & = f(y \mid x). 
    \end{aligned}
\end{equation}
According to the Chauchy's inequality, the $'\geq'$ ($*$) turns to be $'='$ if and only if $f(\bm{U} \mid x,y) = f(\bm{U} \mid x)$. Combining with Eqn.~\eqref{totally_equal}, we have $f(\bm{U} \mid x,y) = f(\bm{U} \mid x) = f(\bm{U} \mid x^c) = f(\bm{U})$. It holds if and only if $f(\bm{W} \mid x,y) = f(\bm{W} \mid x) = f(\bm{W} \mid x^c) = f(\bm{W})$, or else the lower bound is not tight.
\end{proof}

\subsubsection{Discussion 2: an acceleration trick of PI-SFP} \label{pre_train}

In this section, we provide a fast-convergent local optimization method to produce a good upper bound $\underline{f(Y_x = y)}$. It can help accelerate our algorithm. That is, if we find our optimal result in the reduced space is already larger then the local optimal value here, then it will be larger than $\underline{f(Y_x = y)}$. On this basis, we can conclude that this partition must not include the optimal solution, and we can delete this partition forever. The principle of our algorithm is based on the lemma:
\begin{lemma}
$\forall i,j$, if we make adjustment:
\begin{equation}{
    \begin{aligned}\left[
     \begin{matrix}
     &\breve{\theta}_i & \breve{\theta}_j \\
     &\breve{\psi}_i & \breve{\psi}_j \\
     &\breve{\omega}_i & \breve{\omega}_j
     \end{matrix}
     \right]
     =  \left[
     \begin{matrix}
     &\theta_i & \theta_j \\
     &\psi_i & \psi_j \\
     &\omega_i & \omega_j
     \end{matrix}
     \right]
     \left[
     \begin{matrix}
     &\alpha &1-\alpha \\
     &1-\alpha & \alpha 
     \end{matrix} \right],
     \text{~where~} \alpha \in \begin{cases}
 (0,1] \text{~if~} (\frac{\theta_{i}}{\theta_j}-\frac{\psi_{i}}{\psi_{j}})(\frac{\omega_{i}}{\omega_{j}}-\frac{\psi_{i}}{\psi_{j}}) \geq 0. \\ [1,+\infty) \text{~if~} (\frac{\theta_{i}}{\theta_j}-\frac{\psi_{i}}{\psi_{j}})(\frac{\omega_{i}}{\omega_{j}}-\frac{\psi_{i}}{\psi_{j}}) \leq 0,
 \end{cases}
    \end{aligned}}
\end{equation}
Then we have
\begin{equation}
    \begin{aligned}
     \sum_{m=i,j}\frac{\breve{\theta}_m}{\breve{\psi}_m}\breve{\omega}_m  \leq
     \sum_{m=i,j}\frac{{\theta}_m}{{\psi}_m}{\omega}_m.
    \end{aligned}
\end{equation}\label{adjustment}

\end{lemma}

\begin{proof}
We consider the case $\alpha \in (0,1)$, and the second case is symmetric. Due to $(\frac{\theta_i}{\theta_{j}}-\frac{\psi_{i}}{\psi_{j}}) (\frac{\omega_{i}}{\omega_{j}}-\frac{\psi_{i}}{\psi_{j}}) \geq 0$, we have
\begin{equation}
    \begin{aligned}
        (\theta_i - \theta_j)(\omega_j - \omega_i)\psi_i \psi_j + (\theta_j \omega_j \psi_i - \theta_i \omega_i \psi_j)(\psi_j - \psi_i) \leq 0.
    \end{aligned}\label{original_condition}
\end{equation}
If we denote that 
\begin{equation}
    \begin{aligned}
        Q_{ij} := \alpha \psi_i \psi_j (\theta_i - \theta_j) (\omega_j - \omega_i)
        \left[ (1-\alpha)\psi_i + \alpha \psi_j \right] + (\theta_j \omega_j \psi_i - \theta_i \omega_i \psi_j)  \left[ (1-\alpha)\psi_i + \alpha \psi_j \right] \psi_j.
    \end{aligned}
\end{equation}
Then Formulation.~\ref{original_condition} is equal to 
\begin{equation}
    \begin{aligned}
        Q_{ij}  \leq - Q_{ji}. 
    \end{aligned}\label{leq_q}
\end{equation}
Furthermore, we find 
\begin{equation}
    \begin{aligned}
        Q_{ij} &= \psi_j \breve{\psi}_j \left[ \alpha \psi_i (\theta_i - \theta_j)(\omega_j - \omega_i) + \theta_j \omega_j \psi_i - \theta_i \omega_i \psi_j\right]\\
        &= \frac{1}{1-\alpha} \psi_j \breve{\psi}_j \left[\breve{\theta}_i \breve{\omega}_i \psi_i - \theta_i \omega_i \breve{\psi}_i \right].
    \end{aligned}
\end{equation}
Hence Formulation.~\ref{leq_q} can be transformed as 
\begin{equation}
    \begin{aligned}
        \psi_j \breve{\psi}_j \left[\breve{\theta}_i \breve{\omega}_i \psi_i - \theta_i \omega_i \breve{\psi}_i\right] < - \psi_i \breve{\psi}_i \left[\breve{\theta}_j \breve{\omega}_j \psi_j - \theta_j \omega_j \breve{\psi}_j\right]
    \end{aligned}
\end{equation}
Hence 
\begin{equation}
    \begin{aligned}
        \sum_{m=i,j}\frac{\breve{\theta}_m}{\breve{\psi}_m}\breve{\omega}_m  \leq
     \sum_{m=i,j}\frac{{\theta}_m}{{\psi}_m}{\omega}_m.
    \end{aligned}
\end{equation}

Thus we have proved.
\end{proof}

In this strategy, we should choose suitable $\alpha$ to satisfy $\bm{\phi} \in IR_{\bm{\Phi}}$, namely that $f(y,\bm{W},\bm{U},\bm{X})\in \mathcal{\widetilde{F}}$.

\subsubsection{Discussion 3: Fig.~\ref{Fig.sub.2} and \ref{Fig.sub.3}}
$\mathcal{\widetilde{F}}_{Z}$ is identified as follows:

We denote
\begin{equation}
    \begin{aligned}
    \begin{matrix}
    ~~~~~~~~~~~~~~~~\theta_i = f(y,u_i,X=x,z\in \mathcal{Z})\\
    ~~~~~~~~~~~~\psi_i = f(u_i,X=x,z\in \mathcal{Z})\\
    \omega_i = f(u_i,X\neq x)\\
    \end{matrix},~
    \begin{matrix}
    \bm{\theta_{\mathcal{Z}}} &= (\theta_1, \theta_2,...\theta_d)^T\\
    \bm{\psi_{\mathcal{Z}}} &= (\psi_1, \psi_2,...\psi_d)^T\\
    \bm{\omega_{\mathcal{Z}}} &= (\omega_1, \omega_2,...\omega_d)^T\\
    \end{matrix},
    ~\bm{\phi_{\mathcal{Z}}} = \left(\begin{matrix}
    \bm{\theta_{\mathcal{Z}}}~  \bm{\psi_{\mathcal{Z}}}~ \bm{\omega_{\mathcal{Z}}}
    \end{matrix} \right).
    \end{aligned}
\end{equation}

where $f(y,\bm{W}, \bm{U}, \bm{X},\mathcal{Z}) \in \mathcal{\widetilde{F}}_{Z} = \{\bm{\phi}_{\mathcal{Z}} \in IR_{{Z}}, IR_{Z} = IR^1_{Z} \cap IR^2_{Z}\}$ leads to the following constraints that we really use:

\begin{equation}
\begin{aligned}
   &IR^{1}_{{Z}} = \{\bm{\phi}_{Z}: \left[\begin{matrix}
   -\bm{I_{d*d}} \\ \bm{I_{d*d}}
   \end{matrix}\right] \left[\begin{matrix}
   &f(y,\bm{W},X=x,z \in \mathcal{Z})^{T} \\ &f(\bm{W},X=x,z \in \mathcal{Z})^{T} \\ &f(\bm{W},X\neq x,z \in \mathcal{Z})^{T}
   \end{matrix} \right]^{T}  -  \left[ \begin{matrix}
   &-\overline{P(\bm{W}\mid \bm{U})} \\  &\underline{P(\bm{W} \mid \bm{U})}
   \end{matrix} \right] \bm{\phi}_{\mathcal{Z}} \geq \bm{0}\}.\\
   \end{aligned}
\end{equation}

Moreover, the set $IR^{2}_{\bm{\Phi}}$ indicates the natural constraints by default:

\begin{equation}
    \begin{aligned}
    IR_{Z}^{2} = \left\{ \bm{\phi}_Z:
    \left[\begin{matrix}
    &\bm{1}^{T} \bm{\theta} \\ &\bm{1}^{T} \bm{\phi} \\ &\bm{1}^{T} \bm{\omega} 
    \end{matrix}\right] = \left[ \begin{matrix}
    &f(y,X=x,z \in \mathcal{Z})\\ &f(X=x,z \in \mathcal{Z})\\ &f(X\neq x,z \in \mathcal{Z}) \}
    \end{matrix}\right], \forall i,
\left\{\begin{matrix}
    \theta_i \in [0,f(y,X=x,z \in \mathcal{Z})] \\
    \phi_i \in (0,f(X=x,z \in \mathcal{Z})] \\
    \omega_i \in [0,f(X\neq x,z \in \mathcal{Z})]
\end{matrix} \right\} \right\}.
    \end{aligned}
\end{equation}

\subsubsection{Discussion 4: the proof of Corollary.~\ref{theorem_continuous}} \label{proof_continuous}

We do partition on the confounding interval $[U^{L}, U^{U}]$ as $[u_{0}, u_{1}, u_2, ...,u_{d-1}, u_{d}]$, where $u_0 = U^{L}, u_d = U^U$. The independent variables is re-defined by
    \begin{equation}
    \begin{aligned}
    \begin{matrix}
     &~~~~\theta_i = {f(y,U\in [u_i, u_{i+1}],X=x)},\\
    &\psi_i = {f(U\in [u_i,u_{i+1}],X=x)}^{},\\  
    &\omega_i = {f(U\in [u_i, u_{i+1}],X\neq x)}^{},
    \end{matrix}
    ~~~~i=0,1,...d-1.
    \end{aligned}
\end{equation}

\begin{lemma}
Suppose that Ass.~\ref{positive definite assumption}-\ref{ass_lipschitz} hold. $\forall i \in \{0,1,...d-1\}, \forall u \in [u_{i}, u_{i+1}]$, we have
\begin{equation}
    \frac{\int_{u_i}^{u_{i+1}}f(y,u,X=x)du}{\int_{u_i}^{u_{i+1}}f(u,X=x)du} \leq \frac{f(y,u^{},X=x)}{f(u^{},X=x)}\frac{1}{1-\frac{1}{2}C_2 \eta} + \frac{\frac{1}{2}C_1 C_2 \eta}{1-\frac{1}{2}C_2 \eta}. 
\end{equation}
On the other hand,
\begin{equation}
    \frac{\int_{u_i}^{u_{i+1}}f(y,u,X=x)du}{\int_{u_i}^{u_{i+1}}f(u,X=x)du} \geq \frac{f(y,u,X=x)}{f(u,X=x)}\frac{1}{1+\frac{1}{2}C_2 \eta} - \frac{\frac{1}{2}C_1 C_2 \eta}{1+\frac{1}{2}C_2 \eta}. 
\end{equation}

\end{lemma}

\begin{proof} $\forall u^{'} \in [u_i, u_{i+1}]$, we have

\begin{equation}
    \begin{aligned}
    \frac{\int_{u_i}^{u_{i+1}}f(y,u,X=x)du}{\int_{u_i}^{u_{i+1}}f(u,X=x)du} &\leq \frac{\int_{u_i}^{u_{i+1}}\left[f(y,u^{'},X=x) + C_1 \left|f(u, X=x) - f(u^{'}, X=x)\right| \right] du}{\int_{u_i}^{u_{i+1}}\left[f(u^{'},X=x)+\left(f(u,X=x)-  f(u^{'},X=x) \right) \right] du} \\
    &\leq \frac{f(y,u^{'}, X=x)(u_{i+1} - u_i) + C_1 C_2 \frac{1}{2}(u_{i+1}-u_i)^2}{f(u^{'},X=x)(u_{i+1} - u_i) - C_2 \frac{1}{2}(u_{i+1} - u_i)^2}\\
    &\leq \frac{\frac{f(y, u^{'},X=x)}{f(u^{'},X=x)}+\frac{\frac{1}{2} C_1 C_2 \eta \delta}{f(u^{'}, X=x)}}{1-\frac{\frac{1}{2} C_2 \eta \delta}{f(u^{'}, X=x)}} \\
    &\leq \frac{f(y,u^{'},X=x)}{f(u^{'},X=x)}\frac{1}{1-\frac{1}{2}C_2 \eta} + \frac{\frac{1}{2}C_1 C_2 \eta}{1-\frac{1}{2}C_2 \eta}. 
    \end{aligned}
\end{equation}

Analogously, we can prove the other direction. Thus we have proved the lemma.
\end{proof}

Then we prove our main theorem. 
\begin{proof}
If we use $\underline{f(y,\bm{U},X=x)}, \underline{f(\bm{U},X=x)},\underline{f(\bm{U},X\neq x)}$ to denote the optimal solution of the optimal value $\underline{f(Y_x = y)}$ in the continuous case, then we have
\begin{equation}
    \begin{aligned}
    &\underline{f(Y_x = y)} -  f(y,X=x) \\
    =& \int_{U^{L}}^{U^U} \frac{\underline{f(y,u,X=x)}}{\underline{f(u,X=x)}}\underline{f(u,X\neq x)}du \\
    =& \sum_{i=0}^{d-1}\int_{u_i}^{u_{i+1}} \frac{\underline{f(y,u,X=x)}}{\underline{f(u,X=x)}}\underline{f(u,X\neq x)}du \\
    \geq &  \sum_{i=0}^{d-1}\int_{u_i}^{u_{i+1}} \left[\frac{\int_{u_i}^{u_{i+1}}\underline{f(y,u,X=x)}du}{\int_{u_i}^{u_{i+1}}\underline{f(u,X=x)}du} - \frac{\frac{1}{2}C_1 C_2 \eta}{1-\frac{1}{2} C_2 \eta}\right] (1-\frac{1}{2} C_2 \eta) \underline{f(u,X\neq x)}du \\
    =& (1-\frac{1}{2} C_2 \eta) \sum_{i=0}^{d-1} \frac{\int_{u_i}^{u_{i+1}}\underline{f(y,u,X=x)}du}{\int_{u_i}^{u_{i+1}}\underline{f(u,X=x)}du} \int_{u_i}^{u_{i+1}}\underline{f(u,X\neq x)}du - \frac{1}{2}C_1 C_2 \eta f(X\neq x).
    \end{aligned}\label{partial_bound}
\end{equation}
Here $\{\int_{u_i}^{u_{i+1}}\underline{f(y,u,X=x)}du, \int_{u_i}^{u_{i+1}}\underline{f(u,X=x)}du, \int_{u_i}^{u_{i+1}}\underline{f(u,X\neq x)}du, i=0,1,...d-1\}$ is within the feasible region of PI-SFP in the discrete case. Then we have
\begin{equation}
    \begin{aligned}
    \eqref{partial_bound} &\geq (1-\frac{1}{2} C_2 \eta) \left( \lim\limits_{n \rightarrow +\infty}\underline{\underline{f_{opt}^n (Y_x = y)}} - f(y, X=x)\right) - \frac{1}{2}C_1 C_2 \eta f(X\neq x)\\
    \underline{f(Y_x = y)} &\geq (1-\frac{1}{2} C_2 \eta)\lim\limits_{n \rightarrow +\infty} \underline{\underline{f_{opt}^n (Y_x = y)}} + \frac{1}{2}C_2 \eta f(y, X=x) - \frac{1}{2}C_1 C_2 \eta f(X\neq x)\\
    \lim\limits_{n \rightarrow +\infty} \underline{\underline{f_{opt}^n (Y_x = y)}} &\leq \frac{1}{1-\frac{1}{2}C_2 \eta}\underline{f(Y_x = y)} + \frac{\frac{1}{2}C_1f(X\neq x) - \frac{1}{2} f(y, X=x)}{1-\frac{1}{2}C_2 \eta} C_2 \eta 
    \end{aligned}
\end{equation}

On the other hand, each optimal solution by PI-SFP corresponds to a solution in the continuous case. Namely if the discrete PI-SFP's optimal solution is denoted as $\{\int_{u_i}^{u_{i+1}}{f(y,u,X=x)}du, \int_{u_i}^{u_{i+1}}{f(u,X=x)}du, \int_{u_i}^{u_{i+1}}{f(y,u,X\neq x)}du, i=0,1,...d-1\}$. Then we can construct  
\begin{equation}
    \begin{aligned}
    f^{opt}(y,u,X=x) &= \frac{\int_{u_i}^{u_{i+1}}{f(y,u,X=x)}du}{u_{i+1} - u_i}, u\in [u_i, u_{i+1}).\\
    f^{opt}(u,X=x) &= \frac{\int_{u_i}^{u_{i+1}}{f(u,X=x)}du}{u_{i+1} - u_i}, u\in [u_i, u_{i+1}).\\
    f^{opt}(u,X\neq x) &= \frac{\int_{u_i}^{u_{i+1}}{f(u,X\neq x)}du}{u_{i+1} - u_i}, u\in [u_i, u_{i+1})
    \end{aligned}
\end{equation}
as one of the solution in the continuous case. Hence $\lim\limits_{n \rightarrow +\infty}\underline{\underline{f_{opt}^n (Y_x = y)}} \geq \underline{f(Y_x = y)}$. Hence we have proved.

\end{proof}

\end{document}